\documentclass [twoside,reqno, 11pt] {amsart}

\usepackage{amsfonts}
\usepackage{amssymb}
\usepackage{a4}
\usepackage{color}

\newtheorem{thm}{Theorem}[section]

\newtheorem{lem}[thm]{Lemma}
\newtheorem{prop}[thm]{Proposition}

\newtheorem{defn}[thm]{Definition}

\theoremstyle{definition}

\numberwithin{equation}{section}

\renewcommand{\Re}{\hbox{Re}\,}
\renewcommand{\Im}{\hbox{Im}\,}

\newcommand{\C}{\mathbb{C}}

\renewcommand{\div}{\operatorname{div}}

\newcommand{\R}{\mathbb{R}}

\newcommand{\supp}{\operatorname{supp}}

\newcommand{\tr}{\operatorname{tr}}

\newcommand{\Z}{\mathbb{Z}}

\def\hat{\widehat}
\def\tilde{\widetilde}
\def \bfo {\begin {eqnarray*} }
\def \efo {\end {eqnarray*} }
\def \ba {\begin {eqnarray*} }
\def \ea {\end {eqnarray*} }
\def \beq {\begin {eqnarray}}
\def \eeq {\end {eqnarray}}
\def \supp {\hbox{supp }}

\def \dist {\hbox{dist}}

\def \det {\hbox{det}}

\def \p {\partial}

\def\hat{\widehat}
\def\tilde{\widetilde}
\def \bfo {\begin {eqnarray*} }
\def \efo {\end {eqnarray*} }
\def \ba {\begin {eqnarray*} }
\def \ea {\end {eqnarray*} }
\def \beq {\begin {eqnarray}}
\def \eeq {\end {eqnarray}}
\def \supp {\hbox{supp }}

\def \dist {\hbox{dist}}

\def \det {\hbox{det}}

\def \p {\partial}

\begin{document}

 \title[Magnetic Schr\"odinger operator ]{Inverse problems for magnetic Schr\"odinger operators in transversally anisotropic geometries}

\author[Krupchyk]{Katya Krupchyk}

\address
        {K. Krupchyk, Department of Mathematics\\
University of California, Irvine\\ 
CA 92697-3875, USA }

\email{katya.krupchyk@uci.edu}

\author[Uhlmann]{Gunther Uhlmann}

\address
       {G. Uhlmann, Department of Mathematics\\
       University of Washington\\
       Seattle, WA  98195-4350\\
       USA\\
       Department of Mathematics and Statistics\\ 
       University of Helsinki\\ 
       Finland\\
        and Institute for Advanced Study of the Hong Kong University of Science and Technology}
\email{gunther@math.washington.edu}

\maketitle

\begin{abstract} 
We study inverse boundary problems for magnetic Schr\"odinger operators on a compact Riemannian manifold with boundary of dimension $\ge 3$. In the first part of the paper we are concerned with the case of admissible geometries, i.e. compact Riemannian manifolds with boundary which are conformally embedded in a product of the Euclidean line and a simple manifold. We show that the knowledge of the Cauchy data on the boundary of the manifold for the magnetic Schr\"odinger operator with $L^\infty$ magnetic and electric potentials, determines the magnetic field and electric potential uniquely.  

In the second part of the paper we address the case of more general conformally transversally anisotropic geometries, i.e. compact Riemannian manifolds with boundary which are conformally embedded in a product of the Euclidean line and a compact manifold, which need not be simple. Here, under the assumption that the geodesic ray transform on the transversal manifold is injective, we prove that the knowledge of the Cauchy data on the boundary of the manifold for a magnetic Schr\"odinger operator with continuous potentials, determines the magnetic field uniquely. 
Assuming that the electric potential is known, we show that the Cauchy data determines the magnetic potential up to a gauge equivalence. 
\end{abstract}

\section{Introduction and statement of results}

Let $(M,g)$ be a   smooth compact oriented Riemannian manifold of dimension $n\ge 3$ with smooth boundary $\p M$. Let $A\in C^\infty(M, T^*M)$ be a $1$-form with complex valued $C^\infty$ coefficients and let $q\in C^\infty(M,\C)$.  
Let $d: C^\infty(M)\to C^\infty(M, T^*M)$ be the de Rham differential and let us introduce 
\[
d_A=d+i A: C^\infty(M)\to C^\infty(M, T^*M).
\]
In geometric terms, the map $d_A$ can be viewed as a connection on the trivial line bundle $M\times\C$ over $M$. 

When $u,v\in C^\infty(M)$, we introduce the natural $L^2$--scalar product, 
\[
(u,v)_{L^2(M)}=\int_M u\overline{v}dV,
\] 
where $dV$ is the Riemannian volume element on $M$.   Similarly, when $\alpha,\beta\in C^\infty(M, T^*M)$, we define the $L^2$--scalar product in the space of $1$--forms,
\[
( \alpha, \beta)_{L^2(M, T^*M)}=\int_M  \langle \alpha, \overline{\beta}\rangle_g dV(x),
\]
where $\langle\cdot, \cdot\rangle_g$ is the pointwise scalar product in the space of $1$--forms induced by the Riemannian metric $g$. In local coordinates $(x_1,\dots, x_n)$ in which $\alpha=\alpha_j dx^j$, $\beta=\beta_j dx^j$ and $(g^{jk})$ is the matrix inverse of 
$(g_{jk})$, $g=g_{jk}dx^jdx^k$, we have
\[
\langle \alpha, \beta\rangle_g=g^{jk}\alpha_j \beta_k. 
\]
Here and in what follows we use Einstein's summation convention: repeated indices in lower and upper positions are summed. 

The formal $L^2$--adjoint $d_A^*: C^\infty(M, T^*M)\to C^\infty(M)$ of $d_A$ is defined by
\[
(d_A u, v)_{L^2(M,T^*M)}=(u,d^*_Av)_{L^2(M)}, \quad u\in C_0^\infty(M^0), \quad v\in C_0^\infty(M^0, T^*M^0),
\] 
where $M^0=M\setminus\p M$ stands for the interior of $M$. 
We have
\[
d_A^*=d^*-i \langle\overline{A}, \cdot\rangle_g. 
\]
Also in local coordinates,   we see that 
\begin{equation}
\label{eq_int_2}
d^*v=-|g|^{-\frac{1}{2}}\p_{x_j}(|g|^{\frac{1}{2}}g^{jk}v_k),
\end{equation}
where $|g|=\det(g_{jk})$ and $v=v_jdx^j$.

In this paper we shall be concerned with inverse boundary problems for the magnetic Schr\"odinger operator defined by
\begin{equation}
\label{eq_int_1}
\begin{aligned}
L_{g,A,q}u&=(d_{\overline{A}}^*d_A+q)u\\
&= -\Delta u +i d^*(Au)- i \langle A,du\rangle _g+ (\langle A, A\rangle_g+q)u,\quad u\in C^\infty(M).
\end{aligned}
\end{equation}
Specifically, our focus is on establishing global uniqueness results in the case when the magnetic potential $A$ is of low regularity.  Let us now proceed to introduce the precise assumptions and state the main results of this paper.  

In the first part of this paper we assume that $A\in L^\infty(M, T^*M)$ and $q\in L^\infty(M, \C)$. It follows then from \eqref{eq_int_1} that  
\begin{equation}
\label{eq_int_1_mapping_prop}
L_{g,A,q}: C^\infty_0(M^0)\to H^{-1}(M^0)
\end{equation}
is a bounded operator.  Here and in what follows $H^s(M^0)$, $s\in \R$, is the standard Sobolev space on $M^0$, see \cite[Chapter 4]{Taylor_book_1}, and  $d^*(Au)\in H^{-1}(M^0)$ is defined by 
\begin{equation}
\label{eq_int_2_dis}
\langle d^*(Au), \varphi\rangle_{M^0}:=\int_{M} \langle  Au, d\varphi\rangle_g dV,\quad u, \varphi\in C_0^\infty(M^0), 
\end{equation}
where $\langle \cdot, \cdot\rangle_{M^0}$ is the distributional duality on $M^0$.

Let $u\in H^1(M^0)$ be a solution to 
\begin{equation}
\label{eq_int_3}
L_{g,A,q} u=0
\end{equation}
in $\mathcal{D}'(M^0)$. Associated to $u$ is  the trace of the magnetic normal derivative $\langle d_Au , \nu \rangle_g\in H^{-1/2}(\p M)$ defined as follows, 
\begin{equation}
\label{eq_trace_normal_int}
\begin{aligned}
\langle \langle d_A u, \nu \rangle_g, f \rangle_{H^{-\frac{1}{2}}(\p M)\times H^{\frac{1}{2}}(\p M)}&:=(d_A u,d_{\overline{A}} \overline{v})_{L^2(M)}+(qu, \overline{v})_{L^2(M)}\\
&=\int_M \langle du+i A u, dv-iA v \rangle_g dV +\int_M q uv dV,
\end{aligned}
\end{equation}
where  $f\in H^{1/2}(\p M)$ and $v\in H^1(M^0)$  is  a continuous extension of $f$.  
As $u$ satisfies the equation \eqref{eq_int_3}, the above definition of the trace $\langle  d_A u, \nu \rangle_g$  is independent of the choice of an extension of $f$.  Let us also remark that   the definition \eqref{eq_trace_normal_int} is motivated by the integration by parts formula valid in the case of smooth potentials $A$ and $q$, 
\begin{equation}
\label{eq_int_integration_by_parts}
(L_{g,A,q} u, v)_{L^2(M)}= (d_A u, d_{\overline{A}} v)_{L^2(M)}+(qu,v)_{L^2(M)}-\int_{\p M} \langle d_A u,\nu \rangle_g \overline{v}dS,
\end{equation}
where  $u,v\in C^\infty(M)$ and $\nu $ is the conormal such that $\nu_j=g_{jk}\nu^k$ with $\nu=\nu^k \p_{x_k}$ being the unit outer normal to $\p M$. 

We shall next introduce the set of the Cauchy data for solutions of the magnetic Schr\"odinger equation given by 
\[
C_{g, A,q}:=\{ (u|_{\p M}, \langle d_A u, \nu \rangle_g|_{\p M}):  u\in H^1(M^0) \text{ satisfies } L_{g,A,q} u=0\text{ in } \mathcal{D}'(M^0) \}. 
\]
The inverse boundary value problem for the magnetic Schr\"odinger operator that we are interested in is to
determine $A$ and $q$ in  $M$ from the knowledge of the set of the Cauchy data $C_{g, A,q}$.  A well-known feature of this problem is that there is an obstruction to uniqueness given by the following class of gauge transformations, see \cite{Sun_1993}, \cite{DKSaloU_2009}. Let $F\in W^{1,\infty}(M^0)$ be a non-vanishing function. For any $u\in H^1(M^0)$,  we have
\[
(F^{-1}\circ L_{g,A,q}\circ F)u=L_{g,A -i F^{-1}dF,q}u, 
\]
in $\mathcal{D}'(M^0)$. Furthermore, if  $F|_{\p M}=1$, a computation using \eqref{eq_trace_normal_int} shows that 
\[
C_{g,A,q}=C_{g,A -i F^{-1}dF,q}.
\] 
Hence, from the knowledge of the set of the Cauchy data $C_{g,A,q}$ one may only hope to recover the magnetic field $dA$ and electric potential $q$ in $M$ uniquely, and the magnetic potential itself up to a gauge transformation.

The inverse boundary problem for the magnetic Schr\"odinger operator was studied extensively in the Euclidian setting with the goal of obtaining global uniqueness results under the minimal possible regularity assumptions on $A$. First in \cite{Sun_1993}, the global uniqueness was established for magnetic potentials in $W^{2,\infty}$, satisfying a smallness condition. In \cite{NakSunUlm_1995}, the smallness condition 
was eliminated for smooth magnetic and electric potentials, and for compactly supported $C^2$ magnetic potentials and $L^\infty$ electric potentials.  The uniqueness results were subsequently extended to 
$C^1$ magnetic potentials in \cite{Tolmasky_1998}, to some less regular but small potentials in \cite{Panchenko_2002}, and to Dini continuous magnetic potentials in \cite{Salo_diss}.  To the best of our knowledge, the sharpest results are obtained in \cite{Krup_Uhlmann_magnet} for $A\in L^\infty$ and $q\in L^\infty$ in $\R^n$, $n\ge 3$,   and in  \cite{Haberman_mag} for small $A\in W^{s,3}$ and $q\in W^{-1,3}$ in $\R^3$, for $s>0$.

Let us now turn the attention to the setting of  smooth compact Riemannian manifolds.  Here a powerful method to study inverse boundary problems based on the technique of  Carleman estimates with limiting Carleman weights was developed in \cite{DKSaloU_2009}. The concept of a limiting Carleman weight was first introduced and applied to inverse boundary problems in the Euclidean setting in \cite{Kenig_Sjostrand_Uhlmann}.  An important result of   \cite{DKSaloU_2009} states that on a simply connected open manifold, the existence of  a limiting Carleman weight is equivalent to the existence of a parallel unit vector field for a conformal multiple of the metric. Locally, the latter condition is equivalent to the fact that the manifold is conformal to the product of a Euclidean interval and some Riemannian manifold of dimension $n-1$.  Now it turns out that the existence of a limiting Carleman weight does not in itself suffice to solve the inverse boundary problem, and further conditions should be introduced on the transversal $(n-1)$-dimensional manifold. Following  \cite{DKSaloU_2009}, let us give the following definitions.

\begin{defn}
A compact manifold $(M_0, g_0)$ with boundary is called simple if for any $p\in M_0$, the exponential map $\exp_p$ with  its maximal domain of definition in $T_p M_0$ is a diffeomorphism onto $M_0$, and if $\p M_0$ is strictly convex in the sense that the second fundamental form of $\p M_0\hookrightarrow M_0$ is positive definite.  
\end{defn}

\begin{defn}
A  compact Riemannian manifold $(M,g)$ of dimension $n\ge 3$ with boundary $\p M$ is called admissible if there exists  an $(n-1)$--dimensional compact simple manifold $(M_0,g_0)$ such that $M\subset\R\times M_0$  and  $g=c(e\oplus g_0)$ where $e$ is the Euclidean metric on $\R$ and $c$ is a positive smooth function on $M$. 
\end{defn}

\textbf{Remark}.  It well known  that a simple manifold $M_0$ enjoys the following geometric and dynamical properties, see \cite{DKSaloU_2009}, \cite{Guillarmou_2017}. First, since the maximal domain of $\exp_p$  is starshaped,  $M_0$ is diffeomorphic to a closed ball, and  in particular, $M_0$ is simply connected.  Furthermore, $M_0$ has no conjugate points and no trapped geodesics, i.e. geodesics entirely contained in $M^0_0$, and between any two boundary points $x,x'\in \p M_0$ there is a unique geodesic in $M_0$ with the endpoints $x,x'$.   

An example of a simple manifold is a spherical cap strictly smaller than the northern hemisphere, 
\[
\mathbb{S}^{n-1}_{\ge \alpha_0}=\{x\in \mathbb{S}^{n-1}: x_n\ge \alpha_0\}, \quad 0<\alpha_0<1,
\]
while the hemisphere itself is not a simple manifold. 

Turning the attention to admissible manifolds, as explained in \cite{DKSaloU_2009}, examples include bounded smooth domains in the Euclidean space, in the sphere minus a point and in the hyperbolic space, sufficiently small subdomains of any conformally flat manifold, and bounded smooth domains in $\R^n$ equipped with a metric of the form,
\[
g(x_1,x')=c(x)\begin{pmatrix}
1& 0\\
0& g_0(x')
\end{pmatrix},
\] 
where $c$ is a positive smooth function and $g_0$ is a simple metric in the $x'$ variables.  

Starting with the fundamental work \cite{Syl_U}, a basic strategy for establishing global uniqueness in inverse boundary problems for elliptic equations consists in constructing complex geometric optics solutions for the equations in question. It turns out that the setting of admissible manifolds is particularly well adapted to such constructions as the eikonal and transport equations can be solved globally in suitable  global coordinates on an admissible manifold, once the limiting Carleman weight, governing the exponential growth and decay of the solutions, has been chosen. To conclude the proof of a global uniqueness result in the context of admissible manifolds, it turns out that one should invert an attenuated geodesic ray transform on the transversal manifold. Thanks to the work \cite{DKSaloU_2009},  extending the previous results of \cite{Anikonov}, \cite{Mukhometov}, \cite{Sharafutdinov_1} and \cite{Sharafutdinov_2}, the injectivity of the latter transform is known in the case of simple manifolds.

The inverse boundary problem for the magnetic Schr\"odinger operator on admissible manifolds was studied in  \cite{DKSaloU_2009} and the global uniqueness was established for $C^\infty$ electric and magnetic potentials. To the best of our knowledge, the problem of weakening the regularity assumptions on the potentials in the context of admissible manifolds was only addressed in \cite{DKSalo_2013} when $A=0$ and $q\in L^{\frac{n}{2}}(M)$.  The first result of our paper is a global uniqueness result on admissible manifolds for electric and magnetic potentials that are merely bounded.  It can be viewed as an analog of our previous result \cite{Krup_Uhlmann_magnet} in the Euclidean case, in the setting of   admissible manifolds.

\begin{thm}
\label{thm_main}
Let $(M,g)$ be admissible. Let $A^{(1)}, A^{(2)}\in L^\infty(M, T^*M)$ be complex valued $1$-forms, and $q^{(1)},q^{(2)}\in L^\infty(M,\C)$. If $C_{g, A^{(1)},q^{(1)}}=C_{g, A^{(2)},q^{(2)}}$, then $dA^{(1)}=dA^{(2)}$ and $q^{(1)}=q^{(2)}$. 
 \end{thm}

Let us point out that the main difficulty in proving Theorem \ref{thm_main} is due to the fact that when $A\in L^\infty$, the operator $L_{g,A,q}$ has singular coefficients, see \eqref{eq_int_1} and \eqref{eq_int_1_mapping_prop}. To overcome this difficulty when constructing complex geometric optics solutions to the magnetic Schr\"odinger equation, we shall first prove a Carleman estimate with a gain of two derivatives on a manifold supporting a limiting Carleman weight, and use it to construct complex geometric optics solutions on an admissible manifold. When doing so, we also rely crucially on a smoothing argument, approximating the $L^\infty$ magnetic potential by smooth 1-forms in the $L^2$ sense.   To complete the proof of Theorem  \ref{thm_main}, we exploit a result related to the injectivity of the attenuated ray transform acting in the space of  $L^\infty$ functions and $L^\infty$ 1-forms, established in \cite[Proposition 5.1]{DKSalo_2013} and \cite[Proposition 5.1]{Assylbekov_Yang_2017}.  In \cite{DKSalo_2013} and \cite{Assylbekov_Yang_2017}, in order to circumvent the difficulty related to the fact that $L^\infty$ functions and forms cannot be restricted to geodesics, the authors use duality arguments and the ellipticity of the normal operator, associated to the attenuated geodesic ray transform.

In the second part of the paper, following \cite{DKuLS_2016}, we are concerned with removing the simplicity assumption on the transversal manifold $M_0$. To that end, we have the following definition. 

\begin{defn}
Let $(M,g)$ be a smooth compact Riemannian manifold of dimension $n\ge 3$ with smooth boundary. We say that  $(M,g)$ is conformally transversally anisotropic if there exists  a smooth compact Riemannian manifold  $(M_0,g_0)$ with smooth boundary of dimension $n-1$ such that 
$(M,g)$ is conformally embedded into a manifold of the form $(\R\times M_0, e\oplus g_0)$. 
\end{defn}

An important role in what follows is played by the geodesic ray transform on the transversal manifold $(M_0,g_0)$ and let us proceed to recall its definition following \cite{Guillarmou_2017}, \cite{DKSaloU_2009}. The geodesics on $M_0$ can be parametrized by points on the unit sphere bundle $SM_0=\{(x,\xi)\in TM_0: |\xi|=1\}$.  Let 
\[
\p_\pm SM_0=\{ (x,\xi)\in SM_0: x\in \p M_0, \pm \langle\xi,  \nu(x) \rangle>0\},
\] 
be the incoming ($-$) and outgoing ($+$) boundaries of $SM_0$. Here  $\nu$ is the unit outer normal vector field to $\p M_0$. Here and in what follows $\langle \cdot,\cdot\rangle$ is the duality between $T^*M_0$ and $TM_0$.

Let $(x,\xi)\in \p_-SM_0$ and let $\gamma=\gamma_{x,\xi}(t)$ be the geodesic on $M_0$ such that $\gamma(0)=x$ and $\dot{\gamma}(0)=\xi$. Let us denote by $\tau(x,\xi)$ the first time when the geodesic  $\gamma$ exits $M_0$ with the convention that $\tau(x,\xi)=+\infty$ if the geodesic does not exit $M_0$. We define the incoming tail by
\[
\Gamma_-=\{(x,\xi)\in \p_-SM_0:\tau(x,\xi)=+\infty\}.
\]
When $f\in C(M_0,\C)$ and $\alpha\in C(M_0,T^*M_0)$ is a complex valued $1$-form, we define the geodesic ray transform on $(M_0,g_0)$  as follows
\[
I(f,\alpha)(x,\xi)=\int_0^{\tau(x,\xi)} \big[ f(\gamma_{x,\xi}(t))+ \langle \alpha(\gamma_{x,\xi}(t)), \dot{\gamma}_{x,\xi}(t) \rangle \big]dt, \quad (x,\xi)\in \p_-SM_0\setminus\Gamma_-.
\]

A unit speed geodesic segment $\gamma=\gamma_{x,\xi}:[0,\tau(x,\xi)]\to M_0$, $\tau(x,\xi)>0$, is called non-tangential if $\gamma(0),\gamma(\tau(x,\xi))\in \p M_0$, $\dot{\gamma}(0),\dot{\gamma}(\tau(x,\xi))$ are non-tangential vectors on $\p M_0$, and $\gamma(t)\in M_0^0$ for all $0<t<\tau(x,\xi)$.

\textbf{Assumption 1.} \textit{We assume that the geodesic ray transform on $(M_0,g_0)$ is injective in the sense that  if $I(f,\alpha)(x,\xi)=0$ for all $(x,\xi)\in \p_-SM_0\setminus\Gamma_-$  such that  $\gamma_{x,\xi}$ is a non-tangential geodesic then  $f=0$ and $\alpha=dp$ in $M_0$ for some $p\in C^1(M_0,\C)$ with $p|_{\p M_0}=0$}. 

The second principal result of this paper is the following generalization of \cite{DKuLS_2016} and \cite{Cekic} to the case of continuous magnetic potentials. In \cite{DKuLS_2016} it is assumed that $A=0$ and $q$ is continuous, while in \cite{Cekic} one considers  $A\in C^\infty$ and $q=0$. 

\begin{thm}
\label{thm_main_2}
Let $(M,g)$ be a conformally transversally anisotropic manifold such that  Assumption 1 holds for the transversal manifold.  Let $A^{(1)}, A^{(2)}\in C(M, T^*M)$ be complex valued $1$-forms, and $q^{(1)},q^{(2)}\in L^\infty(M,\C)$.  
If $C_{g,A^{(1)},q^{(1)}}=C_{g, A^{(2)},q^{(2)}}$, then $dA^{(1)}=dA^{(2)}$. Assuming furthermore that $q^{(1)}=q^{(2)}$, we have
\[
A^{(2)}=A^{(1)}-iF^{-1}dF,
\]
 for some $F\in C^1(M,\C)$ non-vanishing with $F|_{\p M}=1$. 
 \end{thm}

Let us now proceed to give some examples of non-simple manifolds satisfying Assumption 1. 

\begin{itemize}

\item In \cite{Uhl_V_2016}, \cite{Stef_Uhl_Vas} the injectivity of  the geodesic ray transform is obtained when $M_0$ is a compact Riemannian manifold with strictly convex boundary, foliated by strictly convex hypersurfaces $\Sigma_t$, $0\le t<T$, such that $M_0\setminus \cup_{t\in [0,T)}\Sigma_t$ has measure zero, for $f\in L^2(M_0)$ and $\alpha\in L^2(M_0, T^*M_0)$. Furthermore, if $M_0\setminus \cup_{t\in [0,T)}\Sigma_t$ has empty interior the injectivity holds for $f\in C(M_0)$ and $\alpha\in C(M_0, T^*M_0)$. 

\item In \cite{Guillarmou_2017} the injectivity of the geodesic ray transform is established when $M_0$ is a compact Riemannian manifold with strictly convex boundary such that the geodesic flow has no conjugate points and the trapped set is hyperbolic, and for $f\in L^p(M_0)+H^{-1/2}_{\text{comp}}(M_0^0)$ with $p>2$ and $\alpha\in C^\infty(M_0,T^*M_0)+H^{-1/2}_{\text{comp}}(M_0,T^*M_0)$.  As an example of such manifold $M_0$ one can consider a manifold with negative sectional curvature and strictly convex boundary. 

\end{itemize}

When proving Theorem \ref{thm_main_2}, we still exploit the existence of the natural Carleman weight $\varphi(x)=x_1$ on $M$ and the corresponding Carleman estimate with a gain of two derivatives.  On the other hand, since  the transversal manifold $M_0$ need not be simple, complex geometric optics solutions can no longer easily be constructed by means of a global WKB method. Following  \cite{DKuLS_2016}, when constructing complex geometric optics solutions, we replace global WKB quasimodes by Gaussian beam quasimodes, localized near non-tangential geodesics on the transversal manifold.  As we know from the Euclidean case \cite{Krup_Uhlmann_magnet}, for the solution of the inverse problem, it suffices to construct  $o(h)$--quasimodes for the semiclassically rescaled conjugated magnetic Schr\"odinger operator, and we carry out this construction for continuous magnetic potentials, combining the techniques of  \cite{DKuLS_2016} with those based on regularization. 
Exploiting the concentration properties of the corresponding complex geometric optics solutions together with Assumption 1, we conclude, similarly to  \cite{Cekic},  that $dA^{(1)}=dA^{(2)}$. Here we also make use of a boundary reconstruction result for continuous magnetic potentials. 

Finally, assuming that $q^{(1)}=q^{(2)}$, using parallel transport along loops in $M$, boundary reconstruction of the magnetic potential, and unique continuation arguments, as in \cite{Guillarmou_Zhou} and  \cite{Cekic}, we show that the fluxes of the magnetic potentials $A^{(1)}$ and $A^{(2)}$ satisfy
\[
\int_\gamma(A^{(1)}-A^{(2)})\in 2\pi \Z,
\]
when $\gamma$ is a  loop on $M$, allowing us to construct the gauge $F$ and thus, obtain the second statement in Theorem \ref{thm_main_2}.

The plan of the paper is as follows. In Section \ref{sec_Carleman_estimates}, we derive Carleman estimates with a gain of two derivatives, leading to a solvability result for the conjugated magnetic Schr\"odinger operator. Complex geometric optics solutions for the magnetic Schr\"odinger operator with bounded potentials on an admissible manifold are constructed in Section \ref{sec_CGO_admissible}, and the proof of Theorem \ref{thm_main} is concluded in Section \ref {sec_proof_thm_main_1}.  Section \ref{sec_Gaussian_beam} is concerned with the construction of Gaussian beam quasimodes on a conformally transversally anisotropic manifold, and in Section \ref{sec_CGO_based_quasi}  complex geometric optics solutions on such manifolds are obtained. The magnetic field is determined in Section  \ref{sec_thm_main_2},  thereby establishing the first part of Theorem \ref{thm_main_2}. The proof of Theorem \ref{thm_main_2}  is concluded in Section \ref{sec_hol} where the existence of a gauge is proved.  Finally, in Appendix \ref{sec_boundary_rec} the boundary determination of a continuous magnetic potential on a compact manifold with boundary, from the set of the Cauchy data, is shown.

\section{Carleman estimates with a gain of two derivatives}
\label{sec_Carleman_estimates}

The purpose of this section is to prove a Carleman estimates with a gain of two derivatives for $-h^2\Delta$ on a Riemannian manifold admitting limiting Carleman weights. This can be viewed as an extension of \cite[Lemma 2.1]{Salo_Tzou_2009} from the Euclidean setting  to that of  Riemannian manifolds. 

Let $(M, g)$ be a compact smooth Riemannian manifold with boundary.  Assume that $(M,g)$ is  embedded in a compact smooth manifold $(N,g)$ without boundary of the same dimension, and let $U$ be open in $N$ such that $M\subset U$.  In the discussion below it will be convenient to rely on the standard calculus of semiclassical pseudodifferential operators on $N$, obtained by quantizing the Kohn--Nirenberg symbol classes $S^m(T^*N)$, given by 
\[
S^m(T^*N)=\{a(x,\xi;h)\in C^\infty(T^*N):|\p_x^\alpha\p_\xi^\beta a(x,\xi;h)|\le C_{\alpha\beta}\langle \xi\rangle^{m-|\beta|}\},
\]
$0<h\le 1$ and $m\in \R$.  The local formula for the standard $h$-quantization, 
\begin{equation}
\label{eq_Carleman_quant}
\text{Op}_h(a)u(x)=\frac{1}{(2\pi h)^n}\int\!\int e^{\frac{i}{h}(x-y)\cdot\xi}a(x,\xi;h) u(y)dyd\xi, \quad a\in S^m,
\end{equation}
defines a class of semiclassical pseudodifferential operators on $N$ which will be denoted by $\Psi^m(N)$. We shall fix a choice of 
the quantization map 
\[
\text{Op}_h: S^m(T^*N)\to \Psi^m(N), 
\]
given by  \eqref{eq_Carleman_quant}  in local coordinate charts.   We refer to \cite[Chapter 14]{Zworski_book} for the semiclassical pseudodifferential calculus on $N$. 

Let us also recall G{\aa}rding's inequality, which plays an important role below,  see \cite{Lebeau_Carleman}. When doing so, we let  $H^s(N)$, $s\in\R$, be the standard  Sobolev space, equipped with the natural semiclassical norm,
\[
\|u\|_{H^s_{\text{scl}}(N)}=\|(1-h^2\Delta)^{\frac{s}{2}} u\|_{L^2(N)}. 
\]

\begin{thm}
\label{thm_Garding}
Let $p\in S^m(T^*N)$ be such that there exists $C>0$ such that  
\[
\emph{\Re} p(x,\xi)\ge \frac{1}{C}\langle \xi \rangle^m, \quad (x,\xi)\in T^*N. 
\]
Then there exists $h_0>0$ such that  for $h\in (0,h_0]$, we have
\[
\emph{\Re} (\emph{\text{Op}}_h(p)u,u)_{L^2(N)}\ge \frac{1}{2C}\|u\|^2_{H^{\frac{m}{2}}_{\emph{\text{scl}}}(N)}, 
\]
for all $u\in C^\infty(N)$. 
\end{thm}

Let  $\varphi\in C^\infty(U,\R)$ and let us consider the conjugated operator 
\begin{equation}
\label{eq_Car_-2}
P_\varphi=e^{\frac{\varphi}{h}}(-h^2\Delta)e^{-\frac{\varphi}{h}}=-h^2\Delta -|\nabla \varphi|^2+2\langle \nabla \varphi, h\nabla\rangle +h\Delta\varphi,
\end{equation}
with the semiclassical principal symbol 
\begin{equation}
\label{eq_Car_1}
p_\varphi=|\xi|^2-|d\varphi|^2+2i \langle \xi, d\varphi\rangle\in C^\infty(T^*U).
\end{equation}
Here and in what follows we use $\langle \cdot, \cdot \rangle$ and $|\cdot|$ to denote the Riemannian scalar product and norm both on the tangent and cotangent space.  

When $(x,\xi)\in T^*M$ and $|\xi|\ge C\gg 1$, we have that $|p_\varphi(x,\xi)|\sim |\xi|^2$ so that $P_\varphi$ is elliptic at infinity in the semiclassical sense. Following \cite{Kenig_Sjostrand_Uhlmann},  \cite{DKSaloU_2009}, we say that $\varphi\in C^\infty(U,\R)$ is a limiting Carleman weight for $-h^2\Delta$ on $(U,g)$ if $d\varphi\ne 0$ on $U$, and the Poisson bracket of $\Re p_\varphi$ and $\Im p_\varphi$ satisfies,
\[
\{\Re p_\varphi, \Im p_\varphi\}=0 \quad\text{when}\quad  p_\varphi=0. 
\]
We refer to \cite{DKSaloU_2009} for a characterization of Riemannian manifolds admitting limiting Carleman weights.

The following is the main result of this section.  
\begin{prop}
\label{prop_Carleman_est_gain_2}
Let $\varphi$ be a limiting Carleman weight for $-h^2\Delta$ on $(U,g)$ and let $\tilde \varphi=\varphi+\frac{h}{2\varepsilon}\varphi^2$. Then for all $0<h\ll \varepsilon\ll 1$ and $s\in \R$, we have
\begin{equation}
\label{eq_Car_for_laplace}
\frac{h}{\sqrt{\varepsilon}}\|u\|_{H^{s+2}_{\emph{\text{scl}}}(N)}\le C\|e^{\frac{\tilde \varphi }{h}}(-h^2\Delta)e^{-\frac{\tilde \varphi}{h}}  u\|_{H^{s}_{\emph{\text{scl}}}(N)}, \quad C>0,
\end{equation}
for all $u\in C_0^\infty(M^0)$. 
\end{prop} 

\begin{proof} 

We shall first establish \eqref{eq_Car_for_laplace} in the case $s=0$. Let us explain that when doing so we can assume that the limiting Carleman weight $\varphi$ on $(U,g)$ is also a distance function in the sense that 
\begin{equation}
\label{eq_conformal_prop_0}
|\nabla\varphi|=1. 
\end{equation}
Indeed,  by  \cite[Lemma 2.1]{DKSaloU_2009} we know that  $\varphi$ is a limiting Carleman weight on the conformal manifold $(U,cg)$ with $0<c\in C^\infty(U)$.  Taking $c=|\nabla_g \varphi|_g^2$, we see that $|\nabla_{cg} \varphi|_{cg}=1$, and hence, the Carleman weight $\varphi$ is also a distance function on $(U,cg)$. Assume now that we have proved that  for all $0<h\ll \varepsilon\ll 1$ and $s\in \R$, we have
\begin{equation}
\label{eq_conformal_prop_1}
\frac{h}{\sqrt{\varepsilon}}\|u\|_{H^{2}_{\text{scl}}(N)}\le C\|e^{\frac{\tilde \varphi }{h}}(-h^2\Delta_{cg})e^{-\frac{\tilde \varphi}{h}}  u\|_{L^2(N)}, \quad C>0,
\end{equation}
for all $u\in C_0^\infty(M^0)$.  By the conformal properties of the Laplace--Beltrami operator, we have
\begin{equation}
\label{eq_conformal_prop_2}
c^{-\frac{n+2}{4}} (-\Delta_g) (c^{\frac{(n-2)}{4}} u)=-\Delta_{c g} u+q_c u, \quad q_c=c^{-\frac{n+2}{4}} (-\Delta_g) (c^{\frac{(n-2)}{4}})\in C^\infty(U).
\end{equation}
Combining  \eqref{eq_conformal_prop_1} and \eqref{eq_conformal_prop_2}, we obtain \eqref{eq_conformal_prop_1} for  the operator $-h^2\Delta_g$.

In what follows when proving \eqref{eq_Car_for_laplace} in the case $s=0$,   we assume therefore that the limiting Carleman weight  $\varphi$ satisfies \eqref{eq_conformal_prop_0}. 
In view of \eqref{eq_Car_1}, we see that 
\begin{equation}
\label{eq_Car_2}
\begin{aligned}
\Re p_\varphi&=|\xi|^2-|d\varphi|^2=|\xi^\sharp|^2-|\nabla \varphi|^2,\\
\Im p_\varphi&=2\langle \xi, d\varphi\rangle =2\langle \nabla \varphi, \xi^\sharp\rangle,
\end{aligned}
\end{equation}
where the vector field  $\xi^\sharp$ is given by $\xi^\sharp=g^{jk}\xi_j \p_{x_k}$ in local coordinates. 

By \cite[Lemma 2.3]{DKSaloU_2009}, we have
\begin{equation}
\label{eq_Car_3}
\{\Re p_\varphi, \Im p_\varphi \}(x,\xi)=4D^2\varphi(\nabla \varphi, \nabla \varphi)+4D^2\varphi(\xi^\sharp, \xi^\sharp),
\end{equation}
where $D^2\varphi$ is the Hessian of $\varphi$. Recall from \cite[Appendix]{DKSaloU_2009} that the Hessian of a smooth function $\varphi$ is the symmetric $(2,0)$--tensor $D^2\varphi=Dd\varphi$, where $D$ is the Levi--Civita connection on $(M,g)$.

Since $\varphi$ is both a Carleman weight and a distance function, it follows from  \cite[Lemma 2.5]{DKSaloU_2009} that the Hessian satisfies
\[
D^2\varphi=0 \quad \text{on}\quad U,
\] 
and thus, \eqref{eq_Car_3} implies that $\{\Re p_\varphi, \Im p_\varphi \}(x,\xi)=0$ on $T^*U$. 

Consider now $\tilde \varphi$  instead of $\varphi$. We have
\begin{align*}
\nabla \tilde\varphi &=\bigg(1+\frac{h}{\varepsilon}\varphi\bigg)\nabla \varphi,\\
D^2\tilde \varphi&=\frac{h}{\varepsilon} d\varphi\otimes d\varphi + \bigg(1+\frac{h}{\varepsilon}\varphi\bigg) D^2\varphi=\frac{h}{\varepsilon} d\varphi\otimes d\varphi. 
\end{align*}
Therefore, using \eqref{eq_Car_2}, we get 
\begin{align*}
\Re p_{\tilde \varphi} &=|\xi^\sharp|^2- \bigg(1+\frac{h}{\varepsilon}\varphi\bigg)^2 |\nabla \varphi|^2,\\
\Im p_{\tilde \varphi} &=2 \bigg(1+\frac{h}{\varepsilon}\varphi\bigg) \langle \nabla \varphi, \xi^\sharp\rangle,
\end{align*}
and using  \eqref{eq_Car_3} and the fact that $|\nabla \varphi|=1$,  we also obtain that 
\begin{equation}
\label{eq_Car_4}
\begin{aligned}
\{\Re p_{\tilde \varphi}, \Im p_{\tilde \varphi} \}(x,\xi)&=4\frac{h}{\varepsilon}\bigg(1+\frac{h}{\varepsilon}\varphi\bigg)^2 (d\varphi\otimes d\varphi) (\nabla \varphi, \nabla \varphi)+  4\frac{h}{\varepsilon} (d\varphi\otimes d\varphi)(\xi^\sharp, \xi^\sharp) \\
&=4\frac{h}{\varepsilon}\bigg(1+\frac{h}{\varepsilon}\varphi\bigg)^2 |\nabla \varphi|^4+ 4\frac{h}{\varepsilon}\langle \nabla \varphi, \xi^\sharp\rangle^2\\
&=4\frac{h}{\varepsilon}\bigg(1+\frac{h}{\varepsilon}\varphi\bigg)^2+ \beta (\Im p_{\tilde \varphi} )^2.
\end{aligned}
\end{equation}
Here 
\[
\beta=\frac{h}{\varepsilon}\bigg(1+\frac{h}{\varepsilon}\varphi\bigg)^{-2},
\]
and we choose $\frac{h}{\varepsilon}\le \varepsilon_0<1$ with $\varepsilon_0$ small enough so that $(1+\frac{h}{\varepsilon}\varphi)\ge \frac{1}{2}$ on $M$. 

Let 
\begin{equation}
\label{eq_Car_5_0}
P_{\tilde \varphi}=e^{\frac{\tilde \varphi }{h}}(-h^2\Delta)e^{-\frac{\tilde \varphi}{h}}  
=-h^2\Delta -|\nabla \tilde \varphi|^2+2\langle \nabla \tilde \varphi, h\nabla\rangle +h\Delta\tilde \varphi= \tilde A+i \tilde B, 
\end{equation}
where $\tilde A$ and $\tilde B$ are formally self-adjoint operators on $L^2(U)$ given by 
\begin{equation}
\label{eq_Car_5_0-real-im}
\tilde A=-h^2\Delta -|\nabla \tilde \varphi|^2, \quad \tilde B=-2i\langle \nabla \tilde \varphi, h\nabla\rangle -i h\Delta\tilde \varphi.
\end{equation}
Letting  $u\in C_0^\infty(M^0)$, and integrating by parts, we get
\begin{equation}
\label{eq_Car_5}
\|P_{\tilde \varphi}u\|^2_{L^2(M)}=\|\tilde A u\|^2_{L^2(M)}+ \|\tilde B u\|^2_{L^2(M)}+ i([\tilde A, \tilde B]u, u)_{L^2(M)}. 
\end{equation}

We have  
\begin{equation}
\label{eq_Car_6}
i [\tilde A, \tilde B]=h\text{Op}_h( \{\Re p_{\tilde \varphi}, \Im p_{\tilde \varphi} \}   )+h^2 R_1, 
\end{equation}
where $R_1=\text{Op}_h(r_1)$ with $r_1\in S^1$ uniformly in $h$ and $\varepsilon$  in view of  \eqref{eq_Car_5_0} and \eqref{eq_Car_5_0-real-im}. 
Using  \eqref{eq_Car_4}, we can write
\begin{equation}
\label{eq_Car_7}
h \{\Re p_{\tilde \varphi}, \Im p_{\tilde \varphi} \}=4\frac{h^2}{\varepsilon}\bigg(1+\frac{h}{\varepsilon}\varphi\bigg)^2+ h  \beta (\Im p_{\tilde \varphi})^2=\frac{h^2}{\varepsilon} d + h  \beta (\Im p_{\tilde \varphi})^2-\frac{h^2}{\varepsilon} (\Re p_{\tilde \varphi})^2,
\end{equation}
where the symbol
\[
d=4\bigg(1+\frac{h}{\varepsilon}\varphi\bigg)^2+(\Re p_{\tilde \varphi})^2\in S^4(T^*U),
\]
satisfies 
\[
d(x,\xi)\ge \frac{1}{C_0}\langle \xi\rangle^4, \quad (x,\xi)\in T^*\tilde U
\]
with  $C_0>0$ independent of $\varepsilon$ and $h$. Here $M\subset\subset \tilde U\subset \subset U$, $\tilde U$ is open.    This follows from the fact that the symbol $(\Re p_{\tilde \varphi})^2$ is elliptic for $|\xi|$ large. 

An application of  G{\aa}rding's inequality, Theorem \ref{thm_Garding}, allows us to conclude that
\begin{equation}
\label{eq_Car_8}
(\text{Op}_h(d)u, u)_{L^2(N)}\ge \frac{1}{2C_0}\|u\|_{H^2_{\text{scl}}(N)}^2, \quad u\in C^\infty_0(M^0),
\end{equation}
for all $0<h<1$ small enough. 

Using \eqref{eq_Car_6}, \eqref{eq_Car_7}, and the fact that 
\[
\tilde A^2=\text{Op}_h((\Re p_{\tilde \varphi})^2)+hR_2, 
\]
where $R_2=\text{Op}_h(r_2)$ with $r_2\in S^3$ uniformly in $h$ and $\varepsilon$, 
 we obtain that 
\begin{equation}
\label{eq_Car_9}
i [\tilde A, \tilde B]=\frac{h^2}{\varepsilon} \text{Op}_h (d)+   h \tilde B\beta \tilde B- \frac{h^2}{\varepsilon} (\tilde A^2-hR_2) +h^2 R_3,
\end{equation}
where $R_3=\text{Op}_h(r_3)$  with $r_3\in S^1$ uniformly in $h$ and $\varepsilon$. 

Using \eqref{eq_Car_5},  \eqref{eq_Car_8}, and \eqref{eq_Car_9},  we get  that for all $0<h<1$ small enough, 
\begin{equation}
\label{eq_Car_10}
\begin{aligned}
\|P_{\tilde \varphi}u\|^2_{L^2(N)}&\ge \frac{h^2}{2C_0\varepsilon}\|u\|_{H^2_{\text{scl}}(N)}^2+\|\tilde A u\|^2_{L^2(N)}+ \|\tilde B u\|^2_{L^2(N)}
-\mathcal{O}(h)\|\tilde B u\|^2_{L^2(N)}\\
&-\frac{h^2}{\varepsilon}\|\tilde A u \|^2_{L^2(N)}- \mathcal{O}(h^2)\|u\|^2_{H^2_{\text{scl}}(N)}.
\end{aligned}
\end{equation}
Here we have used that 
\[
| (R_j u, u)_{L^2(N)}|\le \|R_j u\|_{H^{-2}_{\text{scl}}(N)}\|u\|_{H^2_{\text{scl}}(N)}\le \mathcal{O}(1)\|u\|^2_{H^2_{\text{scl}}(N)}, \quad j=2,3,
\]
uniformly in $h$.  

Now we conclude from \eqref{eq_Car_10} that for all $0<h\ll \varepsilon\ll 1$, 
\begin{equation}
\label{eq_Car_11}
\|P_{\tilde \varphi}u\|^2_{L^2(N)}\ge  \frac{h^2}{C\varepsilon}\|u\|_{H^2_{\text{scl}}(N)}^2, \quad u\in C^\infty_0(M^0). 
\end{equation}
This completes the proof of \eqref{eq_Car_for_laplace} in the case $s=0$.

We shall next establish \eqref{eq_Car_for_laplace} for an arbitrary $s\in \R$. In doing so let us set 
\[
J^s=(1-h^2\Delta)^{\frac{s}{2}}, \quad s\in \R,
\]
defined by means of the spectral theorem, and let us recall the basic fact that  $J^s\in \text{Op}_h(S^s(T^*N))$, see \cite[Proposition 10.1]{Sjostrand_2009}. We then have the following
pseudolocal estimate: if $\psi,\chi\in C^\infty(N)$ with $\chi=1$ near $\text{supp} (\psi)$ and if $s,\alpha, \beta\in \R$, then 
\begin{equation}
\label{eq_Car_12}
(1-\chi)J^s\psi =\mathcal{O}(h^\infty): H^\alpha_{\text{scl}}(N)\to H^\beta_{\text{scl}}(N).
\end{equation}

Let $\chi\in C^\infty_0(U)$ be such that $\chi=1$ near $M$. Then using \eqref{eq_Car_11} for functions supported on a slightly larger set than $M^0$, as well as \eqref{eq_Car_12},   we have for all $0<h\ll \varepsilon\ll 1$ and  $u\in C^\infty_0(M^0)$, 
\begin{equation}
\label{eq_Car_13}
\begin{aligned}
h\|u\|_{H^{s+2}_{\text{scl}}(N)} & \le   h\| \chi J^s u\|_{H^{2}_{\text{scl}}(N)} +h\| (1- \chi) J^s u\|_{H^{2}_{\text{scl}}(N)}\\
&\le
\mathcal{O}(\sqrt{\varepsilon})\| P_{\tilde \varphi} \chi J^s u\|_{L^2 (N)} +\mathcal{O}(h^\infty)\|u\|_{H^{s+2}_{\text{scl}}(N)} \\
& \le \mathcal{O}(\sqrt{\varepsilon})\big (\|  \chi P_{\tilde \varphi}  J^s u\|_{L^2(N)} + \| [P_{\tilde \varphi} ,  \chi ] J^s u\|_{L^2(N)}\big) +\mathcal{O}(h^\infty)\|u\|_{H^{s+2}_{\text{scl}}(N)} \\
&\le \mathcal{O}( \sqrt{\varepsilon})\|  \chi P_{\tilde \varphi}  J^s u\|_{L^2(N)} +\mathcal{O}(h^\infty)\|u\|_{H^{s+2}_{\text{scl}}(N)}. 
\end{aligned}
\end{equation}
Here we have used that 
\[
\| [P_{\tilde \varphi} ,  \chi ] J^s u\|_{L^2(N)}\le \mathcal{O}(h^\infty)\|u\|_{H^{s+2}_{\text{scl}}(N)},
\]
which is a consequence of the pseudolocality \eqref{eq_Car_12}. By absorbing the error term $\mathcal{O}(h^\infty) \|u\|_{H^{s+2}_{\text{scl}}(N)}$ in the left hand side of \eqref{eq_Car_13}, for all $0<h\ll \varepsilon\ll 1$, we get 
\begin{equation}
\label{eq_Car_14}
h\|u\|_{H^{s+2}_{\text{scl}}(N)}\le \mathcal{O}( \sqrt{\varepsilon})(\|  \chi   J^s P_{\tilde \varphi} u\|_{L^2(N)} + \|  \chi  [P_{\tilde \varphi},  J^s ] u\|_{L^2(N)}).
\end{equation}

When estimating the last term in  \eqref{eq_Car_14}, we extend $\varphi$ smoothly outside of $U$, and use the fact that $ \chi  [P_{\tilde \varphi},  J^s ]\in h\text{Op}_h(S^{s+1}(T^*N))$. Hence, for all $0<h\ll \varepsilon\ll 1$,
\begin{equation}
\label{eq_Car_15}
\chi  [P_{\tilde \varphi},  J^s ]=\mathcal{O}(h): H^{s+1}_{\text{scl}}(N)\to L^2(N).
\end{equation}
It follows from \eqref{eq_Car_14} and \eqref{eq_Car_15} that 
\[
h\|u\|_{H^{s+2}_{\text{scl}}(N)}\le \mathcal{O} (\sqrt{\varepsilon})\|   J^s P_{\tilde \varphi} u\|_{L^2(N)} + \mathcal{O}(\sqrt{\varepsilon} h)\|u\|_{H^{s+1}_{\text{scl}}(N)},
\]
and hence, 
\[
h\|u\|_{H^{s+2}_{\text{scl}}(N)}\le \mathcal{O} (\sqrt{\varepsilon})\|   P_{\tilde \varphi} u\|_{H^{s}_{\text{scl}}(N)}.  
\]
This completes the proof of the proposition. 
\end{proof}

We shall next state the following Carleman estimate for the magnetic Schr\"odinger operator, which generalizes \cite[Proposition 2.2]{Krup_Uhlmann_magnet} to the Riemannian setting. 

\begin{prop}
\label{prop_Carleman_magnetic_mnfld}
Let $\varphi$ be a limiting Carleman weight for $-h^2\Delta$ on $(U,g)$ and let $A\in L^\infty(M, T^*M)$, $q\in L^\infty(M,\C)$. Then for all $0<h\ll 1$, we have
\begin{equation}
\label{eq_Car_for_magnetic}
h \|u\|_{H^{1}_{\emph{\text{scl}}}(N)}\le C\|e^{\frac{ \varphi }{h}}(h^2L_{g,A,q})e^{-\frac{\varphi}{h}}  u\|_{H^{-1}_{\emph{\text{scl}}}(N)}, \quad C>0,
\end{equation}
for all $u\in C_0^\infty(M^0)$. 
\end{prop} 
\begin{proof}

To prove the estimate \eqref{eq_Car_for_magnetic} we shall use the following characterization of the semiclassical norm in the Sobolev space $H^{-1}(N)$, 
\begin{equation}
\label{eq_Car_16}
\|v\|_{H^{-1}_{\text{scl}}(N)}=\sup_{0\ne \psi\in C^\infty(N)}\frac{|\langle v,\psi\rangle_N|}{\| \psi\|_{H^1_{\text{scl}}(N)}}. 
\end{equation}
Let $\tilde \varphi=\varphi+\frac{h}{2\varepsilon}\varphi^2$ with $\varepsilon>0$ be such that $0<h\ll \varepsilon\ll 1$, and let $u\in C^\infty_0(M^0)$. Then for all $0\ne \psi\in C^\infty(N)$, using \eqref{eq_int_2_dis}, we get 
\begin{equation}
\label{eq_Car_17}
\begin{aligned}
&|\langle e^{\frac{\tilde \varphi}{h}} h^2  d^*(A e^{-\frac{\tilde\varphi}{h}}u) , \psi\rangle_N|=\bigg| h^2\int_N \langle Ae^{-\frac{\tilde\varphi}{h}}u, d(e^{\frac{\tilde\varphi}{h}} \psi)\rangle_g dV \bigg|\\
&\le h^2 \int_N | \langle A, d\psi \rangle_g  u  | dV
+ h \int_N |\langle A, d\tilde \varphi\rangle_g u\psi | dV
\le \mathcal{O}(h)\|u\|_{H^{1}_{\text{scl}}(N)}\|\psi\|_{H^{1}_{\text{scl}}(N)}.
\end{aligned}
\end{equation}

We also have
\begin{equation}
\label{eq_Car_18}
\begin{aligned}
|\langle e^{\frac{\tilde \varphi}{h}} h^2\langle  A, d(e^{-\frac{\tilde \varphi}{h}} u) \rangle_g, \psi\rangle_N|&
\le h^2\int_N | \langle A,d u\rangle_g \psi| dV +h\int_N |\langle A,d \tilde \varphi \rangle_g u \psi |dV 
\\
& \le \mathcal{O}(h)\|u\|_{H^{1}_{\text{scl}}(N)}\|\psi\|_{H^{1}_{\text{scl}}(N)},
\end{aligned}
\end{equation}
and 
\begin{equation}
\label{eq_Car_19}
|\langle e^{\frac{\tilde \varphi}{h}} h^2(\langle A, A\rangle_g^2+q) e^{-\frac{\tilde \varphi}{h}}u  , \psi\rangle_N|\le \mathcal{O}(h^2)\|u\|_{H^{1}_{\text{scl}}(N)}\|\psi\|_{H^{1}_{\text{scl}}(N)},
\end{equation}
Using \eqref{eq_Car_16}, it follows from \eqref{eq_Car_17}, \eqref{eq_Car_18}, and \eqref{eq_Car_19}, that 
\begin{equation}
\label{eq_Car_19_new}
\begin{aligned}
\| e^{\frac{\tilde \varphi}{h}} h^2 i d^*(A e^{-\frac{\tilde\varphi}{h}}u) -i e^{\frac{\tilde \varphi}{h}} h^2 (A, d(e^{-\frac{\tilde \varphi}{h}} u )) + e^{\frac{\tilde \varphi}{h}} h^2(\langle A, A\rangle_g^2 +q) e^{-\frac{\tilde \varphi}{h}}u \|_{H^{-1}_{\text{scl}}(N)}\\
\le \mathcal{O}(h)\|u\|_{H^{1}_{\text{scl}}(N)}. 
\end{aligned}
\end{equation}
Choosing $\varepsilon>0$ sufficiently small but fixed, i.e. independent of $h$, we obtain from \eqref{eq_Car_for_laplace} with $s=-1$ and \eqref{eq_Car_19_new} that for all $h>0$ small enough and $u\in C_0^\infty(M^0)$, 
\[
\|e^{\frac{\tilde \varphi }{h}}(h^2L_{g,A,q})e^{-\frac{\tilde \varphi}{h}}  u\|_{H^{-1}_{\text{scl}}(N)}\ge \frac{h}{C}\|u\|_{H^{1}_{\text{scl}}(N)},
\]
which implies \eqref{eq_Car_for_magnetic}.
\end{proof}

To construct complex geometric optics solutions for the magnetic Schr\"odinger operator $L_{g,A,q}$ we need to convert the Carleman estimate \eqref{eq_Car_for_magnetic} for the adjoint $L_{g,\overline{A},\overline{q}}$ into the following solvability result. Here we also use the fact that if $\varphi$ is a limiting Carleman weight then so is $-\varphi$ and we make use of the following semiclassical Sobolev norms on $M^0$, 
\[
\|u\|^2_{H^1_{\text{scl}}(M^0)}=\|u\|^2_{L^2(M^0)}+\|h\nabla_g u\|^2_{L^2(M^0)},
\]
and
\[
\|u\|_{H^{-1}_{\text{scl}}(M^0)}=\sup_{0\ne \psi\in C^\infty_0(M^0)}\frac{|\langle u, \psi\rangle_{M^0} |}{\|\psi\|_{H^1_{\text{scl}}(M^0)}}.
\]

\begin{prop}
\label{prop_solvability}
Let  $A\in L^\infty(M, T^*M)$, $q\in L^\infty(M,\C)$, and let  $\varphi$ be a limiting Carleman weight for $-h^2\Delta$ on $(U,g)$. If  $h>0$ is small enough, then for any $v\in H^{-1}(M^0)$, there is a solution $u\in H^1(M^0)$ of the equation 
\[
e^{\frac{\varphi}{h}}(h^2L_{g,A,q})e^{-\frac{\varphi}{h}}  u=v \quad \text{in}\quad M^0,
\]
which satisfies 
\[
\|u\|_{H^1_{\emph{\text{scl}}}(M^0)}\le \frac{C}{h} \|v\|_{H^{-1}_{\emph{\text{scl}}}(M^0)}.
\]
\end{prop} 
The proof of this result is well-known and we refer to \cite{DKSaloU_2009} and \cite{Krup_Uhlmann_magnet} for the details.

\section{Complex geometric optics solutions in admissible geometries}

\label{sec_CGO_admissible}

Let $(M,g)$ be an admissible manifold. Then we know that $(M,g)$ is isometrically  embedded in $(\R\times M_0, c(e\oplus g_0))$ for some compact simple $(n-1)$--dimensional manifold $(M_0,g_0)$ and some $0<c\in C^\infty(\R\times M_0)$. Assume, 
after replacing $M_0$ by a slightly larger simple manifold if needed, that for some simple manifold $(D,g_0)\subset\subset (M_0^0,g_0)$ one has  
\begin{equation}
\label{eq_cgo_mnfl_0_1}
(M,g)\subset\subset (\R\times D^0, c(e\oplus g_0))\subset\subset (\R\times M_0^0, c(e\oplus g_0)). 
\end{equation}

Let $A\in L^\infty(M,  T^*M)$ and $q\in L^\infty(M,\C)$. It will be convenient to extend $A$ and $q$ to all of $\R\times M_0^0$ by taking the zero extension. We shall denote the extensions of $A$ and $q$ by the same letters.  Then $A\in L^\infty(\R\times M_0^0, T^*(\R\times M_0^0))$ is compactly supported, and using a partition of unity argument together with a regularization in each coordinate patch, we get the following result. 

\begin{prop}
\label{prop_approximation}
There exists a family  $A_\tau\in C^\infty_0(\R\times M_0^0, T^*(\R\times M_0^0))$, $\tau>0$,  such that 
\begin{equation}
\label{eq_Car_20_0}
\|A-A_\tau\|_{L^2}=o(1),\quad \tau\to 0,
\end{equation}
and 
\begin{equation}
\label{eq_Car_20}
\begin{aligned}
\|A_\tau\|_{L^\infty}=\mathcal{O}(1), \quad
\|\nabla A_\tau\|_{L^\infty}=\mathcal{O}(\tau^{-1}),\quad
\|\Delta A_\tau\|_{L^\infty}=\mathcal{O}(\tau^{-2}),\quad \tau\to 0.  
\end{aligned}
\end{equation}
\end{prop}

We have global coordinates $ x=(x_1,x')$ on $\R\times M_0$ in which the metric $g$ has the form
\begin{equation}
\label{eq_cgo_metric_0}
g(x)=c(x)\begin{pmatrix}
1& 0\\
0 & g_0(x')
\end{pmatrix},
\end{equation} 
where $c>0$ and $g_0$ is a simple metric on $M_0$. Then the globally defined function 
\begin{equation}
\label{eq_cgo_metric_0_1}
\varphi(x)=x_1
\end{equation}
 on $M$
is a limiting Carleman weight. 

We shall construct solutions to 
\begin{equation}
\label{eq_cgo_1}
L_{g,A,q}u=0\quad\text{in}\quad M^0
\end{equation}
 of the form
\begin{equation}
\label{eq_cgo_2}
u=e^{-\frac{\rho}{h}}(a+r_0),
\end{equation}
where $\rho=\varphi+i\psi$ is a complex phase, $\psi\in C^\infty(M,\R)$, $a\in C^\infty(M,\C)$ is an amplitude, obtained by a WKB construction and $r_0$ is a remainder term.  We get 
\[
e^{\frac{\rho}{h}} \circ (-h^2\Delta) \circ e^{-\frac{\rho}{h}}=-h^2\Delta +h\Delta \rho+2h \nabla\rho -|\nabla \rho|^2,
\]
where $\nabla\rho$ is a complex vector field, and $|\nabla \rho|^2=\langle \nabla \rho,\nabla \rho\rangle $ is computed using the bilinear extension of the Riemannian scalar product to the complexified tangent bundle.  
We also have
\[
e^{\frac{\rho}{h}} ih^2d^*(A  e^{-\frac{\rho}{h}}u)= i h^2 d^*(Au)+ih\langle A,d\rho\rangle_g u,
\]
and
\[
- i h^2 e^{\frac{\rho}{h}}  \langle A, d(e^{-\frac{\rho}{h}} u)\rangle_g= -i h^2 \langle A, du\rangle_g +i h\langle A,d\rho\rangle_g u, 
\]
and therefore, 
\begin{align*}
e^{\frac{\rho}{h}} h^2L_{g,A,q} &(e^{-\frac{\rho}{h}}u)=-h^2\Delta u  +h(\Delta \rho) u+2h \nabla\rho(u) -|\nabla \rho|^2 u \\
&-ih^2 \langle A,du\rangle_g+2ih \langle A,d\rho \rangle_g u+  i h^2 d^*(Au) +h^2( \langle A, A \rangle_g+q)u.
\end{align*}

In order that \eqref{eq_cgo_2} be a solution to \eqref{eq_cgo_1}, following the WKB method, we require that the complex phase $\rho$ satisfies the eikonal equation,  
\begin{equation}
\label{eq_cgo_3}
|\nabla \rho|^2=0,
\end{equation}
and the amplitude $a$ satisfies  the regularized transport equation,
\begin{equation}
\label{eq_cgo_4}
(2i \langle A_\tau ,d\rho\rangle_g  +2 \nabla\rho) a+  (\Delta \rho) a=0. 
\end{equation}
The remainder term $r_0$ will be then determined by solving the equation, 
\begin{equation}
\label{eq_cgo_5}
\begin{aligned}
e^{\frac{\rho}{h}} h^2L_{g,A,q} (e^{-\frac{\rho}{h}}r_0)=-( &-h^2\Delta a -ih^2 \langle A,da \rangle_g+2ih \langle A-A_\tau,d\rho\rangle_g a+  i h^2 d^*(Aa)\\
& +h^2(\langle A, A \rangle_g+q)a). 
\end{aligned}
\end{equation}

As $\varphi$ is given in \eqref{eq_cgo_metric_0_1}, the eikonal equation \eqref{eq_cgo_3} becomes a pair of equations for $\psi$,
\begin{equation}
\label{eq_cgo_6}
|\nabla \psi|^2=|\nabla \varphi|^2, \quad \langle \nabla \varphi, \nabla \psi\rangle =0. 
\end{equation}
Using \eqref{eq_cgo_metric_0} and \eqref{eq_cgo_metric_0_1}, we get 
\begin{equation}
\label{eq_cgo_7}
\nabla \varphi =\frac{1}{c}\p_{x_1}, \quad |\nabla \varphi|^2=\frac{1}{c}. 
\end{equation}
It follows from \eqref{eq_cgo_6} and \eqref{eq_cgo_7} that 
\begin{equation}
\label{eq_cgo_8}
|\nabla \psi|^2=\frac{1}{c}, \quad \p_{x_1}\psi=0. 
\end{equation}

Let $\omega\in D$ be a point such that $(x_1,\omega)\notin M$ for all $x_1\in \R$.  We have the global coordinates on $M$ given by $x=(x_1,r,\theta)$, where $(r,\theta)$ are the polar normal coordinates in $(D,g_0)$ with center $\omega$, i.e. $x'=\exp_\omega^D(r\theta)$ where $r>0$ and $\theta\in \mathbb{S}^{n-2}$. Here $\exp_\omega^D$ is the exponential map which takes its maximal domain in $T_\omega D$ diffeomorphically onto $D$, since $D$ is simple.   In these coordinates  the metric $g$ has the form,
\begin{equation}
\label{eq_cgo_8_metrix}
g(x_1,r,\theta)=c(x_1,r,\theta) \begin{pmatrix}
1& 0& 0\\
0& 1& 0\\
0& 0& m(r,\theta)
\end{pmatrix},
\end{equation}
where $m$ is a smooth positive definite matrix. Hence, the eikonal equation \eqref{eq_cgo_8} has a global solution 
\[
\psi(x)=\psi_\omega(x)=r. 
\]
Hence, 
\[
\rho=x_1+i r,
\]
 and therefore, the vector field 
 \[
 \nabla \rho=\frac{2}{c}\overline{\p},
 \]
where 
\[
\overline{\p}=\frac{1}{2}(\p_{x_1}+i\p_{r}). 
\]
We also get 
\[
\Delta \rho=|g|^{-1/2}\p_{x_1} \bigg(\frac{|g|^{1/2}}{c}\bigg)+i |g|^{-1/2}\p_{r} \bigg(\frac{|g|^{1/2}}{c}\bigg)=\frac{1}{c}\overline{\p} \log\bigg(\frac{|g|}{c^2}\bigg).
\]
Thus, the transport equation \eqref{eq_cgo_4} has the form,
\begin{equation}
\label{eq_cgo_9}
4\overline{\p }a + \bigg(\overline{\p} \log\bigg(\frac{|g|}{c^2}\bigg)\bigg) a+2i ((A_\tau)_1 +i (A_\tau)_r) a=0. 
\end{equation}
Following \cite{DKSaloU_2009}, we choose a solution of \eqref{eq_cgo_9} in the form, 
\[
a=|g|^{-1/4} c^{1/2} e^{i\Phi_\tau} a_0(x_1,r)b(\theta),
\]
where $\Phi_\tau$ is a solution of  
\begin{equation}
\label{eq_cgo_10}
\overline{\p}\Phi_\tau =-\frac{1}{2}((A_\tau)_1 +i (A_\tau)_r),
\end{equation}
$a_0$ is  a non-vanishing holomorphic function, 
\[
\overline{\p}a_0=0,
\]
and $b(\theta)$ is smooth.  The equation \eqref{eq_cgo_10} is given in the global coordinates $(x_1,r)$, and using the standard fundamental solution  $1/ (\pi (x_1+i r))\in L^1_{\text{loc}}(\R^2)$ of  the $\overline{\p}$ operator, we can take
\begin{equation}
\label{eq_cgo_11}
\Phi_\tau(x_1,r,\theta) = -\frac{1}{2} \frac{1}{\pi (x_1+i r)} * ((A_\tau)_1 +i (A_\tau)_r),
\end{equation}
with $*$ denoting the convolution in the variables $(x_1,r)$ and $A_\tau(\cdot,\cdot,\theta)$ being viewed as a compactly supported smooth $1$-form in the complex $x_1+i r$ plane.

In  \eqref{eq_cgo_11} we are interested the solution $\Phi_\tau$ when $x_1$, $r$ vary in a bounded region and therefore, using \eqref{eq_Car_20}  we see that  
\begin{equation}
\label{eq_cgo_12}
\begin{aligned}
\|\Phi_\tau\|_{L^\infty(M)}=\mathcal{O}(1), \quad
\|\nabla \Phi_\tau\|_{L^\infty(M)}=\mathcal{O}(\tau^{-1}),\\ 
\|\Delta \Phi_\tau\|_{L^\infty(M)}=\mathcal{O}(\tau^{-2}),\quad \tau\to 0.  
\end{aligned}
\end{equation}
Setting 
\[
\Phi(x_1,r,\theta) = -\frac{1}{2} \frac{1}{\pi (x_1+i r)} * (A_1 +i A_r)\in L^\infty(M),
\]
using Young's inequality, \eqref{eq_cgo_11} and \eqref{eq_Car_20_0}, we obtain that 
\begin{equation}
\label{eq_cgo_13}
\|\Phi-\Phi_\tau\|_{L^2(M)}=o(1), \quad \tau\to 0. 
\end{equation}

Finally, we shall solve the equation \eqref{eq_cgo_5} for the remainder term $r_0$. First notice that the right hand side of  \eqref{eq_cgo_5},
\begin{equation}
\label{eq_cgo_14_v}
\begin{aligned}
v=-( &-h^2\Delta a -ih^2 \langle A,da \rangle_g+2ih \langle A-A_\tau,d\rho\rangle_g a+  i h^2 d^*(Aa)\\
& +h^2(\langle A, A \rangle_g+q)a)\in H^{-1}(M^0), 
\end{aligned}
\end{equation}
and we shall estimate $\|v\|_{H^{-1}_{\text{scl}}(M^0)}$. To that end, let $0\ne \phi\in C^\infty_0(M^0)$. Then using \eqref{eq_cgo_12} we get 
\begin{equation}
\label{eq_cgo_14}
\begin{aligned}
&|\langle h^2\Delta a, \phi\rangle_{M^0}|\le \mathcal{O}(h^2/\tau^2)\|\phi\|_{L^2(M^0)}\le \mathcal{O}(h^2/\tau^2)\|\phi\|_{H^1_{\text{scl}}(M^0)},\\
&|\langle h^2  \langle A,da\rangle_g, \phi\rangle_{M^0}|\le  \mathcal{O}(h^2/\tau)\|\phi\|_{L^2(M^0)}\le  \mathcal{O}(h^2/\tau)\|\phi\|_{H^1_{\text{scl}}(M^0)}. 
\end{aligned}
\end{equation}
Using \eqref{eq_cgo_12}, \eqref{eq_Car_20_0}, and the Cauchy--Schwarz inequality,  we obtain that 
\begin{equation}
\label{eq_cgo_15}
\begin{aligned}
|\langle h  \langle  A-A_\tau,d\rho\rangle_g  a, \phi\rangle_{M^0}|
 &\le \mathcal{O}(h)\|a\|_{L^\infty(M)}\|A-A_\tau\|_{L^2(M^0)}\|\phi\|_{L^2(M^0)}\\
 &\le  \mathcal{O}(h)o_{\tau\to 0}(1) \|\phi\|_{H^1_{\text{scl}}(M^0)}. 
\end{aligned}
\end{equation}
By \eqref{eq_int_2_dis}, \eqref{eq_cgo_12}, \eqref{eq_Car_20_0}, we get 
\begin{equation}
\label{eq_cgo_16}
\begin{aligned}
|\langle  h^2  d^*(Aa), \phi\rangle_{M^0}|&=\bigg| h^2 \int \langle Aa, d\phi \rangle_g dV \bigg|\\
& \le h^2\int | \langle (A-A_\tau) a, d\phi \rangle_g| dV+h^2 \bigg| \int d^*(A_\tau a)\phi dV\bigg|\\
&\le (\mathcal{O}(h)o_{\tau\to 0}(1) + \mathcal{O}(h^2/\tau)) \|\phi\|_{H^1_{\text{scl}}(M^0)}.
\end{aligned}
\end{equation}
Finally, we have 
\begin{equation}
\label{eq_cgo_17}
\|h^2(\langle  A, A\rangle_g ^2+q)a\|_{L^2(M)}\le \mathcal{O}(h^2). 
\end{equation}
It follows from  \eqref{eq_cgo_14_v}, \eqref{eq_cgo_14},  \eqref{eq_cgo_15},  \eqref{eq_cgo_16} and  \eqref{eq_cgo_17} that 
\[
\|v\|_{H^{-1}_{\text{scl}}(M^0)}\le \mathcal{O}(h^2/\tau^2)+ \mathcal{O}(h)o_{\tau\to 0}(1), 
\]
and choosing $\tau=h^\sigma$ with $\sigma$, $0<\sigma<1/2$, we get 
\begin{equation}
\label{eq_cgo_18}
\|v\|_{H^{-1}_{\text{scl}}(M^0)}=o(h), \quad h\to 0. 
\end{equation}

Thus, by Proposition \ref{prop_solvability} and \eqref{eq_cgo_18}, for all $h>0$ small enough, there exists a solution $r_0\in H^1(M^0)$ of    \eqref{eq_cgo_5} which satisfies $\|r_0\|_{H^{1}_{\text{scl}}(M^0)}=o(1)$ as $h\to 0$. 

The discussion in this section can be summarized in the following proposition.
\begin{prop}
\label{prop_cgo-admiss}
Assume that $(M,g)$ satisfies \eqref{eq_cgo_mnfl_0_1} and \eqref{eq_cgo_metric_0}, and let $A\in L^\infty(M, T^*M)$, $q\in L^\infty(M,\C)$.  Let $\omega\in D$ be such that $(x_1,\omega)\notin M$ for all $x_1$, and let $(r,\theta)$ be the polar normal coordinates in $(D,g_0)$ with center $\omega$. Then for all $h>0$ small enough, there exists a solution $u\in H^1(M^0)$ to the magnetic Schr\"odinger equation 
\[
L_{g,A,q} u=0\quad \text{in}\quad \mathcal{D}'(M^0)
\]
of the form
\[
u=e^{-\frac{1}{h}(x_1+ir )}(|g|^{-1/4} c^{1/2} e^{i\Phi_h} a_0(x_1,r)b(\theta) +r_0 ),
\]
where $a_0$ is  a non-vanishing holomorphic function, $(\p_{x_1}+i\p_r)a_0=0$, and $b(\theta)$ is smooth. The function
$\Phi_h\in C^\infty(M)$ satisfies  
\begin{equation}
\label{eq_cgo_19}
\begin{aligned}
\|\Phi_h\|_{L^\infty(M)}=\mathcal{O}(1), \quad
\|\nabla \Phi_h\|_{L^\infty(M)}=\mathcal{O}(h^{-\sigma}),\\ 
\|\Delta \Phi_h\|_{L^\infty(M)}=\mathcal{O}(h^{-2\sigma}),\quad h\to 0, \quad 0<\sigma<1/2,
\end{aligned}
\end{equation}
and    
\begin{equation}
\label{eq_cgo_20}
\|\Phi-\Phi_h\|_{L^2(M)}=o(1), \quad h\to 0, 
\end{equation}
where 
\[
\Phi(x_1,r,\theta) = -\frac{1}{2} \frac{1}{\pi (x_1+i r)} * (A_1 +i A_r),
\]
where $A=A_1 dx_1+A_rdr+A_\theta d\theta$.  The remainder $r_0$ is such that $\|r_0\|_{H^{1}_{\emph{\text{scl}}}(M^0)}=o(1)$ as $h\to 0$. 
\end{prop}

\section{Proof of Theorem \ref{thm_main}}

\label{sec_proof_thm_main_1}

Let us start by recalling some auxiliary, essentially well-known, results, see \cite{DKSaloU_2009}, \cite{Krup_Uhlmann_magnet}. 

\begin{lem}
\label{lem_gauge}
Let $A\in L^\infty(M,T^*M)$, $q\in L^\infty(M,\C)$, and let $\phi\in W^{1,\infty}(M^0)$. Then we have 
\begin{equation}
\label{eq_proof_gauge_1}
e^{-i\phi}\circ L_{g,A,q} \circ e^{i\phi}=L_{g, A+d\phi,q}. 
\end{equation}
If furthermore $\phi|_{\p M}=0$ then 
\begin{equation}
\label{eq_proof_gauge_2}
C_{g,A,q}=C_{g,A+d\phi,q}. 
\end{equation}
\end{lem} 
\begin{proof}
A direct computation using \eqref{eq_int_1} shows \eqref{eq_proof_gauge_1}.  In order to see \eqref{eq_proof_gauge_2}, let $u\in H^1(M^0)$ be a solution to $L_{g,A,q}u=0$ in $\mathcal{D}'(M^0)$. Then it follows from \eqref{eq_proof_gauge_1} that  $e^{-i\phi}u\in H^1(M^0)$ satisfies $L_{g,A+d\phi,q} (e^{-i\phi}u)=0$ in $\mathcal{D}'(M^0)$. We have $e^{-i\phi}u|_{\p M}=u|_{\p M}$. Using  \eqref{eq_trace_normal_int}, for $f\in H^{1/2}(\p M)$, 
we get 
\begin{align*}
&\langle  \langle d_{A+d\phi}(e^{-i\phi}u) ,\nu \rangle_g  , f\rangle_{H^{-\frac{1}{2}}(\p M)\times H^{\frac{1}{2}}(\p M)}\\
&= \langle  \langle d_{A+d\phi}(e^{-i\phi}u) ,\nu \rangle_g  , e^{i\phi} f\rangle_{H^{-\frac{1}{2}}(\p M)\times H^{\frac{1}{2}}(\p M)}
=\langle  \langle d_{A}u ,\nu \rangle_g  , f\rangle_{H^{-\frac{1}{2}}(\p M)\times H^{\frac{1}{2}}(\p M)},
\end{align*}
which completes the proof. 
\end{proof}

The following result is proved in exactly the same way as  \cite[Proposition 3.4]{Krup_Uhlmann_magnet}. 

\begin{lem}
\label{lem_Cauchy_ext}
Let $(M,g)$ and $(\tilde M, g)$ be smooth compact Riemannian manifolds with smooth boundaries such that $M\subset \tilde M^0$.  
Let $A^{(1)}, A^{(2)}\in L^\infty(\tilde M,T^*\tilde M)$ and $q^{(1)},q^{(2)}\in L^\infty(\tilde M,\C)$. Assume that 
\[
A^{(1)}=A^{(2)}\quad \text{and}\quad q^{(1)}=q^{(2)}\quad \text{in}\quad \tilde M\setminus M.
\]
If $C_{g, A^{(1)},q^{(1)}}=C_{g, A^{(2)},q^{(2)}}$ then $C'_{g, A^{(1)},q^{(1)}}=C'_{g, A^{(2)},q^{(2)}}$, where $C'_{g, A^{(j)},q^{(j)}}$ is the set of the Cauchy data for $L_{g,A^{(j)},q^{(j)}}$ on $\tilde M$, $j=1,2$. 
\end{lem}

Finally, we shall also need  the following standard integral identity, see  \cite{DKSaloU_2009}, \cite{Krup_Uhlmann_magnet}. 

\begin{prop}
\label{prop_integral_identify}
Let $A^{(1)},A^{(2)}\in L^\infty(M,T^*M)$, $q^{(1)},q^{(2)}\in L^\infty(M,\C)$.  If $C_{g,A^{(1)},q^{(1)}}=C_{g,A^{(2)},q^{(2)}}$ then 
\begin{equation}
\label{eq_proof_1}
\begin{aligned}
\int_M i  \langle A^{(1)}-A^{(2)},& u_1d\overline{u_2}-\overline{u_2}du_1 \rangle_g dV_g \\
&+ \int_M  ( \langle A^{(1)}, A^{(1)}\rangle_g - \langle A^{(2)}, A^{(2)}\rangle_g +q^{(1)}-q^{(2)})u_1\overline{u_2}dV=0, 
\end{aligned}
\end{equation}
for any $u_1,u_2\in H^1(M^0)$ satisfying $L_{g,A^{(1)},q^{(1)}}u_1=0$ and $L_{g,\overline{A^{(2)}},\overline{q^{(2)}}}u_2=0$ in $\mathcal{D}'(M^0)$.  
\end{prop}

\begin{proof}
 As $C_{g,A^{(1)},q^{(1)}}=C_{g,A^{(2)},q^{(2)}}$, there is $v_2\in H^1(M^0)$ such that $L_{g, A^{(2)}, q^{(2)}} v_2=0$ in $\mathcal{D}'(M^0)$ and 
\[
u_1|_{\p M}=v_2|_{\p M} , \quad \langle d_{A^{(1)}} u_1,\nu\rangle_g|_{\p M}= \langle d_{A^{(2)}} v_2,\nu\rangle_g|_{\p M}.
\] 
Hence, 
\begin{equation}
\label{eq_int_trace_1}
\langle \langle d_{A^{(1)}} u_1,\nu\rangle_g, \overline{u_2}\rangle_{H^{-\frac{1}{2}}(\p M)\times H^{\frac{1}{2}}(\p M)}=
\langle \langle d_{A^{(2)}} v_2,\nu\rangle_g, \overline{u_2}\rangle_{H^{-\frac{1}{2}}(\p M)\times H^{\frac{1}{2}}(\p M)}.
\end{equation}
Now using the fact that $L_{q,-A^{(2)}, q^{(2)}} \overline{u_2}=0$ in $\mathcal{D}'(M^0)$ and \eqref{eq_trace_normal_int}, we get 
\begin{equation}
\label{eq_int_trace_2}
\begin{aligned}
\langle \langle d_{A^{(2)}} v_2,\nu\rangle_g, \overline{u_2}\rangle_{H^{-\frac{1}{2}}(\p M)\times H^{\frac{1}{2}}(\p M)}&=
\langle \langle d_{-A^{(2)}} \overline{u_2},\nu\rangle_g, v_2\rangle_{H^{-\frac{1}{2}}(\p M)\times H^{\frac{1}{2}}(\p M)}\\
&=
\langle \langle d_{-A^{(2)}} \overline{u_2},\nu\rangle_g, u_1\rangle_{H^{-\frac{1}{2}}(\p M)\times H^{\frac{1}{2}}(\p M)},
\end{aligned}
\end{equation}
where in the last equality we have used that $u_1=v_2$ on $\p M$.  It follows from \eqref{eq_int_trace_1}
and \eqref{eq_int_trace_2}
that 
\[
\langle \langle d_{A^{(1)}} u_1,\nu\rangle_g, \overline{u_2}\rangle_{H^{-\frac{1}{2}}(\p M)\times H^{\frac{1}{2}}(\p M)}=\langle \langle d_{-A^{(2)}} \overline{u_2},\nu\rangle_g, u_1\rangle_{H^{-\frac{1}{2}}(\p M)\times H^{\frac{1}{2}}(\p M)},
\]
which proves the claim in view of \eqref{eq_trace_normal_int}. 
\end{proof}

Let  $(\tilde M,g)$ be an admissible simply connected manifold with connected boundary such that    $(M,g)\subset (\tilde M^0,g)\subset  (\R\times M_0^0, c(e\oplus g_0))$.  Applying Lemma \ref{lem_Cauchy_ext} and using that  $A^{(1)}=A^{(2)}=0$ and $q^{(1)}=q^{(2)}=0$ outside of $M$, we may and will assume in what follows that the manifold $(M,g)$ is simply connected with connected boundary and the coefficients $A^{(j)}$ and $q^{(j)}$ are compactly supported in the interior of $M$. 

Let us now rewrite the integral identity \eqref{eq_proof_1} in the following form, 
\begin{equation}
\label{eq_proof_1_new_iden}
\begin{aligned}
\int_M i  \langle A^{(1)}-A^{(2)},& u_1d u_2-u_2 du_1 \rangle_g dV_g \\
&+ \int_M  ( \langle A^{(1)}, A^{(1)}\rangle_g - \langle A^{(2)}, A^{(2)}\rangle_g +q^{(1)}-q^{(2)})u_1 u_2dV=0, 
\end{aligned}
\end{equation}
for any $u_1,u_2\in H^1(M^0)$ satisfying $L_{g,A^{(1)},q^{(1)}}u_1=0$ and $L_{g,- A^{(2)}, q^{(2)}} u_2=0$ in $\mathcal{D}'(M^0)$.

By Proposition  \ref{prop_cgo-admiss} for all $h>0$ small enough, there are  solutions $u_1,u_2\in H^1(M^0)$ to the magnetic Schr\"odinger equations 
 $L_{g,A^{(1)},q^{(1)}}u_1=0$ and $L_{g,-A^{(2)},q^{(2)}}u_2=0$ in $\mathcal{D}'(M^0)$,
of the form
\begin{equation}
\label{eq_proof_2}
\begin{aligned}
u_1&=e^{-\frac{\rho}{h}}(|g|^{-1/4} c^{1/2} e^{i\Phi^{(1)}_{h}} a_0(x_1,r)b(\theta) +r_1 )=e^{-\frac{\rho}{h}} (\alpha_1+r_1),\\
u_2&=e^{\frac{\rho}{h}}(|g|^{-1/4} c^{1/2} e^{i\Phi^{(2)}_{h}}  +r_2)=e^{\frac{\rho}{h}}(\alpha_2  +r_2 ),
\end{aligned}
\end{equation}
respectively. Here $\rho=x_1+ir $, 
\begin{equation}
\label{eq_proof_3}
\alpha_1=|g|^{-1/4} c^{1/2} e^{i\Phi^{(1)}_{h}} a_0(x_1,r)b(\theta), \quad \alpha_2=|g|^{-1/4} c^{1/2} e^{i\Phi^{(2)}_{h}},
\end{equation}
and
\begin{equation}
\label{eq_proof_4}
\|r_j\|_{L^2(M^0)}=o(1), \quad \|dr_j\|_{L^2(M^0)}=o(h^{-1}), \quad h\to 0, \quad j=1,2.
\end{equation}
Furthermore, $\Phi^{(j)}_{h}\in C^\infty (M) $ satisfies  \eqref{eq_cgo_19}, and 
\begin{equation}
\label{eq_proof_4_0}
\|\Phi^{(j)}-\Phi^{(j)}_{h}\|_{L^2(M)}=o(1), \quad h\to 0, 
\end{equation}
where 
\begin{align*}
&\Phi^{(1)}(x_1,r,\theta) = -\frac{1}{2} \frac{1}{\pi (x_1+i r)} * (A^{(1)}_1 +i A^{(1)}_r),\\
&\Phi^{(2)}(x_1,r,\theta)= \frac{1}{2} \frac{1}{\pi (x_1+i r)} * (A^{(2)}_1 +i A^{(2)}_r),
\end{align*}
where $A^{(j)}=A^{(j)}_1 dx_1+A^{(j)}_rdr+A^{(j)}_\theta d\theta$, $j=1,2$.  Thus, 
\begin{equation}
\label{eq_proof_4_1}
\Phi=\Phi^{(1)}+\Phi^{(2)}\in L^\infty(M)
\end{equation}
 satisfies 
\begin{equation}
\label{eq_proof_4_2}
\overline{\p} \Phi+\frac{1}{2}(\tilde A_1+i\tilde A_r)=0\quad \text{in}\quad M,
\end{equation}
where $\tilde A=A^{(1)}-A^{(2)}$. 

It follows from \eqref{eq_cgo_19} that 
\begin{equation}
\label{eq_proof_5}
\|\alpha_j\|_{L^\infty(M)}=\mathcal{O}(1), \quad \|d \alpha_j\|_{L^\infty(M)}=\mathcal{O}(h^{-\sigma}), \quad h\to 0, \quad 0<\sigma<1/2.
\end{equation}

We shall next insert  $u_1$ and $u_2$, given by \eqref{eq_proof_2},  into \eqref{eq_proof_1_new_iden}, multiply it by $h$, and let $h\to 0$.  To that end, we first have
\begin{equation}
\label{eq_proof_6}
\begin{aligned}
u_1du_2-u_2du_1=&\frac{2d\rho}{h} (\alpha_1\alpha_2 + \alpha_1r_2+\alpha_2 r_1+r_1r_2\\
&+(\alpha_1+r_1)(d\alpha_2+dr_2)-(\alpha_2+r_2)(d\alpha_1+dr_1).
\end{aligned}
\end{equation}
We conclude from \eqref{eq_proof_1_new_iden} in view of \eqref{eq_proof_6}, \eqref{eq_proof_4} and \eqref{eq_proof_5} that 
\begin{equation}
\label{eq_proof_7}
\lim_{h\to 0} \int_M \langle \tilde A, d\rho \rangle_g \alpha_1\alpha_2 dV= \lim_{h\to 0} \int_M \langle \tilde A, d\rho \rangle_g |g|^{-1/2}ce^{i(\Phi^{(1)}_{h}+\Phi^{(2)}_{h})}a_0b dV=0. 
\end{equation}

We claim next that 
\begin{equation}
\label{eq_proof_8}
\lim_{h\to 0} \int_M \langle \tilde A, d\rho \rangle_g  |g|^{-1/2}ce^{i(\Phi^{(1)}_{h}+\Phi^{(2)}_{h})}a_0b dV =
 \int_M \langle \tilde A, d\rho \rangle_g |g|^{-1/2}ce^{i\Phi}a_0b dV,
\end{equation}
where $\Phi$ is given by \eqref{eq_proof_4_1}. This follows from the estimates,
\begin{align*}
\bigg|  \int_M \langle \tilde A, d\rho \rangle_g |g|^{-1/2}c(e^{i(\Phi^{(1)}_{h}+\Phi^{(2)}_{h})}-e^{i\Phi})a_0b dV \bigg|\le C\| e^{i(\Phi^{(1)}_{h}+\Phi^{(2)}_{h})}-e^{i\Phi} \|_{L^2(M)}\\
\le \| \Phi^{(1)}_{h}+\Phi^{(2)}_{h}-\Phi_1-\Phi_2 \|_{L^2(M)}\to 0, \quad h\to 0.
\end{align*}
Here we have used \eqref{eq_proof_4_0}, the inequality 
\[
|e^z-e^w|\le |z-w|e^{\max(\text{Re}\, z,\text{Re}\, w)}, \quad z,w\in \C,
\]
and the fact that $\Phi^{(j)}, \Phi^{(j)}_{h}\in L^\infty(M)$ and $\|\Phi^{(j)}_{h}\|_{L^\infty(M)}\le C$ uniformly in $h$.

Now \eqref{eq_proof_7} and \eqref{eq_proof_8} imply that 
\[
 \int_M \langle \tilde A, d\rho \rangle_g |g|^{-1/2}ce^{i\Phi}a_0b dV=0. 
\]
Writing out this integral in the global coordinates $(x_1,r,\theta)$, and using that $dV=|g|^{1/2}dx_1drd\theta$, and \eqref{eq_cgo_8_metrix},  we get
\[
\int_M (\tilde A_1+i\tilde A_r)e^{i\Phi}a_0(x_1,r)b(\theta)dx_1drd\theta=0. 
\]
Now since the function
\[
(x_1,r,\theta)\mapsto (\tilde A_1+i\tilde A_r)e^{i\Phi}a_0(x_1,r)\in L^1(M),
\]
by Fubini's theorem the function 
\[
(x_1,r)\mapsto (\tilde A_1+i\tilde A_r)e^{i\Phi}a_0(x_1,r)\in L^1
\] 
for almost all $\theta$, and 
\[
\theta\mapsto \int_{\Omega_\theta} (\tilde A_1+i\tilde A_r)e^{i\Phi}a_0(x_1,r)dx_1dr \in L^1
\]
where $\Omega_\theta=\{(x_1,r)\in \R^2: (x_1,r,\theta)\in M\}$. Since $b\in C^\infty(\mathbb{S}^{n-2})$ is arbitrary, we conclude that 
\begin{equation}
\label{eq_proof_9}
\int_{\Omega_\theta} (\tilde A_1+i\tilde A_r)e^{i\Phi}a_0(x_1,r)d\overline{\rho} \wedge d\rho=0,\quad \text{for a.a. }\theta. 
\end{equation}

In what follows we view $\Omega_\theta$ as a domain in the complex plane with the complex variable $\rho$, and by Sard's theorem, for almost all $\theta$ the boundary of $\Omega_\theta$ is $C^\infty$ smooth, see \cite{Krup_Lassas_Uhlmann_trans}. 
If follows from \eqref{eq_proof_9} and \eqref{eq_proof_4_2} that 
\begin{equation}
\label{eq_proof_10}
\int_{\Omega_\theta} \overline{\p} (e^{i\Phi}a_0) d\overline{\rho} \wedge d\rho=0, \quad \text{for a.a. }\theta. 
\end{equation}
We shall now discuss regularity properties of $\Phi(\cdot, \cdot, \theta)$.  In view of \eqref{eq_proof_4_2} and the fact that for a.a. $\theta$,  $\tilde A(\cdot, \cdot, \theta)\in L^\infty(\C)$ is compactly supported, we conclude that  $ \overline{\p} \Phi(\cdot, \cdot, \theta)\in L^p(\C)$ for $1\le p\le \infty$.  Using the boundedness of  the Beurling--Ahlfors operator $\p \overline{\p}^{-1}$  on $L^p(\C)$ for $1<p<\infty$, we get $\p \Phi=\p \overline{\p}^{-1} (\overline{\p} \Phi)\in L^p(\C)$, and therefore, $\nabla \Phi \in L^p(\C)$. Since $\Phi\in L^\infty(\C)$, we obtain that $\Phi(\cdot, \cdot, \theta)\in W^{1,p}_{\text{loc}}(\C)$, $1<p<\infty$. 
In particular, $\Phi(\cdot, \cdot, \theta)\in H^1(\Omega_\theta)$, and thus, $ e^{i\Phi(\cdot, \cdot, \theta)}\in  H^1(\Omega_\theta)$.  Hence, by Stokes' theorem, we conclude from \eqref{eq_proof_10} that 
\begin{equation}
\label{eq_proof_11}
\int_{\p \Omega_\theta} e^{i\Phi}a_0 d\rho=0. 
\end{equation}

Furthermore, as  $\Phi(\cdot, \cdot, \theta)\in W^{1,p}_{\text{loc}}(\C)$, $1<p<\infty$,  for a.a. $\theta$, by Sobolev's embedding $\Phi (\cdot, \cdot, \theta)$ is locally H\"older continuous of order $1-\delta$ for any $\delta>0$, see \cite[Theorem 4.5.12]{Hormander_book_1} and \cite[Lemma 3.4]{Sun_Uhlmann_1993}. 
Thanks to validity of  the Plemelj--Sokhotski--Privalov formula for H\"older continuous functions, see  \cite{McLean_1988}, the argument in \cite[Lemma 5.1]{DKSU_2007} shows that there exists a non-vanishing function $F\in C(\overline{\Omega_\theta})$, holomorphic in $\Omega_\theta$, such that 
\[
F|_{\p \Omega_\theta}=e^{i\Phi}|_{\p \Omega_\theta}.
\]
The arguments in \cite[Section 7]{Knudsen_Salo} show that $F$ admits a holomorphic logarithm $G\in C(\overline{\Omega_\theta})$, i.e.  $F=e^G$, and furthermore,  $(G-i\Phi)|_{\p \Omega_\theta}\in 2\pi i \Z$. Choosing 
\[
a_0=G e^{-G} e^{i\lambda(x_1+ir)},
\]
with $\lambda\in \R$ in \eqref{eq_proof_11}, we get 
\[
\int_{\p \Omega_\theta} i\Phi  e^{i\lambda(x_1+ir)}  d\rho=0. 
\]
By Stokes's theorem and \eqref{eq_proof_4_2}, we obtain that 
\begin{equation}
\label{eq_proof_12}
\int\!\!\!\int_{\Omega_\theta} (\tilde A_1+i\tilde A_r) e^{i\lambda(x_1+ir)} dx_1dr=0, \quad \text{for a.a. }\theta. 
\end{equation}
Multiplying \eqref{eq_proof_12} by an arbitrary function $b\in C^\infty(\mathbb{S}^{n-2})$ and integrating with respect to $\theta$, we get 
\begin{equation}
\label{eq_proof_13}
\int\!\!\!\int\!\!\!\int_{\R\times D} (\tilde A_1+i\tilde A_r) e^{i\lambda(x_1+ir)} b(\theta) dx_1drd\theta=0. 
\end{equation}
Here we take $\omega\in \p D$ to define the Riemannian polar normal coordinates $(r,\theta)$. 

Set
\begin{equation}
\label{eq_proof_14_-1}
f(x')=\int e^{i\lambda x_1}\tilde A_1(x_1,x')dx_1,\quad x'\in D.
\end{equation}
We have $f\in L^\infty(D)$.  
Set also 
\begin{equation}
\label{eq_proof_14_0}
\alpha=\sum_{j=2}^n  \bigg(\int e^{i\lambda x_1} \tilde A_j(x_1,x')dx_1\bigg)dx_j.
\end{equation}
We have $\alpha\in L^\infty(D, T^*D)$.  Since $(r,\theta)$ are Riemannian polar coordinates in $D$ with center $\omega$, we know that the curve $\gamma_\theta: r\mapsto (r,\theta)$
is the unit speed geodesic in $D$ emanating from $\omega$ in the direction $\theta$. Notice that when $\tilde A$ is smooth, we have $\alpha_r=\alpha(\p_r)=\alpha(\dot{\gamma}_\theta(r))$.  In our case,  \eqref{eq_proof_13} can be written as 
\begin{equation}
\label{eq_proof_14}
\int_{\mathbb{S}^{n-2}}\int_0^{\tau(\omega, \theta)}  [f(\gamma_\theta(r))+i \alpha (\dot{\gamma}_\theta(r))] e^{-\lambda r} b(\theta)dr d\theta=0, 
\end{equation}
where $\tau(\omega, \theta)$  is the time when  the geodesic $\gamma_\theta$ exits $D$.  The integral in \eqref{eq_proof_14} is related to the attenuated geodesic ray transform acting on the function $f$ and $1$-form $i\alpha$ in $D$ with constant attenuation $-\lambda$. In order to proceed we shall need the following result from \cite[Proposition 5.1]{Assylbekov_Yang_2017} dealing with  the injectivity of this transform, which is 
an extension of \cite[Lemma 5.1]{DKSalo_2013} established in the case when $\alpha=0$.

\begin{prop}
\label{prop_Assylbekov}
Let $(D,g_0)$ be an $(n-1)$-dimensional simple manifold. Let $f\in L^\infty(D)$ and $\alpha\in L^\infty(D, T^*D)$ be a $1$-form.  Consider the integrals
\[
\int_{\mathbb{S}^{n-2}}\int_0^{\tau(\omega, \theta)}  [f(\gamma_\theta(r))+ \alpha (\dot{\gamma}_\theta(r))] e^{-\lambda r} b(\theta)dr d\theta,
\] 
where $(r,\theta)$ are polar normal coordinates in $(D,g_0)$ centered at some $\omega\in \p D$ and $\tau(\omega, \theta)$ is the time when the geodesic $\gamma_\theta:  r\mapsto (r, \theta)$ exits $D$.  If  $|\lambda|$ is sufficiently small, and if these integrals vanish  for all $\omega\in \p D$ and all $b\in C^\infty(\mathbb{S}^{n-2})$, then there is $p\in W^{1,\infty}(D^0)$ with $p|_{\p D}=0$ such that $f=-\lambda p$ and $\alpha=dp$. 
\end{prop}

Now varying the point $\omega\in \p D$ in the construction of the complex geometric optics solution in Proposition \ref{prop_cgo-admiss} and applying  Proposition \ref{prop_Assylbekov},  for all $\lambda$ small enough, we have 
\begin{equation}
\label{eq_proof_15_0}
f=-\lambda p
\end{equation}
 and 
\begin{equation}
\label{eq_proof_15}
\alpha=-i dp
\end{equation}
 where $p\in W^{1,\infty}(D^0)$ with $p|_{\p D}=0$.  Using \eqref{eq_proof_14_0},  \eqref{eq_proof_15},  we get 
\begin{equation}
\label{eq_proof_15_exp}
\int e^{i\lambda x_1} \tilde A_j(x_1,x')dx_1=-i \p_{x_j} p, 
\end{equation}
for all $|\lambda|$ small enough. Viewing $\tilde A_j\in (L^\infty\cap \mathcal{E}')(\R_{x_1}, L^2(D))$, we see that $\mathcal{F}^{-1}_{x_1\to \lambda}(\tilde A(x_1,x'))\in \text{Hol}(\C, L^2(D))$. It follows from \eqref{eq_proof_15_exp} and the analyticity of the Fourier transform that 
\[
\mathcal{F}^{-1}_{x_1\to \lambda} \big(\p_{x_k}\tilde A_j -\p_{x_j}\tilde A_k\big)=0, \quad\text{for all } \lambda\in \C,\quad  j,k=2,\dots, n,
\]
where $\p_{x_k}\tilde A_j -\p_{x_j}\tilde A_k\in (L^\infty\cap \mathcal{E}')(\R_{x_1}, H^{-1}(D))$. 
We conclude that 
\[
\p_{x_k}\tilde A_j -\p_{x_j}\tilde A_k=0, \quad j,k=2,\dots, n. 
\]

By \eqref{eq_proof_15_0},  \eqref{eq_proof_15}, \eqref{eq_proof_14_-1}, \eqref{eq_proof_14_0},  we have
\begin{align*}
0=\p_{x_j} f+i\lambda \alpha_j&= 2\pi \mathcal{F}^{-1}_{x_1\to \lambda} (\p_{x_j} \tilde A_1)+ 2\pi  i\lambda  \mathcal{F}^{-1}_{x_1\to \lambda} (\tilde A_j)\\
&= 2\pi \mathcal{F}^{-1}_{x_1\to \lambda} (\p_{x_j} \tilde A_1 -\p_{x_1}\tilde A_j), \quad j=2,\dots, n,
\end{align*}
and therefore, $\p_{x_j}\tilde A_1-\p_{x_1} \tilde A_j=0$. Hence, $d\tilde A=0$ in $M$, and thus, $dA^{(1)}=dA^{(2)}$ in $M$.

Our next goal is to show that $q^{(1)}=q^{(2)}$ in $M$. 
Since $M$ is simply connected, by the Poincar\'e lemma for currents, see \cite{de_Rham},  we conclude that  there is $\phi\in \mathcal{D}'(M)$ such that $d\phi =A^{(1)}-A^{(2)}\in L^\infty(M, T^*M) \cap \mathcal{E}'(M^0, T^*M^0)$. It follows from \cite[Theorem 4.5.11]{Hormander_book_1} that $\phi$ is continuous and $\phi$ is a constant $c$ near $\p M$.  Therefore, $\phi\in W^{1,\infty}(M^0)$, and since the boundary $\p M$ is connected by considering $\phi-c$, we may assume that $\phi=0$ on $\p M$. 

By Lemma \ref{lem_gauge}, we have   $C_{g, A^{(1)},q^{(1)}}= C_{g, A^{(2)},q^{(2)}}=C_{g, A^{(2)}+d\phi,q^{(2)}}=C_{g, A^{(1)},q^{(2)}}$.  We may assume therefore that $A^{(1)}=A^{(2)}$ and we will denote this $1$-form by $A$. 
The integral identity   \eqref{eq_proof_1_new_iden} becomes
\begin{equation}
\label{eq_proof_1_new_iden_q}
\int_M  ( q^{(1)}-q^{(2)})u_1 u_2dV=0, 
\end{equation}
for any $u_1,u_2\in H^1(M^0)$ satisfying $L_{g,A,q^{(1)}}u_1=0$ and $L_{g,- A, q^{(2)}} u_2=0$ in $\mathcal{D}'(M^0)$.

By Proposition  \ref{prop_cgo-admiss} for all $h>0$ small enough, there are  solutions $u_1,u_2\in H^1(M^0)$ to the magnetic Schr\"odinger equations 
 $L_{g,A,q^{(1)}}u_1=0$ and $L_{g,-A,q^{(2)}}u_2=0$ in $\mathcal{D}'(M^0)$,
of the form
\begin{equation}
\label{eq_proof_2_q}
\begin{aligned}
u_1&=e^{-\frac{\rho}{h}}(|g|^{-1/4} c^{1/2} e^{i\Phi_{h}} e^{i\lambda(x_1+ir)}b(\theta) +r_1 ),\\
u_2&=e^{\frac{\rho}{h}}(|g|^{-1/4} c^{1/2} e^{-i\Phi_{h}}  +r_2),
\end{aligned}
\end{equation}
respectively. Here $\rho=x_1+ir $, $\lambda\in \R$, $b$ is smooth,   $\Phi_{h}\in C^\infty (M) $ satisfies  \eqref{eq_cgo_19}, and  \eqref{eq_cgo_20}, and $\|r_j\|_{H^1_{\text{scl}}(M^0)}=o(1)$ as $h\to 0$. 
Substituting  $u_1$ and $u_2$, given by \eqref{eq_proof_2_q},  into \eqref{eq_proof_1_new_iden_q}, letting $h\to 0$, and using that $q^{(1)}=q^{(2)}=0$ outside of $M$, we get  
\begin{equation}
\label{eq_proof_16}
\int\!\!\!\int\!\!\!\int_{\R\times D}(q^{(1)}-q^{(2)}) e^{i\lambda(x_1+ir)} c(x_1,r,\theta)b(\theta)dx_1drd\theta=0. 
\end{equation}
Here  we take $\omega\in \p D$ to define the Riemannian polar normal coordinates $(r,\theta)$. 

Set 
\[
f (r,\theta)=\int e^{i \lambda x_1} (q^{(1)}-q^{(2)})  c(x_1,r,\theta)  dx_1. 
\]
We have  $f\in L^\infty(D)$. Thus, it follows from \eqref{eq_proof_16} that 
\[
\int_{\mathbb{S}^{n-2}}\int_0^{\tau(\omega, \theta)} f(r,\theta) e^{-\lambda r} b(\theta)drd\theta=0,
\]
for all $\omega\in \p D$, all $b\in C^\infty(\mathbb{S}^{n-2})$.

An application of Proposition \ref{prop_Assylbekov}  with $\alpha=0$ gives that 
\[
f (r,\theta)=\int e^{i \lambda x_1} (q^{(1)}-q^{(2)})  c(x_1,r,\theta)  dx_1=0,\quad \text{for a.a. } r, \theta,
\]
and for $|\lambda|$ sufficiently small, and hence, for all $\lambda\in \C$ by the analyticity of the Fourier transform of the compactly supported function 
$(q^{(1)}-q^{(2)})c(\cdot, r,\theta)\in L^\infty$ for a.a. $(r,\theta)$. We conclude that $q^{(1)}=q^{(2)}$. 
This completes the proof of Theorem \ref{thm_main}.

\section{Gaussian beams quasimodes on conformally transversally anisotropic manifolds}

\label{sec_Gaussian_beam}

Let $(M,g)$ be  a conformally transversally anisotropic manifold so that $(M,g)\subset (\R\times M_0, c(e\oplus g_0))$.  Throughout this section we shall work under the simplifying assumption that the conformal factor $c=1$.  Replacing $(M_0,g_0)$ by a slightly larger manifold if necessary, we may assume that $(M,g)\subset (\R\times M^0_0, e\oplus g_0)$.

In this section we shall be concerned with constructing Gaussian beam quasimodes  for the magnetic Schr\"odinger operator $L_{g,A,q}$,  conjugated  by a liming Carleman weight,  with $A\in C(M, T^*M)$ and $q\in L^\infty(M,\C)$.  The construction of Gaussian beam quasimodes has a very long tradition in spectral theory and microlocal analysis, see \cite{Babich},  \cite{Ralston_1977}, \cite{Ralston_1982}. In the context closely related to our discussion it was given recently in \cite{DKuLS_2016}
in the case when $A=0$, and  in \cite{Cekic} assuming that $A\in C^\infty (M, T^*M)$.  Similarly to  \cite{Cekic}, our quasimodes will be constructed on the manifold $M$ and will be localized to a non-tangential  geodesic on the transversal manifold $M_0$.  Here a unit speed geodesic $\gamma:[0,L]\to M_0$ is called non-tangential if $\gamma(0),\gamma(L)\in \p M_0$,  $\dot{\gamma}(0), \dot{\gamma}(L)$ are non-tangential vectors on  $\p M_0$ and $\gamma(t)\in M_0^0$ for all $0<t<L$, see \cite{DKuLS_2016}.

In what follows it will be convenient to extend $A$ to a continuous $1$-form with compact support in $\R\times M_0^0$, and we shall write $A\in C_0(\R\times M_0^0, T^*(\R\times M_0^0))$.  Now using a partition of unity argument together with a regularization  in each coordinate patch, we get the following result. 

\begin{prop}
\label{prop_approximation_cont}
There exists a family  $A_\tau\in C^\infty_0(\R\times M_0^0, T^*(\R\times M_0^0))$ such that 
\begin{equation}
\label{eq_Car_20_0_cont}
\|A-A_\tau\|_{L^\infty}=o(1),\quad \tau\to 0,
\end{equation}
and 
\begin{equation}
\label{eq_Car_20_cont}
\begin{aligned}
\|A_\tau\|_{L^\infty}=\mathcal{O}(1), \quad
\|\nabla A_\tau\|_{L^\infty}=\mathcal{O}(\tau^{-1}),\\ 
\|\Delta A_\tau\|_{L^\infty}=\mathcal{O}(\tau^{-2}),\quad \tau\to 0.  
\end{aligned}
\end{equation}
\end{prop}

The Gaussian beam quasimodes to be constructed in this section will be used to construct complex geometric optics solutions for the magnetic Schr\"odinger operator $L_{g,A,q}$ in Section \ref{sec_CGO_based_quasi}.  To motivate our construction, we shall now proceed to introduce the conjugated operator and to this end let us write  $x=(x_1,x')$ for coordinates in $\R\times M_0$, globally in $\R$ and locally in $M_0$. Let
\[
s=\mu+i\lambda,\quad \mu,\lambda\in \R, \quad \mu\ge 1, \quad \lambda \text{ fixed}.  
\]

Our complex geometric optics solution to the equation 
\begin{equation}
\label{eq_gauss_1}
L_{g,A,q} u=0\quad \text{in}\quad M,
\end{equation}
will have the form
\begin{equation}
\label{eq_gauss_2}
u=e^{-sx_1}(v+r),
\end{equation}
where $v=v_s$ is an amplitude type term and $r=r_s$ is a correction term.  A function $u$ given by \eqref{eq_gauss_2} is a solution of \eqref{eq_gauss_1}
 provided that 
 \[
 e^{sx_1}L_{g,A,q} e^{-sx_1}r=- e^{sx_1}L_{g,A,q} e^{-sx_1}v.
 \]

As $g=e\oplus g_0$, we have
\begin{align*}
e^{sx_1}\circ (-\Delta_g) \circ e^{-sx_1}&= e^{sx_1}\circ (-\p_{x_1}^2-\Delta_{g_0}) \circ e^{-sx_1}
=-(\p_{x_1}- s)^2-\Delta_{g_0}\\
&=-\Delta_g+2s\p_{x_1}-s^2,
\end{align*}
and
\begin{align*}
&i e^{sx_1} d^* (A e^{-sx_1}v)=i d^*(Av)+i A_1 s v,\\
&- i e^{sx_1}  \langle A, d (e^{-sx_1} v)\rangle_g= -i  \langle A,dv \rangle_g+i A_1 s v. 
\end{align*}
Therefore, 
\begin{equation}
\label{eq_gauss_3}
\begin{aligned}
e^{sx_1}L_{g,A,q} e^{-sx_1}v=&-\Delta_g v +  i d^*(Av) -i \langle A,dv \rangle_g +(\langle A, A \rangle_g+q)v \\
& +2s\p_{x_1} v-s^2 v   +2 i s A_1  v.
\end{aligned}
\end{equation}
Here we are interested in choosing $v$ so that  the expression in  \eqref{eq_gauss_3} is small and to this end we have the following result. 

\begin{prop}
\label{prop_Gaussian_beams}
Let $(M,g)$ be a transversally anisotropic manifold so that $(M,g)\subset (\R\times M_0^0, g)$ with   $g=e\oplus g_0$. 
Let $A^{(1)}, A^{(2)}\in C(M,T^*M)$ and $q^{(1)},q^{(2)}\in L^\infty(M,\C)$. Let  $\gamma:[0,L]\to M_0$ be a unit speed  non-tangential geodesic on $M_0$,  and let $s=\mu+i\lambda$, $\mu\ge 1$, with $\lambda\in \R$  being fixed. Then there exist  families  of Gaussian beam quasimodes $v_s, w_s\in C^\infty(M)$ such that 
\begin{equation}
\label{eq_prop_gaussian_1}
\|v_s\|_{H^{1}_{\emph{\text{scl}}}(M^0)}=\mathcal{O}(1),\quad  \| e^{sx_1}h^2L_{g,A^{(1)},q^{(1)}} e^{-sx_1}v_s\|_{H^{-1}_{\emph{\text{scl}}}(M^0)}=o(h),
\end{equation}
and
\begin{equation}
\label{eq_prop_gaussian_2}
\|w_s\|_{H^{1}_{\emph{\text{scl}}}(M^0)}=\mathcal{O}(1),\quad 
\| e^{-sx_1}h^2L_{g,\overline{A^{(2)}},\overline{q^{(2)}}} e^{sx_1}w_s\|_{H^{-1}_{\emph{\text{scl}}}(M^0)}=o(h),
\end{equation}
as $h=\frac{1}{\mu}\to 0$. 
Furthermore, for each $\psi\in C(M_0)$ and $x'_1\in \R$, we have
\begin{equation}
\label{eq_prop_gaussian_3}
\lim_{h\to 0} \int_{\{x'_1\}\times M_0} v_s \overline{w_s} \psi dV_{g_0}= \int_0^L e^{-2\lambda t} \eta(x_1,t)e^{\Phi^{(1)}(x'_1,t)+\overline{\Phi^{(2)} (x'_1,t)}} \psi(\gamma(t)) dt.
\end{equation}
 Here $\Phi^{(1)}, \Phi^{(2)}\in C(\R\times [0,L])$ satisfy the following transport equations,
\[
(\p_{x_1}-i\p_t) \Phi^{(1)}=-i A^{(1)}_1(x_1,\gamma(t))- A^{(1)}_t(x_1,\gamma(t)),
\]
\[
 (\p_{x_1}+i\p_t) \Phi^{(2)}=-i \overline{A^{(2)}_1(x_1,\gamma(t))}+\overline{A^{(2)}_t (x_1,\gamma(t))},
\]
where 
\[
A^{(j)}_t (x_1,\gamma(t))=\langle A^{(j)}(x_1,\gamma(t)), (0,\dot{\gamma}(t))\rangle, \quad j=1,2,
\]
with $\langle \cdot, \cdot \rangle$ being the dually between tangent and cotangent vectors, 
 and $\eta\in C^\infty(\R\times [0,L])$ is such that $(\p_{x_1}-i\p_t) \eta=0$. 
Finally,  for any $1$-form $\alpha\in C(M,T^*M)$, we have
\begin{equation}
\label{eq_prop_gaussian_4}
\begin{aligned}
\lim_{h\to 0} h\int_{\{x'_1\}\times M_0} &\langle \alpha,dv_s  \rangle_g\overline{w_s}\psi dV_{g_0}\\
&=\int_0^L i \alpha(\dot{\gamma}(t)) e^{-2\lambda t} \eta(x_1,t)e^{\Phi^{(1)}(x'_1,t)+\overline{\Phi^{(2)} (x'_1,t)}} \psi(\gamma(t)) dt,
\end{aligned}
\end{equation}
and
\begin{equation}
\label{eq_prop_gaussian_5}
\begin{aligned}
\lim_{h\to 0} h\int_{\{x'_1\}\times M_0} &\langle \alpha,d \overline{w_s} \rangle_g v_s\psi dV_{g_0}\\
&=- \int_0^L i \alpha(\dot{\gamma}(t)) e^{-2\lambda t} \eta(x_1,t)e^{\Phi^{(1)}(x'_1,t)+\overline{\Phi^{(2)} (x'_1,t)}} \psi(\gamma(t)) dt.
\end{aligned}
\end{equation}
\end{prop}

\begin{proof}
We shall follow \cite{DKuLS_2016} and \cite{Cekic} closely, modifying the argument slightly to accommodate the magnetic potential of low regularity.  

Let us isometrically embed our manifold $(M_0,g_0)$ into a larger closed manifold $(\hat{M_0},g_0)$ of the same dimension. This is possible since we can form the manifold $\hat{M_0}=M_0\sqcup_{\p M_0} M_0$, which is the disjoint union of two copies of $M_0$, glued along the boundary.  We extend $\gamma$ as a unit speed geodesic in $\hat{M_0}$. Let $\varepsilon>0$ be such that $\gamma(t)\in 
\hat{M_0}\setminus M_0$ and $\gamma(t)$ has no self-intersection for $t\in [-2\varepsilon, 0)\cup (L,L+2\varepsilon]$. This choice of $\varepsilon$ is possible since $\gamma$ is non-tangential. 

Our goal is to construct Gaussian beam quasimodes near $\gamma([-\varepsilon, L+\varepsilon])$. We shall start by carrying out the quasimode construction locally near a given point $p_0=\gamma (t_0)$ on  $\gamma([-\varepsilon, L+\varepsilon])$. Let $(t,y)\in U=\{ (t,y)\in \R\times \R^{n-2}: |t-t_0|<\delta, |y|<\delta'\}$, $\delta, \delta'>0$,  be Fermi coordinates near $p_0$, see \cite{Kenig_Salo_APDE_2013}. We may assume that the coordinates $(t, y)$ extends smoothly to a neighborhood of $\overline{U}$. 
The geodesic $\gamma$ near $p_0$ is then given by $\Gamma=\{(t,y): y=0\}$, and 
\[
g_0^{jk}(t,0)=\delta^{jk},\quad \p_{y_l} g_0^{jk}(t, 0)=0. 
\]
Hence, near the geodesic 
\begin{equation}
\label{eq_gauss_3_metric}
g_0^{jk}(t,y)=\delta^{jk}+\mathcal{O}(|y|^2). 
\end{equation}

We shall first construct the quasimode $v=v_s$ in \eqref{eq_prop_gaussian_1} for the operator $e^{sx_1}h^2 L_{g,A^{(1)},q^{(1)}}e^{-sx_1}$. In doing so, let us write for simplicity $A=A^{(1)}$ and $q=q^{(1)}$.  
Let us consider the following Gaussian beam ansatz, 
\begin{equation}
\label{eq_gauss_4_v_s}
v(x_1, t,y; s)=e^{i s\varphi(t,y)} a(x_1, t,y; s).
\end{equation}
Here  $\varphi\in C^\infty (U,\C)$ is  such that 
\begin{equation}
\label{eq_gauss_4_v_s_phase}
\Im \varphi\ge 0,\quad \Im \varphi|_{\Gamma}=0,  \quad \Im \varphi(t,y)\sim |y|^2= \text{dist} ((y,t), \Gamma)^2,
\end{equation}
and $a \in C^\infty(\R\times U, \C)$ is an amplitute  such that $\text{supp}(a(x_1,\cdot))$ is close to  $\Gamma$, see \cite{Ralston_1982}, 
\cite{KKL_book}.  
Notice that here we choose  $\varphi$ to depend on the transversal variables $(t,y)$ only while $a$ is a function of all the variables. 

As  $\varphi$ is independent of $x_1$,   we get  
\begin{equation}
\label{eq_gauss_4}
e^{-is\varphi}  (-\Delta_g) e^{is\varphi} a =-\Delta_g a - is [2 \langle d\varphi, d a(x_1,\cdot)\rangle_{g_0} +(\Delta_{g_0} \varphi) a] +s^2 \langle d\varphi,d\varphi\rangle_{g_0}a,
\end{equation}
and  
\begin{equation}
\label{eq_gauss_5}
\begin{aligned}
&i e^{-i s\varphi } d^* (A e^{i s\varphi }a )=i d^*(Aa)+s a \langle d\varphi, A(x_1,\cdot)\rangle_{g_0}, \\
&- i e^{-i s\varphi }   \langle A, d (e^{i s\varphi }  a)\rangle_g= -i \langle A,da \rangle_g+s a \langle d\varphi, A(x_1,\cdot)\rangle_{g_0}, 
\end{aligned}
\end{equation}
Using \eqref{eq_gauss_3}, \eqref{eq_gauss_4},  \eqref{eq_gauss_5}, and the fact    that $\varphi$ is independent of $x_1$, we obtain that 
\begin{equation}
\label{eq_gauss_6}
\begin{aligned}
e^{sx_1}&L_{g,A,q} e^{-sx_1}v=  e^{is\varphi} [e^{-is\varphi}  e^{sx_1}L_{g,A,q} e^{-sx_1} e^{is\varphi} a]\\
&=e^{is\varphi} \bigg[  s^2 \big( \langle d\varphi,d\varphi \rangle_{g_0}-1\big)a \\
&+ s\big( 2\p_{x_1}a -2i \langle d\varphi, d a(x_1,\cdot) \rangle_{g_0}  -i  (\Delta_{g_0} \varphi) a + 2 a \langle d\varphi, A(x_1,\cdot) \rangle_{g_0} +2 i A_1 a \big)\\
 &-\Delta_g a +  i d^*(Aa ) -i \langle A,da \rangle_g +(\langle A, A \rangle_g+q)a\bigg].
\end{aligned}
\end{equation}

We start by considering the eikonal equation, 
\[
 \langle d\varphi,d\varphi \rangle_{g_0}-1=0,
\] 
and proceeding as in the classical Gaussian beam construction, see \cite{Ralston_1977},  \cite{Ralston_1982}, 
\cite{KKL_book},  \cite{DKuLS_2016}, we find $\varphi=\varphi(t,y)\in C^\infty(U,\C)$ such that 
\begin{equation}
\label{eq_gauss_7}
\langle d\varphi,d\varphi \rangle_{g_0}-1=\mathcal{O}(|y|^3),\quad y\to 0,
\end{equation}
and 
\begin{equation}
\label{eq_gauss_8}
\Im \varphi\ge c|y|^2,
\end{equation}
with some $c>0$.  Specifically, as explained in \cite{Ralston_1977}, \cite{Ralston_1982} and \cite{DKuLS_2016}, we can choose 
\begin{equation}
\label{eq_gauss_9}
\varphi(t,y)=t+\frac{1}{2} H(t) y\cdot y,
\end{equation}
where $H(t)$ is a unique smooth complex symmetric solution of the initial value problem for the matrix Riccati equation, 
\begin{equation}
\label{eq_Riccati}
\dot{H}(t)+H(t)^2=F(t), \quad H(t_0)=H_0,
\end{equation}
with $H_0$ being a complex symmetric matrix with $\Im (H_0)$ positive definite and $F(t)$ being a suitable symmetric matrix. Hence, as explained in  \cite{Ralston_1977}, \cite{Ralston_1982} and \cite{DKuLS_2016},  $\Im (H(t))$ is  positive definite for all $t$.

We shall next look for the  amplitude $a$ in the form
\begin{equation}
\label{eq_gauss_11}
a(x_1,t,y;\tau)=\mu^{\frac{n-2}{4}}a_0(x_1,t;\tau) \chi\bigg(\frac{y}{\delta'}\bigg),
\end{equation}
where $a_0(\cdot, \cdot;\tau) \in C^\infty(\R\times  \{t: |t-t_0|<\delta\} )$ is independent of $y$ and 
 satisfies
\begin{equation}
\label{eq_gauss_10}
\begin{aligned}
2\p_{x_1}a_0 -2i  \langle d\varphi, d a_0(x_1,\cdot) \rangle_{g_0}  -i  (\Delta_{g_0} \varphi) a_0 + 2 a_0 \langle d\varphi, A_\tau(x_1,\cdot) \rangle_{g_0} +2 i (A_\tau)_1 a_0\\
=\mathcal{O}(|y| \tau^{-1}), 
\end{aligned}
\end{equation} 
as $y\to 0$ and $\tau\to 0$. Here  $A_\tau$ is the regularization of $A$ given by Proposition \ref{prop_approximation_cont}, and $\chi\in C^\infty_0(\R^{n-2})$ is such that  $\chi=1$ for $|y|\le 1/4$ and $\chi=0$ for $|y|\ge 1/2$.

In order to determine $a_0$ such that \eqref{eq_gauss_10} holds, we shall Taylor expand the coefficients occurring in the left hand side of \eqref{eq_gauss_10}. First writing 
\begin{align*}
A_\tau(x,t,y)&=A_\tau(x,t,0)+\int_0^1\frac{d}{ds} A_\tau(x,t,ys)ds\\
&= A_\tau(x,t,0)+\int_0^1 ( \nabla_y A_\tau(x,t,ys))y ds,
\end{align*}
and using  Proposition \ref{prop_approximation_cont},
we obtain that  
\[
A_\tau(x,t,y)=A_\tau(x,t,0) +\mathcal{O}(|y|\tau^{-1} ).
\]
Next   it follows from \eqref{eq_gauss_9} that 
\[
\p_t \varphi(t,y)=1+\mathcal{O}(|y|^2).
\]
We finally have to compute $\Delta_{g_0}\varphi$ along the geodesic.  We have 
\[
(\Delta_{g_0}\varphi)(t,y)= (\Delta_{g_0}\varphi)(t,0)+\mathcal{O}(|y|),
\]
and  using \eqref{eq_gauss_3_metric} and  \eqref{eq_gauss_9},  we get 
\begin{align*}
(\Delta_{g_0}\varphi)(t,0)&=|g_0|^{-1/2}\p_{x_j}(|g_0|^{1/2} g_0^{jk}\p_{x_k}\varphi)|_{y=0} =\delta^{jk}\p_{x_j}\p_{x_k}\varphi|_{y=0}\\
&=\delta^{jk}H_{jk}=\tr H(t), \quad x=(t,y).
\end{align*}
To achieve \eqref{eq_gauss_10}, we shall therefore  require that $a_0(x_1,t; \tau)$ solves
\begin{equation}
\label{eq_gauss_10_specific}
(\p_{x_1} -i \p_t) a_0=\frac{1}{2}\big( -2i (A_\tau)_1(x_1,t,0)-2 (A_\tau)_t(x_1,t,0) +i\tr H(t)\big) a_0,
\end{equation}
where $A=A_1 dx_1+ A_t dt + A_ydy$. Writing 
\[
\p=\frac{1}{2} (\p_{x_1}-i\p_{t}),
\]
and looking for a solution in the form $a_0(x_1,t;\tau)=e^{\Phi_\tau (x_1,t)+f(t)}\eta(x_1,t)$, where $\p\eta=0$, we get 
\begin{equation}
\label{eq_gauss_12_eq_phi_tau}
\p \Phi_\tau(x_1,t) =-\frac{1}{2}(i (A_\tau)_1(x_1,t,0)+ (A_\tau)_t(x_1,t,0))\in C^\infty(\R_{x_1}\times [t_0- \delta, t_0+\delta]),
\end{equation}
with compact support in $x_1$,
and 
\begin{equation}
\label{eq_gauss_12_eq_f}
\p_t f=-\frac{1}{2} \tr H(t). 
\end{equation}
We solve \eqref{eq_gauss_12_eq_phi_tau} by taking 
\[
\Phi_\tau (x_1,t)=-\frac{1}{2\pi (x_1-it)}*(i (A_\tau)_1(x_1,t,0)+ (A_\tau)_t(x_1,t,0)), 
\]
using the standard fundamental solution of the  operator $\p$.  When forming the convolution in the variables $(x_1,t)$, we take a $C^\infty$ compactly supported extension of the right hand side of \eqref{eq_gauss_12_eq_phi_tau} to all of the $(x_1,t)$--plane so that the estimates of Proposition \ref{prop_approximation_cont} are still valid for the extension. 

We obtain the solution $a_0(x_1,t; \tau)\in C^\infty(\R\times [t_0-\delta, t_0+\delta])$ of \eqref{eq_gauss_10_specific} such that 
\begin{equation}
\label{eq_gauss_12}
\|\nabla_{x_1,t}^\alpha a_0\|_{L^\infty(J\times [t_0-\delta,t_0+\delta] )}=\mathcal{O}(\tau^{-|\alpha|}),\quad |\alpha|\le 2, \quad  \tau\to 0, 
\end{equation}
where $J\subset \R$ is a large fixed bounded open interval. 
Furthermore, we have 
\[
\|\Phi_\tau-\Phi\|_{L^\infty(J\times [t_0-\delta,t_0+\delta] ) }=o(1), \quad \tau\to 0,
\]
where $\Phi$ is continuous and solves
\[
\p \Phi(x_1, t)  =-\frac{1}{2}(i A_1(x_1,t,0)+ A_t(x_1,t,0)). 
\]
It follows from \eqref{eq_gauss_10_specific}  and \eqref{eq_gauss_12} that \eqref{eq_gauss_10} holds. 

In view of  \eqref{eq_gauss_4_v_s} and \eqref{eq_gauss_11} we write
\begin{equation}
\label{eq_gauss_13_v_s_def}
v(x_1,t,y)=e^{i s\varphi} \mu^{\frac{n-2}{4}}a_0(x_1,t; \tau) \chi\bigg(\frac{y}{\delta'}\bigg).
\end{equation}
We shall next check that \eqref{eq_prop_gaussian_1} is valid locally near the point $p_0$ for a suitable choice of $\tau$ depending on $s$. 
First using \eqref{eq_gauss_4_v_s_phase}, we see that 
\begin{equation}
\label{eq_gauss_13_v_s_est}
|v(x_1,t,y)|\le \mathcal{O}(1)  \mu^{\frac{n-2}{4}}  e^{-\mu c|y|^2} \chi\bigg(\frac{y}{\delta'}\bigg), \quad c>0.
\end{equation}
Then we have
\begin{equation}
\label{eq_gauss_v}
\|v\|_{L^2(J\times U)}\le \mathcal{O}(1) \|  \mu^{\frac{n-2}{4}} e^{-\mu c|y|^2}  \|_{L^2(|y|\le \delta'/2)}=\mathcal{O}(1),\quad \mu\to \infty. 
\end{equation}

Let us now estimate $\|e^{sx_1}h^2L_{g,A,q} e^{-sx_1}v\|_{H^{-1}_{\text{scl}}(J\times U)}$, $\mu=\frac{1}{h}$. Let us start with the first term in the right hand side of  \eqref{eq_gauss_6}.
Using that $\|a_0\|_{L^\infty}=\mathcal{O}(1)$,  it follows from \eqref{eq_gauss_7} and \eqref{eq_gauss_8} that 
\begin{equation}
\label{eq_gauss_13}
\begin{aligned}
h^2\|  e^{is\varphi}   s^2 \big(\langle d\varphi,d\varphi\rangle_{g_0}-1\big)a  \|_{L^2(J\times U)}\le \mathcal{O}(1)
\| e^{-c\mu |y|^2}  |y|^3 \mu^{\frac{n-2}{4}} \|_{L^2(|y|\le \delta'/2)}  \\
\le \mathcal{O}(1) \bigg( \int  e^{-2c|z|^2} \mu^{-3} |z|^{6} dz\bigg)^{1/2}=\mathcal{O}(\mu^{-3/2})=o(h). 
\end{aligned}
\end{equation}
 Here we make the change of variables $z=\mu^{1/2}y$. 

Let us now turn to the second term in the right hand side of \eqref{eq_gauss_6}. Consider first the contribution to the second term obtained when $A$ is replaced by its regularization $A_\tau$. Using \eqref{eq_gauss_11}, \eqref{eq_gauss_10}, and the fact that on $\supp d\chi(y/\delta')$, 
\[
|e^{is\varphi}|\sim e^{-\mu \tilde c},\quad \tilde c>0,
\]
we get 
\begin{equation}
\label{eq_gauss_14}
\begin{aligned}
h^2\|  e^{is\varphi}  & s\big( 2\p_{x_1}a -2i \langle d\varphi, d a \rangle_{g_0}  -i  (\Delta_{g_0} \varphi) a + 2 a \langle d\varphi, A_\tau  \rangle_{g_0} +2 i (A_\tau)_1 a \big) \|_{L^2(J\times U)}\\
&\le \mathcal{O}(h) \mu^{\frac{n-2}{4}} 
\bigg\|  e^{is\varphi}   \bigg[ |y|\tau^{-1} \chi\bigg(\frac{y}{\delta'}\bigg) -2i  \langle d\varphi, d \chi\bigg(\frac{y}{\delta'}\bigg) \rangle_{g_0} \bigg] \bigg\|_{L^2(J\times U)}\\
&\le \mathcal{O}(h) \mu^{\frac{n-2}{4}} 
\bigg(\int_{|y|\le \delta'/2} e^{-2c \mu |y|^2}  |y|^2 \tau^{-2} dy \bigg)^{1/2}+\mathcal{O}(e^{-\mu \tilde c})\\
&=\mathcal{O}\bigg(\frac{h \mu^{-1/2}}{\tau}\bigg) +\mathcal{O}(e^{-\mu \tilde c})=o(h),
\end{aligned}
\end{equation}
if we choose $\tau=h^{\sigma}$ with some $0<\sigma<1/2$.  To estimate the rest of the second term in the right hand side of \eqref{eq_gauss_6}, using \eqref{eq_Car_20_0_cont}, we obtain that 
\begin{equation}
\label{eq_gauss_15}
\begin{aligned}
h^2\|  e^{is\varphi} s[2a \langle d\varphi, A-A_\tau\rangle_{g_0} &+2i (A_1-(A_\tau)_1)a] \|_{L^2(J\times U)}\\
&\le \mathcal{O}(h) \|A-A_\tau\|_{L^\infty} \|e^{-\mu c |y|^2}  \mu^{\frac{n-2}{4}} \|_{L^2(|y|\le \delta'/2)} =o(h).
\end{aligned}
\end{equation}

Let us now start estimating the third term in the right hand side of \eqref{eq_gauss_6}. First using \eqref{eq_gauss_11} and 
\eqref{eq_gauss_12}, we get 
\begin{equation}
\label{eq_gauss_16}
\begin{aligned}
&h^2 \|e^{is\varphi} (-\Delta_g a)\|_{L^2(J\times U)}\le h^2\bigg\| e^{is\varphi}\mu^{\frac{n-2}{4}} (\Delta_g a_0(x_1,t))\chi\bigg(\frac{y}{\delta'}\bigg)\bigg\|_{L^2(J\times U)}\\
&+h^2\bigg\| e^{is\varphi}\mu^{\frac{n-2}{4}} 
\bigg[2 \langle \nabla_g a_0(x_1,t), \nabla_g \chi\bigg(\frac{y}{\delta'}\bigg)\rangle_g+ a_0 \Delta_g \chi\bigg(\frac{y}{\delta'}\bigg) \bigg]\bigg\|_{L^2(J\times U)}\\
&\le \mathcal{O}(h^2) 
\| e^{-c\mu |y|^2}\mu^{\frac{n-2}{4}}\tau^{-2}\|_{L^2(|y|\le \delta'/2)}+ \mathcal{O}(e^{-\mu \tilde c})=\mathcal{O}(h^2\tau^{-2})+ \mathcal{O}(e^{-\mu \tilde c})=o(h),
\end{aligned}
\end{equation}
and
\begin{equation}
\label{eq_gauss_17}
\begin{aligned}
h^2 \| e^{is\varphi} & \langle A,da\rangle_g \|_{L^2(J\times U)}\le \mathcal{O}(h^2)\mu^{\frac{(n-2)}{4}}\| e^{is\varphi}  \langle A,da_0\rangle_g    \|_{L^2(J\times U)} + \mathcal{O}(e^{-\mu \tilde c})\\
&\le \mathcal{O}\bigg(\frac{h^2}{\tau}\bigg)\mu^{\frac{n-2}{4}}
\| e^{-c\mu |y|^2}\|_{L^2(|y|\le \delta'/2)}+ \mathcal{O}(e^{-\mu \tilde c})=o(h),
\end{aligned}
\end{equation}
and 
\begin{equation}
\label{eq_gauss_18}
\begin{aligned}
h^2\|e^{is\varphi} (\langle A, A\rangle_g+q)a \|_{L^2(J\times U)} \le \mathcal{O}(h^2)\mu^{\frac{n-2}{4}}
\| e^{-c\mu |y|^2}\|_{L^2(|y|\le \delta'/2)}=\mathcal{O}(h^2). 
\end{aligned}
\end{equation}
Finally, let us estimate $h^2 \|e^{is\varphi} i d^*(Aa ) \|_{H^{-1}_{\text{scl}}(J\times U)}$. To that end, letting $0\ne \psi\in C^\infty_0(I\times U)$, we obtain that  
\begin{align*}
h^2 | \langle e^{is\varphi} d^*(Aa ),\psi\rangle | &\le  h^2 | \langle e^{is\varphi} d^*(A_\tau a ),\psi\rangle|
+ h^2\int |\langle (A-A_\tau) a , d(e^{is\varphi}\psi) \rangle_g |dV_g\\
&\le (h^2\mathcal{O}(\tau^{-1}) \|e^{is\varphi} \mu^{\frac{n-2}{4}}\|_{L^2 (|y|\le \delta'/2)}+ \mathcal{O}(e^{-\mu\tilde c}))\|\psi\|_{L^2}\\
&+(o(h^2)\mu + o(h)) \|e^{is\varphi} \mu^{\frac{n-2}{4}}\|_{L^2 (|y|\le \delta'/2)}\|\psi\|_{H^1_{\text{scl}}}\\
&=o(h) \|\psi\|_{H^1_{\text{scl}}},
\end{align*}
and therefore, 
\begin{equation}
\label{eq_gauss_19}
h^2 \|e^{is\varphi} i d^*(Aa ) \|_{H^{-1}_{\text{scl}}(J\times U)}=o(h), \quad h\to 0. 
\end{equation}

Thus, we conclude from \eqref{eq_gauss_6} with the help of \eqref{eq_gauss_13},  \eqref{eq_gauss_14},  \eqref{eq_gauss_15},  \eqref{eq_gauss_16},  \eqref{eq_gauss_17},  \eqref{eq_gauss_18} and  \eqref{eq_gauss_19}, that 
\begin{equation}
\label{eq_gauss_conj_H_-1}
\|e^{sx_1}h^2L_{g,A,q} e^{-sx_1}v\|_{H^{-1}_{\text{scl}}(J\times U)}=o(h), \quad h\to 0. 
\end{equation}
 
 Finally, using \eqref{eq_gauss_13_v_s_def}, we also get 
\begin{equation}
\label{eq_gauss_dv}
\|d v\|_{L^2(J\times U)}=\mathcal{O}(h^{-1}),\quad h\to 0. 
\end{equation}
This complete the verification of  \eqref{eq_prop_gaussian_1}  locally near the point $p_0$.

For the later purposes we shall need an  estimate for $\|v(x_1,\cdot)\|_{L^2(\p M_0)}$. If $U$ contains a boundary point $x_0=(t_0,0)\in \p M_0$, then $\p_t |_{x_0}$ is transversal to $\p M_0$. Let $\rho$ be a boundary defining function for $M_0$ so that $\p M_0$ is given by the zero set $\rho(t,y)=0$ near $x_0$. Then $\nabla \rho(x_0)$  is normal to $\p M_0$, and hence, $\p_t \rho(x_0)\ne 0$. By the implicit function theorem, there is a smooth function $y\mapsto t(y)$ near $0$ such that $\p M_0$ near $x_0$ is given by $\{ (t(y), y): |y|<r_0\}$ for some $r_0>0$ small,  see also \cite{Kenig_Salo_APDE_2013}. Then using \eqref{eq_gauss_13_v_s_def}, \eqref{eq_gauss_13_v_s_est}, we get
\begin{equation}
\label{eq_v_s_on_boundary}
\begin{aligned}
\|v(x_1,\cdot)\|_{L^2(\p M_0\cap U)}^2&=\int_{|y|<r_0} | v(x_1,t(y),y)|^2 dS(y)\\
&\le \mathcal{O}(1) \int_{\R^{n-2}} \mu^{\frac{n-2}{2}}e^{-2\mu c |y|^2}dy=\mathcal{O}(1),   
\end{aligned}
\end{equation}
as $\mu\to \infty$.

We shall now construct the quasimode $v_s$ in $M$ by gluing together quasimodes defined along small pieces of the geodesic.  
Since $\hat M_0$ is a compact manifold and $\gamma:(-2\varepsilon, L+2\varepsilon)\to \hat M_0$ is a unit speed non-tangential geodesic with no loops, it follows from \cite[Lemma 7.2]{Kenig_Salo_APDE_2013} that $\gamma|_{[-\varepsilon,L+\varepsilon]}$ self-intersects only at finitely many times $t_j$ with 
\[
-\varepsilon=t_0<t_1<\dots <t_N<t_{N+1}=L+\varepsilon.
\]
Then an application of  \cite[Lemma 3.5]{DKuLS_2016} shows that there exists an open cover $\{(U_j,\kappa_j)\}_{j=0}^{N+1}$ of $\gamma([-\varepsilon,L+\varepsilon])$ consisting of coordinate neighborhoods having the following properties: 
\begin{itemize}
\item[(i)] $\kappa_j(U_j)=I_j\times B$, where $I_j$ are open intervals and $B=B(0,\delta')$ is an open ball in $\R^{n-2}$.  Here $\delta'>0$  can be taken arbitrarily small and the same for each $U_j$, 
\item[(ii)] $\kappa_j(\gamma(t))=(t,0)$ for each $t\in I_j$,
\item[(iii)] $t_j$ only belongs to $I_j$ and $\overline{I_j}\cap \overline{I_k}=\emptyset$ unless $|j-k|\le 1$,
\item[(iv)] $\kappa_j=\kappa_k$ on $\kappa_j^{-1}((I_j\cap I_k)\times B)$.
\end{itemize}

As explained in \cite[Lemma 3.5]{DKuLS_2016}, the intervals $I_j$ can be chosen as follows,
\begin{align*}
I_0=(-2\varepsilon, t_1-\tilde \delta), \quad I_j&=(t_j-2\tilde \delta, t_{j+1}-\tilde \delta), \quad j=1,\dots, N, \\
 I_{N+1}&=(t_{N+1}-2\tilde\delta, L+2\varepsilon ),
\end{align*}
for some $\tilde \delta>0$ small enough. Furthermore, the metric $g_0$ expressed in these coordinates satisfies,  
\[
g_0^{jk}|_{\gamma(t)}=\delta^{jk},\quad \nabla g_0^{jk}|_{\gamma(t)}=0,
\]
see \cite[Lemma 3.5]{DKuLS_2016}.  As observed in the proof of \cite[Lemma 3.5]{DKuLS_2016}, in the case when $\gamma$ does not self-intersect, there is a single coordinate neighborhood of  $\gamma|_{[-\varepsilon,L+\varepsilon]}$  so that (i) and (ii) are satisfied. 

To construct the quasimode $v_s$ we proceed as follows. First we find a function $v_s^{(0)}=e^{is\varphi^{(0)}}a^{(0)}$, $a_0^{(0)}=e^{\Phi_\tau^{(0)}+f^{(0)}}$, in $U_0$ with some fixed initial  conditions at $t=-\varepsilon$ for the Riccati equation \eqref{eq_Riccati} determining $\varphi^{(0)}$ and for the equation \eqref{eq_gauss_12_eq_f} determining $f^{(0)}$.  
Choose some $t_0'$ with $\gamma(t_0')\in U_0\cap U_1$, and let $v_s^{(1)}=e^{is\varphi^{(1)}}a^{(1)}$ be the quasimode in $U_1$ obtained by demanding that  
\[
\varphi^{(1)}(t_0')=\varphi^{(0)}(t_0'), \quad f^{(1)}(t_0')=f^{(0)}(t_0'). 
\]
Also notice that $\Phi^{(0)}_\tau$ and $\Phi^{(1)}_\tau$ both satisfy the equation   
\eqref{eq_gauss_12_eq_phi_tau} and we can arrange so that $\Phi^{(0)}_\tau=\Phi^{(1)}_\tau$ on $U_0\cap U_1$. Continuing in this way we obtain  the quasimodes $v^{(2)}_s,\dots, v_s^{(N+1)}$ such that 
\begin{equation}
\label{eq_equal_quasi}
v_s^{(j)}(x_1,\cdot)=v_s^{(j+1)}(x_1,\cdot)\quad \text{in}\quad  U_j\cap U_{j+1},
\end{equation}
for all $x_1$. 

 Let $\chi_j=\chi_j(t)\in C^\infty_0(I_j)$ be such that $\sum_{j=0}^{N+1} \chi_j=1$ near $\gamma([-\varepsilon,L+\varepsilon])$, and define
 \[
 v_s=\sum_{j=0}^{N+1} \chi_j v_s^{(j)}. 
 \]

Let $p_1,\dots,p_R\in M_0$ be the distinct points where the geodesic self-intersects, and let $0\le t_1<\dots<t_{R'}$ be the times of self-intersections. Let $V_1,\dots, V_R$ be small neighborhoods in $\hat M_0$ around $p_j$, $j=1,\dots, R$. Then choosing $\delta'$ small enough we obtain an open cover in $\hat M_0$, 
\begin{equation}
\label{eq_open_cover_sup}
\supp(v_s(x_1,\cdot ))\cap M_0\subset (\cup_{j=1}^R V_j)\cup (\cup_{k=1}^S W_k),
\end{equation}
where in each $V_j$, the quasimode is a finite sum,
\begin{equation}
\label{eq_rep_vs_1}
v_s(x_1,\cdot)|_{V_j}=\sum_{l: \gamma(t_l)=p_j} v_s^{(l)}(x_1,\cdot),
\end{equation}
and in each $W_k$, in view of \eqref{eq_equal_quasi}, there is some $l(k)$ so that the quasimode is given by 
\begin{equation}
\label{eq_rep_vs_2}
v_s(x_1,\cdot)|_{W_k}=v_s^{l(k)}(x_1,\cdot).
\end{equation}
Hence, the $L^2$ bounds $\|v_s\|_{L^2(M)}=\mathcal{O}(1)$ and $\|dv_s\|_{L^2(M)}=\mathcal{O}(h^{-1})$ follow from \eqref{eq_gauss_v} and
\eqref{eq_gauss_dv}.  Furthermore, \eqref{eq_v_s_on_boundary} implies the bound $\|v_s(x_1,\cdot)\|_{L^2(\p M_0)}=\mathcal{O}(1)$. 

We shall show that  $\| e^{sx_1}h^2L_{g,A,q} e^{-sx_1}v_s\|_{H^{-1}_{\text{scl}}(M^0)}=o(h)$. In doing so,  we observe that 
\eqref{eq_open_cover_sup} gives that 
\[
\supp(v_s)\cap M\subset (\cup_{j=1}^R  \tilde J\times V_j)\cup (\cup_{k=1}^S \tilde J\times W_k):=\cup\Omega_l,
\]
where $\tilde J\subset\R$ is a bounded open interval.  It follows from  \eqref{eq_rep_vs_1}, \eqref{eq_rep_vs_2} and \eqref{eq_gauss_conj_H_-1}
that  
\begin{equation}
\label{eq_oper_H_-1_sup}
\begin{aligned}
&\| e^{sx_1}h^2L_{g,A,q} e^{-sx_1}v_s\|_{H^{-1}_{\text{scl}}(  \tilde J\times V_j)}=o(h),  \\
&\| e^{sx_1}h^2L_{g,A,q} e^{-sx_1}v_s\|_{H^{-1}_{\text{scl}}(  \tilde J\times W_k)}=o(h),  \quad h\to 0. 
\end{aligned}
\end{equation}
Let $\psi\in C^\infty_0(M^0)$ and let $0\le \chi_l\in C^\infty_0(\Omega_l)$ be such that $\sum\chi_l=1$ near $\supp(v_s)\cap M$. Writing 
\[
\psi=(1-\sum\chi_l)\psi+\sum\chi_l\psi,
\]
and using \eqref{eq_oper_H_-1_sup}, we get 
\[
|\langle e^{sx_1}h^2L_{g,A,q} e^{-sx_1}v_s, \psi \rangle_{M^0}  |\le o(h)\|\psi\|_{H^1_{\text{scl}}(M^0)},
\]
showing the claim. This completes the proof of \eqref{eq_prop_gaussian_1}.

Now look for a Gaussian beam quasimode for the operator $e^{-sx_1}h^2L_{g,\overline{A^{(2)}},\overline{q^{(2)}}} e^{sx_1}$ in the form
\[
w_s=w=e^{i s\varphi} b,
\]
with the same phase function $\varphi\in C^\infty(U)$ satisfying \eqref{eq_gauss_4_v_s_phase}, and $b\in C^{\infty}(\R\times U)$ supported near $\Gamma$.  Using \eqref{eq_gauss_3} with $s$ replaced by $-s$, \eqref{eq_gauss_4},  \eqref{eq_gauss_5}, and   that $\varphi$ is independent of $x_1$, we obtain, similarly to \eqref{eq_gauss_6}  that 
\begin{equation}
\label{eq_gauss_6_second_eq}
\begin{aligned}
&e^{-sx_1}  L_{g,\overline{A^{(2)}},\overline{q^{(2)}}} e^{sx_1}w=  e^{is\varphi} [e^{-is\varphi}  e^{-sx_1}L_{g,\overline{A^{(2)}},\overline{q^{(2)}}} e^{sx_1} e^{is\varphi} b]\\
&=e^{is\varphi} \bigg[  s^2 \big(\langle d\varphi,d\varphi \rangle_{g_0}-1\big)b  \\
&+ s\big( - 2\p_{x_1}b -2i \langle d\varphi, d b \rangle_{g_0}  -i  (\Delta_{g_0} \varphi) b + 2 b \langle d\varphi, \overline{A^{(2)}}\rangle _{g_0} -2 i \overline{A^{(2)}_1} b \big)\\
 &-\Delta_g b +  i d^*(\overline{A^{(2)}}b ) -i \langle \overline{A^{(2)}},db\rangle_g +(\langle \overline{A^{(2)}}, \overline{A^{(2)}}\rangle_g+\overline{q^{(2)}})b\bigg].
\end{aligned}
\end{equation}

We shall next find the amplitude $b$ in the form
\begin{equation}
\label{eq_gauss_11_second_eq}
b(x_1,t,y;\tau)=\mu^{\frac{n-2}{4}}b_0(x_1,t;\tau) \chi(\frac{y}{\delta'}),
\end{equation}
where $b_0(\cdot, \cdot;\tau) \in C^\infty(\R\times\{t:|t-t_0|<\delta\})$. To that end, similarly to \eqref{eq_gauss_10_specific}, we require that $b_0$ solves
\[
(\p_{x_1} +i \p_t) b_0=\frac{1}{2}\big( -2i (\overline{A^{(2)}_\tau)_1}(x_1,t,0)+2 \overline{(A^{(2)}_\tau)_t}(x_1,t,0) -i\tr H(t)\bigg) b_0.
\]
Writing 
\[
\overline{\p}=\frac{1}{2} (\p_{x_1}+i\p_{t}),
\]
and looking for a solution in the form $b_0=e^{\Phi^{(2)}_\tau (x_1,t)+f^{(2)}(t)}$, we get 
\begin{equation}
\label{eq_gauss_12_eq_phi_tau__second_eq}
\overline{\p} \Phi^{(2)}_\tau =\frac{1}{2}\big( -i (\overline{A^{(2)}_\tau)_1}(x_1,t,0)+ \overline{(A^{(2)}_\tau)_t}(x_1,t,0)) ,
\end{equation}
and 
\[
\p_t f^{(2)}=-\frac{1}{2} \tr H(t). 
\]
Proceeding further as in the construction of the quasimode $v_s$ above, we obtain the quasimode $w_s\in C^\infty(M)$ such that \eqref{eq_prop_gaussian_2}  holds.

Let us now verify \eqref{eq_prop_gaussian_3} for $\psi\in C(M_0)$ and $x'_1\in \R$. Using a partition of unity, it is enough to check \eqref{eq_prop_gaussian_3} for $\psi$ having compact support in one of the sets $V_j$ or   $W_k$, see \eqref{eq_open_cover_sup}.  Let us first consider the easier case when $\psi\in C_0 ( M_0)$, $\supp(\psi)\subset W_k$.  Here on $\supp(\psi)$, we have 
\begin{equation}
\label{eq_gauss_25_v_s_w_s}
v_s=e^{is\varphi} \mu^{\frac{n-2}{4}}a_0 (x'_1,t; \mu)\chi\bigg(\frac{y}{\delta'}\bigg),\quad w_s=e^{is\varphi} \mu^{\frac{n-2}{4}}b_0 (x'_1,t;\mu)\chi\bigg(\frac{y}{\delta'}\bigg),
\end{equation}
with $\varphi=\varphi(t,y)$. First it follows from \eqref{eq_gauss_3_metric} that 
\[
|g_0|^{1/2}=1+\mathcal{O}(|y|^2).
\]
Using \eqref{eq_gauss_9}, we get 
\begin{equation}
\label{eq_gauss_26}
\begin{aligned}
& \int_{\{x'_1\}\times M_0} v_s \overline{w_s} \psi dV_{g_0}\\
 &= \int_{0}^L \int_{\R^{n-2}} e^{-2\mu\text{Im}\varphi} e^{-2\lambda\text{Re}\varphi} \mu^{\frac{n-2}{2}} a_0(x'_1,t;\mu)\overline{b_0(x'_1,t;\mu)} \chi^2\bigg(\frac{y}{\delta'}\bigg) \psi(t,y) |g_0|^{\frac{1}{2}} dt dy\\
 &=\int_0^L \int_{\R^{n-2}} e^{-\mu\text{Im} H(t)y\cdot y}  e^{-2\lambda t} e^{\lambda \mathcal{O}(|y|^2)} \mu^{\frac{n-2}{2}} a_0(x'_1,t;\mu)\overline{b_0(x'_1,t;\mu)}\\ 
 &\chi^2\bigg(\frac{y}{\delta'}\bigg) \psi(t,y)
 (1+\mathcal{O}(|y|^2)) dt dy.
\end{aligned}
\end{equation}
Performing the change of variables $\mu^{1/2} y=x$ in \eqref{eq_gauss_26}, we obtain that  
\begin{equation}
\label{eq_gauss_27}
\begin{aligned}
\int_0^L \int_{\R^{n-2}} & e^{-\text{Im} H(t) x\cdot x}  e^{-2\lambda t} e^{\frac{\lambda}{\mu} \mathcal{O}(|x|^2)} \\
&a_0(x'_1,t;\mu)\overline{b_0(x'_1,t;\mu)}\chi^2\bigg(\frac{x}{\mu^{1/2}\delta'}\bigg) \psi\bigg(t,\frac{x}{\mu^{1/2}}\bigg)
 (1+\mu^{-1}\mathcal{O}(|x|^2)) dt dx.
\end{aligned}
\end{equation}
Recall that  $\text{Im} H(t)x\cdot x\ge c|x|^2$, $c>0$. Furthermore, 
\[
a_0(x'_1,t;\mu)\to e^{\Phi^{(1)}(x'_1,t)+f^{(1)}(t)}\eta(x_1,t), \quad b_0(x'_1,t;\mu)\to e^{\Phi^{(2)}(x'_1,t)+f^{(2)}(t)},\quad \mu\to \infty,
\] 
uniformly,  where 
\[
\p \Phi^{(1)}=-\frac{1}{2} (i A^{(1)}_1(x_1,t,0)+A^{(1)}_t(x_1,t,0)),
\]
\[
\overline{\p} \Phi^{(2)}=\frac{1}{2} (-i \overline{A^{(2)}_1(x_1,t,0)}+\overline{A^{(2)}_t (x_1,t,0)}),
\]
\begin{equation}
\label{eq_gauss_28}
\p_t f^{(j)}=-\frac{1}{2}\tr H(t), \quad j=1,2.
\end{equation}

Passing to the limit as $\mu\to\infty$ in \eqref{eq_gauss_26} and \eqref{eq_gauss_27}, by the dominated convergence theorem, we obtain that 
\begin{equation}
\label{eq_gauss_29}
\begin{aligned}
\lim_{\mu\to \infty} &  \int_{\{x'_1\}\times M_0} v_s \overline{w_s} \psi dV_{g_0}\\
&=\int_0^L e^{-2\lambda t} e^{\Phi^{(1)}(x'_1,t)+\overline{\Phi^{(2)} (x'_1,t)}+f^{(1)}(t)+\overline{f^{(2)}(t)}}\eta(x_1,t) \psi(t,0) \int_{\R^{n-2}} e^{-\text{Im} H(t) x\cdot x} dxdt.
\end{aligned}
\end{equation}
Let us now simplify the expression in the right hand side of \eqref{eq_gauss_29}. To that end, notice that 
\begin{equation}
\label{eq_gauss_30}
\int_{\R^{n-2}} e^{-\text{Im} H(t) x\cdot x} dx=\frac{\pi^{(n-2)/2}}{\sqrt{\det(\text{Im} H(t))}}.
\end{equation}
and recall from \cite[Lemma 2.58]{KKL_book} that 
\begin{equation}
\label{eq_gauss_31}
\det(\text{Im} H(t))= \det(\text{Im} H(t_0)) e^{-2 \int_{t_0}^t \tr \text{Re} (H(s))ds}.
\end{equation}
Now it follows from \eqref{eq_gauss_28} that 
\[
\p_t (f^{(1)}+\overline{f^{(2)}})=-\tr \Re H(t),
\]
and therefore, 
\begin{equation}
\label{eq_gauss_32}
f^{(1)}+\overline{f^{(2)}}=C- \int_{t_0}^t \tr \text{Re} (H(s))ds.
\end{equation}
Choosing $f^{(1)}(t_0)$ and $f^{(2)}(t_0)$ so that 
\begin{equation}
\label{eq_normalization}
 \frac{ e^{f^{(1)}(t_0)+ f^{(2)}(t_0)} \pi^{(n-2)/2}}{\sqrt{\det(\text{Im} H(t_0))}}=1, 
\end{equation}
we obtain from \eqref{eq_gauss_29} using \eqref{eq_gauss_30}, \eqref{eq_gauss_31} and \eqref{eq_gauss_32} that 
\begin{equation}
\label{eq_gauss_32_arrang}
\lim_{\mu\to \infty}   \int_{\{x'_1\}\times M_0} v_s \overline{w_s} \psi dV_{g_0}=\int_0^L e^{-2\lambda t} e^{\Phi^{(1)}(x'_1,t)+\overline{\Phi^{(2)} (x'_1,t)}} \eta(x_1,t)\psi(t,0) dt.
\end{equation}
This completes the proof of \eqref{eq_prop_gaussian_3} in the case when $\supp(\psi)\subset W_k$.

Let us now establish \eqref{eq_prop_gaussian_3} when   $\supp(\psi) \subset V_j$.  Here  on $\supp(\psi)$ we have
\[
v_s=\sum_{l: \gamma(t_l)=p_j} v_s^{(l)}, \quad w_s=\sum_{l: \gamma(t_l)=p_j} w_s^{(l)},
\]
and hence, 
\begin{equation}
\label{eq_gaus_v_s_w_s_sum} 
v_s\overline{w_s}=\sum_{l: \gamma(t_l)=p_j} v_s^{(l)}\overline{w_s^{(l)}}+ \sum_{l\ne l',\gamma(t_l)=\gamma(t_{l'})=p_j} v_s^{(l)}\overline{w_s^{(l')}}.
\end{equation}
Arguing similarly to \cite{DKuLS_2016}, we shall show that the contribution of the mixed terms vanishes  in the limit $\mu\to \infty$, i.e. if $l\ne l'$,
\begin{equation}
\label{eq_gauss_33}
\lim_{\mu\to \infty}   \int_{\{x'_1\}\times M_0} v^{(l)}_s \overline{w^{(l')}_s} \psi dV_{g_0}= 0.
\end{equation}
To that end, write
\[
v_s^{(l)}=e^{i\mu\text{Re}\, \varphi^{(l)}}p^{(l)}, \quad p^{(l)}=e^{-\lambda \text{Re}\, \varphi^{(l)}} e^{-s \text{Im}\, \varphi^{(l)}}a^{(l)},
\]
and
\[
w_s^{(l')}=e^{i\mu\text{Re}\, \varphi^{(l')}}q^{(l')}, \quad q^{(l')}=e^{-\lambda \text{Re}\, \varphi^{(l')}} e^{-s \text{Im}\, \varphi^{(l')}}b^{(l')},
\]
and therefore, 
\begin{equation}
\label{eq_gauss_34}
v^{(l)}_s \overline{w^{(l')}_s} =e^{i\mu\phi} p^{(l)} \overline{q^{(l')}},
\end{equation}
where 
\[
\phi= \Re \varphi^{(l)}-\Re \varphi^{(l')}.
\]
Thus, in view of \eqref{eq_gauss_33} and \eqref{eq_gauss_34} we shall show that for $l\ne l'$,
\begin{equation}
\label{eq_gauss_35}
\lim_{\mu\to \infty}   \int_{\{x'_1\}\times M_0} e^{i\mu\phi} p^{(l)} \overline{q^{(l')}}  \psi dV_{g_0}= 0.
\end{equation}
Since $\p_t\varphi^{(l)}(t,0)=\p_t\varphi^{(l')}(t,0)=1$ and the geodesic intersects itself transversally, as explained in \cite[Lemma 7.2]{Kenig_Salo_APDE_2013},   we see that $d\phi(p_j) \ne 0$.  
By decreasing the set $V_j$ if necessary, we may assume that  $d\phi\ne 0$ in $V_j$.

To prove \eqref{eq_gauss_35}, we shall integrate by parts and in doing so, we let $\varepsilon>0$ and decompose 
$
\psi=\psi_1+\psi_2,
$
where $\psi_1\in C^\infty(M_0)$, $\supp(\psi_1)\subset V_j$ and  and $\|\psi_2\|_{L^\infty(V_j\cap M_0)}\le \varepsilon$.  Notice that $\psi$ may be nonzero on $\p M_0$.  We have
\begin{equation}
\label{eq_gauss_36}
\bigg| \int_{\{x'_1\}\times M_0} e^{i\mu\phi} p^{(l)} \overline{q^{(l')}}  \psi_2 dV_{g_0}\bigg| \le \| v^{(l)}_s\|_{L^2}\| w^{(l)}_s\|_{L^2}\|\psi_2\|_{L^\infty}\le \mathcal{O} (\varepsilon). 
\end{equation}
For the smooth part $\psi_1$, we integrate by parts using that 
\[
e^{i\mu\phi}=\frac{1}{i\mu} L(e^{i\mu \phi}), \quad L=\frac{1}{|d\phi|^{2}}\langle  d\phi, d\cdot \rangle_{g_0}.
\]
We have
\begin{equation}
\label{eq_gauss_37}
\begin{aligned}
 \int_{\{x'_1\}\times M_0} e^{i\mu\phi} p^{(l)} \overline{q^{(l')}}  \psi_1 dV_{g_0} =&\int_{\{x'_1\}\times (V_j\cap \p M_0)} \frac{\p_\nu \phi}{i\mu |d\phi|^2} e^{i\mu\phi} p^{(l)} \overline{q^{(l')}} \psi_1 dS \\
 &+\frac{1}{i\mu } \int_{\{x'_1\}\times M_0} e^{i\mu\phi} L^t( p^{(l)} \overline{q^{(l')}}  \psi_1) dV_{g_0}, 
\end{aligned}
\end{equation}
where $L^t=-L-\div L$ is the transpose of $L$. 

In view of \eqref{eq_v_s_on_boundary}, the boundary term is of $\mathcal{O}(\mu^{-1})$ as $\mu\to \infty$. To estimate the second term in the right hand side of \eqref{eq_gauss_37}, we recall that 
\begin{align*}
p^{(l)} \overline{q^{(l')}}=e^{-\lambda (\text{Re}\, \varphi^{(l)}+\text{Re}\, \varphi^{(l')})} e^{-i\lambda (\text{Im}\, \varphi^{(l)}-\text{Im}\, \varphi^{(l')})} 
e^{-\mu (\text{Im}\, \varphi^{(l)}+\text{Im}\, \varphi^{(l')})} \\
\mu^{\frac{n-2}{2}} a_0^{(l)}(x_1',t,\tau)\overline{b_0^{(l')}(x_1',t,\tau)}\chi^2 \bigg(\frac{y}{\delta'}\bigg).
\end{align*}
This shows that to bound the second term in the right hand side of \eqref{eq_gauss_37}, it suffices to analyze the contributions occurring when differentiating 
\[
e^{-\mu (\text{Im}\, \varphi^{(l)}+\text{Im}\, \varphi^{(l')})} a_0^{(l)}(x_1',t,\tau)\overline{b_0^{(l')}(x_1',t,\tau)},
\]
as all the other contributions are  of $\mathcal{O}(\frac{1}{\mu})$, as $\mu\to \infty$.  

As in \cite{DKuLS_2016}, we have
\[
|L (e^{-\mu (\text{Im}\, \varphi^{(l)}+\text{Im}\, \varphi^{(l')})} )|\le  \mathcal{O}(\mu) |d (\text{Im}\, \varphi^{(l)}+\text{Im}\, \varphi^{(l')}) |e^{-\mu c|y|^2}\le \mathcal{O}
(\mu |y|)e^{-\mu c|y|^2},
\]
which shows that the corresponding contribution to the second term  in the right hand side of \eqref{eq_gauss_37} is of $\mathcal{O}(\mu^{-1/2})$. 

Now it follows from \eqref{eq_gauss_12} that 
\[
|L(a_0^{(l)}(x_1',t,\tau)\overline{b_0^{(l')}(x_1',t,\tau)})|\le \mathcal{O}(\mu^{\sigma}), 
\]
with $0<\sigma<1/2$, and thus, the   corresponding contribution to the second term  in the right hand side of \eqref{eq_gauss_37} is of $\mathcal{O}(\mu^{-(1-\sigma)})$. 
This shows that the integral in the left hand side of \eqref{eq_gauss_37}  goes to $0$ as $\mu\to \infty$, and this together with   \eqref{eq_gauss_36} establishes \eqref{eq_gauss_33}. 

Using \eqref{eq_gauss_32_arrang} for each of the factors $v_s^{(l)}\overline{w_s^{(l)}}$ in \eqref{eq_gaus_v_s_w_s_sum}, we get
\[
\lim_{\mu\to \infty}   \int_{\{x'_1\}\times M_0} v^{(l)}_s \overline{w^{(l)}_s} \psi dV_{g_0}=\int_{I_l} e^{-2\lambda t} e^{\Phi^{(1)}(x'_1,t)+\overline{\Phi^{(2)} (x'_1,t)}} \eta(x_1,t) \psi(t,0) dt.
\]
Summing over $I_l$ such that  $\gamma(t_l)=p_j$, we get \eqref{eq_prop_gaussian_3} when   $\supp(\psi) \subset V_j$, and hence, in general.

Finally let us check  \eqref{eq_prop_gaussian_4} for $\alpha\in C(M,T^*M)$,  $\psi\in C(M_0)$ and $x'_1\in \R$. Using a partition of unity, it is enough to verify \eqref{eq_prop_gaussian_4} in the following two cases: $\supp(\psi)\subset W_k$ and $\supp(\psi)\subset V_j$. Assume first that  $\supp(\psi)\subset W_k$. Using \eqref{eq_gauss_25_v_s_w_s}, and writing $z=(t,y)$, we get 
\begin{equation}
\label{eq_gauss_38}
\begin{aligned}
h\int_{\{x'_1\}\times M_0} \langle \alpha, dv_s\rangle_g\overline{w_s}\psi & dV_{g_0}=h \int_{\{x'_1\}\times M_0} \big(\alpha_1(\p_{x_1}v_s)+  g_0^{kj}\alpha_k (\p_{z_j}v_s)\big)\overline{w_s}\psi dV_{g_0}\\
=&i h (\mu+i\lambda) \int_{\{x'_1\}\times M_0} g_0^{kj}\alpha_k (\p_{z_j}\varphi) v_s\overline{w_s}\psi dV_{g_0}\\
&+h \int_{\{x'_1\}\times M_0} g_0^{kj}\alpha_k e^{is\varphi}\mu^{\frac{n-2}{4}}(\p_{z_j}a_0) \chi\bigg(\frac{y}{\delta'}\bigg)\overline{w_s}\psi dV_{g_0}\\
&+h \int_{\{x'_1\}\times M_0} g_0^{kj}\alpha_k e^{is\varphi}\mu^{\frac{n-2}{4}}a_0 \bigg(\p_{z_j} \chi\bigg(\frac{y}{\delta'}\bigg)\bigg)\overline{w_s}\psi dV_{g_0}\\
&+h \int_{\{x'_1\}\times M_0} \alpha_1 e^{is\varphi}\mu^{\frac{n-2}{4}}(\p_{x_1}a_0) \chi\bigg(\frac{y}{\delta'}\bigg)\overline{w_s}\psi dV_{g_0}.
\end{aligned}
\end{equation}
Let us first show that the second, third and fourth integrals in the right hand side of \eqref{eq_gauss_38} vanish in the limit as $h\to 0$.  For the second integral, we have 
\begin{align*}
h \bigg|\int_{\{x'_1\}\times M_0} g_0^{kj}\alpha_k& e^{is\varphi} \mu^{\frac{n-2}{4}}(\p_{z_j}a_0)  \chi\bigg(\frac{y}{\delta'}\bigg)\overline{w_s}\psi dV_{g_0}\bigg|\\
&\le \mathcal{O} (h^{1-\sigma})\| e^{is\varphi}\mu^{\frac{n-2}{4}}  \|_{L^2(\{|y|\le \delta'/2\})}\|w_s\|_{L^2(M_0)}=\mathcal{O}(h^{1-\sigma})\to 0,
\end{align*}
as $h\to 0$, since $0<\sigma<1/2$. The fourth integral is estimated similarly and bounding the third integral is even more straightforward. 

When computing the limit of the first  integral in the right hand side of \eqref{eq_gauss_38}, we may neglect the contribution containing $\lambda$ and we  only have to show that 
\begin{equation}
\label{eq_gauss_39_0_1}
\begin{aligned}
i h \mu  \int_{\{x'_1\}\times M_0} g_0^{kj}\alpha_k &(\p_{z_j}\varphi) v_s\overline{w_s}\psi dV_{g_0}\\
&\to \int_0^L i \alpha_t(x_1', t, 0) e^{-2\lambda t} e^{\Phi^{(1)}(x'_1,t)+\overline{\Phi^{(2)} (x'_1,t)}} \eta(x_1,t)\psi(t,0) dt,
\end{aligned}
\end{equation}
as $h\to 0$. To that end, we proceed as in \eqref{eq_gauss_26}. Using  \eqref{eq_gauss_3_metric}, \eqref{eq_gauss_9}, 
and that 
\[
\p_t \varphi=1+\mathcal{O}(|y|^2),\quad \nabla_{y}\varphi =\mathcal{O}(|y|),
\]
we obtain that 
\begin{equation}
\label{eq_gauss_39}
\begin{aligned}
  i   \int_{\{x'_1\}\times M_0} & g_0^{kj}\alpha_k(x_1', t, y)  (\p_{z_j}\varphi) v_s\overline{w_s}\psi dV_{g_0} \\
 = & i \int_{0}^L \int_{\R^{n-2}} g_0^{kj}\alpha_k(x_1', t, y)  (\p_{z_j}\varphi)  e^{-2\mu\text{Im}\varphi} e^{-2\lambda\text{Re}\varphi} \mu^{\frac{n-2}{2}} a_0(x'_1,t;\mu)\overline{b_0(x'_1,t;\mu)} \\
 &\chi^2\bigg(\frac{y}{\delta'}\bigg) \psi(t,y) |g_0|^{\frac{1}{2}} dt dy\\
 = &i \int_0^L \int_{\R^{n-2}} 
 (\alpha_t(x_1', t, y)+  \mathcal{O}(|y|))
 e^{-\mu\text{Im} H(t)y\cdot y}  \\
 &e^{-2\lambda t} e^{\lambda \mathcal{O}(|y|^2)} \mu^{\frac{n-2}{2}} a_0(x'_1,t;\mu)\overline{b_0(x'_1,t;\mu)}\chi^2\bigg(\frac{y}{\delta'}\bigg) \psi(t,y) dt dy.
\end{aligned}
\end{equation}
Performing the change of variables $\mu^{1/2} y=x$ in \eqref{eq_gauss_39} and passing to the limit $h=\frac{1}{\mu}\to 0$ by  means of the dominated convergence theorem, we get in view of \eqref{eq_normalization}, 
\begin{align*}
\lim_{h\to 0} & \ i   \int_{\{x'_1\}\times M_0} g_0^{kj}\alpha_k (\p_{z_j}\varphi) v_s\overline{w_s}\psi dV_{g_0}\\
&=\int_0^L i \alpha_t(x_1', t, 0) e^{-2\lambda t} e^{\Phi^{(1)}(x'_1,t)+\overline{\Phi^{(2)} (x'_1,t)}} \eta(x_1,t)\psi(t,0) dt.
\end{align*}
This concludes the proof of  \eqref{eq_gauss_39_0_1} and thus, of  \eqref{eq_prop_gaussian_4} when $\supp(\psi)\subset W_k$.

Assume now that $\supp(\psi)\subset V_j$.  In this case we write 
\[
v_s=\sum_{\gamma(t_l)=p_j} v_s^{(l)}, \quad w_s=\sum_{\gamma(t_l)=p_j} w_s^{(l)},
\]
on $\supp(\psi)$. 
Then we have 
\begin{align*}
h\int_{\{x'_1\}\times M_0} \langle \alpha,dv_s\rangle_g\overline{w_s}\psi dV_{g_0}= h \sum_{l:\gamma(t_l)=p_j} \int_{\{x'_1\}\times M_0} \langle \alpha,dv^{(l)}_s\rangle_g\overline{w^{(l)}_s}\psi dV_{g_0}\\
+h \sum_{l\ne l',  \gamma(t_l)=\gamma(t_{l'})=p_j} \int_{\{x'_1\}\times M_0} \langle \alpha,dv^{(l)}_s\rangle_g\overline{w^{(l')}_s}\psi dV_{g_0}.
\end{align*}
As before, see \eqref{eq_gauss_33}, we want to show that the mixed terms vanish in the limit as $h\to 0$, i.e. if $l\ne l'$,
\begin{equation}
\label{eq_gauss_40}
\lim_{h\to 0}  h \int_{\{x'_1\}\times M_0} \langle \alpha,dv^{(l)}_s\rangle_g\overline{w^{(l')}_s}\psi dV_{g_0}= 0.
\end{equation}
In view of \eqref{eq_gauss_38} we only have to check that 
\[
\lim_{h\to 0}   \int_{\{x'_1\}\times M_0} g_0^{kj}\alpha_k (\p_{z_j}\varphi^{(l)}) v_s^{(l)}\overline{w_s^{(l')}}\psi dV_{g_0}=0.
\]
This follows by repeating a non-stationary phase argument  as in the proof of \eqref{eq_gauss_35}. Hence,
\begin{align*}
\lim_{h\to 0}h\int_{\{x'_1\}\times M_0} \langle \alpha, & dv_s\rangle_g\overline{w_s}\psi dV_{g_0}\\
&=\sum_{l:\gamma(t_l)=p_j}  
\int_{I_l} i \alpha_t(x_1', t, 0) e^{-2\lambda t} e^{\Phi^{(1)}(x'_1,t)+\overline{\Phi^{(2)} (x'_1,t)}} \eta(x_1,t) \psi(t,0) dt,
\end{align*}
which completes the proof of  \eqref{eq_prop_gaussian_4} when $\supp(\psi)\subset V_j$ and hence, in general. 

The proof of \eqref{eq_prop_gaussian_5} is analogous to the proof of \eqref{eq_prop_gaussian_4}. This completes the proof of Proposition \ref{prop_Gaussian_beams}. 
\end{proof}

\section{Construction of complex geometric optics solutions based on Gaussian beam quasimodes}

\label{sec_CGO_based_quasi}

Let $(M,g)$ be a conformally transversally anisotropic manifold  so that $(M,g)\subset(\R\times M_0^0, c(e\oplus g_0))$ and let $A\in C^\infty(M, T^*M)$ and $q\in L^\infty(M,\C)$.

Let $\tilde g=e\oplus g_0$. Then for the magnetic Schr\"odinger operator $L_{g,A,q}$ defined in  \eqref{eq_int_1}, we have
\begin{equation}
\label{eq_CGO_1}
c^{\frac{n+2}{4}} \circ L_{g,A,q} \circ c^{-\frac{(n-2)}{4}} = L_{\tilde g,A,\tilde q}.
\end{equation}
where 
\[
\tilde q= c\big(q - c^{\frac{n-2}{4}}\Delta_g(c^{-\frac{(n-2)}{4}})\big).
\]
Indeed, first recall from \cite{DKuLS_2016} that 
\begin{equation}
\label{eq_CGO_2}
c^{\frac{n+2}{4}}(-\Delta_g)(c^{-\frac{(n-2)}{4}}\tilde u)=-\Delta_{\tilde g}\tilde u-\big(c^{\frac{n+2}{4}}\Delta_g(c^{-\frac{(n-2)}{4}})\big)\tilde u.
\end{equation}
Using that
\[
|g|=c^n|\tilde g|, \quad g^{ij}=c^{-1}\tilde g^{ij},
\]
we get 
\begin{equation}
\label{eq_CGO_3}
c^{\frac{n+2}{4}} i d_g^*(A c^{-\frac{(n-2)}{4}}\tilde u)=i d^*_{\tilde g}(A\tilde u)-i \bigg(\frac{n}{4}-\frac{1}{2}\bigg)c^{-1}\tilde u \langle A, dc \rangle_{\tilde g},
\end{equation}
\begin{equation}
\label{eq_CGO_4}
-i c^{\frac{n+2}{4}}  \langle A,d (c^{-\frac{(n-2)}{4}}\tilde u)\rangle_{g}=-i \langle A, d\tilde u\rangle_{\tilde g}+  i \bigg(\frac{n}{4}-\frac{1}{2}\bigg)c^{-1}\tilde u \langle A, dc\rangle_{\tilde g},
\end{equation}
and 
\begin{equation}
\label{eq_CGO_5}
c^{\frac{n+2}{4}} (\langle A, A\rangle_g^2+q)(c^{-\frac{(n-2)}{4}}\tilde u)= (\langle A, A\rangle_{\tilde g}^2 + cq)\tilde u. 
\end{equation}
Thus, \eqref{eq_CGO_1} follows from \eqref{eq_CGO_2},  \eqref{eq_CGO_3},  \eqref{eq_CGO_4} and  \eqref{eq_CGO_5}.

Let us write  $x=(x_1,x')$ for local coordinates in $\R\times M_0$, and let
\[
s=\mu+i\lambda,\quad 1\le \mu=\frac{1}{h},\quad\lambda\in \R, \quad \lambda \text{ fixed}. 
\]

We are interested in finding complex geometric optics solution to the equation 
\begin{equation}
\label{eq_CGO_6}
L_{\tilde g,A,\tilde q} \tilde u=0\quad \text{in}\quad M,
\end{equation}
having the form
\[
\tilde u=e^{-sx_1}(v+r),
\]
where $v=v_s$ is the Gaussian beam quasimode given in Proposition \ref{prop_Gaussian_beams}, and $r=r_s$ is a correction term.  
Thus,  $\tilde u$  is a solution of \eqref{eq_CGO_6}
 provided that 
 \begin{equation}
 \label{eq_cgo_conjug}
 e^{\frac{x_1}{h}}h^2 L_{\tilde g,A,\tilde q} e^{-\frac{x_1}{h}}( e^{-i\lambda x_1}r)=- e^{-i\lambda x_1} e^{sx_1} h^2 L_{\tilde g,A,\tilde q} e^{-sx_1}v.
 \end{equation}

By Proposition \ref{prop_solvability} and \eqref{eq_prop_gaussian_1}, for all $h>0$ small enough, there is  $r\in H^1(M^0)$ such that \eqref{eq_cgo_conjug} holds and 
$\|r\|_{H^1_{\text{scl}}(M^0)}=o(1)$ as $h\to 0$. To summarize, we have the following result.

\begin{prop}
\label{prop_CGO_general_mnfld}
Let $A^{(1)}, A^{(2)}\in C(M,T^*M)$ and $q^{(1)},q^{(2)}\in L^\infty(M,\C)$.  Let $s=\frac{1}{h}+i\lambda$ with $\lambda\in \R$  being fixed. For all $h>0$ small enough, there is a solution $u_1\in H^1(M^0)$ of $L_{g,A^{(1)},q^{(1)}} u_1=0$ in $\mathcal{D}'(M^0)$ having the form
\[
u_1=e^{-sx_1} c^{-\frac{(n-2)}{4}} (v_s+r_1),
\]
where $v_s\in C^\infty(M)$ is the Gaussian beam quasimode given in Proposition \ref{prop_Gaussian_beams}, and $r_1\in H^1(M^0)$ is such that $\|r_1\|_{H^1_{\text{scl}}(M^0)}=o(1)$ as $h\to 0$.

Similarly, for all $h>0$ small enough, there is a solution $u_2\in H^1(M^0)$ of $L_{g,\overline{A^{(2)}},\overline{q^{(2)}}} u_2=0$ in $\mathcal{D}'(M^0)$ having the form
\[
u_2=e^{sx_1} c^{-\frac{(n-2)}{4}} (w_s+r_2),
\]
where $w_s\in C^\infty(M)$ is the Gaussian beam quasimode given in Proposition \ref{prop_Gaussian_beams}, and $r_2\in H^1(M^0)$ is such that $\|r_2\|_{H^1_{\text{scl}}(M^0)}=o(1)$ as $h\to 0$.
\end{prop}

\section{Determining the magnetic field in Theorem \ref{thm_main_2}}

\label{sec_thm_main_2}

First by the boundary reconstruction result of Proposition \ref{prop_boundary_rec}, we have $\langle  A^{(1)}(x_0)- A^{(2)}(x_0), \tau\rangle=0$ for all $x_0\in \p M$ and all $\tau\in T_{x_0}\p M$.  Furthermore, using a partition of unity, flattening the boundary, and applying \cite[Theorem 1.3.3]{Hormander_book_1}, we conclude that there exists $\psi\in C^1(M)$ such that $\psi|_{\p M}=0$ and $d\psi= A^{(1)}-A^{(2)}$ on $\p M$. By Lemma \ref{lem_gauge}, we get 
\[
C_{g, A^{(1)},q^{(1)}} = C_{g, A^{(2)},q^{(2)}}= C_{g, A^{(2)}+d\psi,q^{(2)}},
\]
and replacing $A^{(2)}$ by $A^{(2)}+d\psi$, we see that in what follows we may assume that $A^{(1)}=A^{(2)}$ on $\p M$.  We may therefore extend $\tilde A=A^{(1)} - A^{(2)}$ by zero to the complement of $M$ in $\R \times M_0$, so that the extension $\tilde A$ is continuous. 

Our next point is the integral identify \eqref{eq_proof_1}. Here, similarly to the proof of Theorem \ref{thm_main}, we would like to substitute the complex geometric optics solutions of Proposition \ref{prop_CGO_general_mnfld}, multiply by $h$ and pass to the limit $h\to 0$.  In the computations below we shall use the following consequences of Proposition \ref{prop_Gaussian_beams}
and Proposition \ref{prop_CGO_general_mnfld}, 
\begin{equation}
\label{eq_recovery_1}
\|v_s\|_{L^2(M)}=\mathcal{O}(1),\quad  \|w_s\|_{L^2(M)}=\mathcal{O}(1),
\end{equation}
\begin{equation}
\label{eq_recovery_1_1}
\|d v_s\|_{L^2(M)}=\mathcal{O}(h^{-1}),\quad  \|d w_s\|_{L^2(M)}=\mathcal{O}(h^{-1}),
\end{equation}
\begin{equation}
\label{eq_recovery_2}
\|r_1\|_{L^2(M)}=o(1), \quad \|r_1\|_{L^2(M)}=o(1), 
\end{equation}
\begin{equation}
\label{eq_recovery_3}
\|dr_1\|_{L^2(M)}=o(h^{-1}), \quad \|dr_2\|_{L^2(M)}=o(h^{-1}), 
\end{equation}
as $h\to 0$.

First,  we have
\begin{align*}
u_1\overline{u_2}=e^{-2 i \lambda x_1} c^{-\frac{(n-2)}{2}} (v_s \overline{w_s} +v_s \overline{r_2}+ \overline{w_s} r_1+r_1\overline{r_2}),
\end{align*}
and therefore, using that \eqref{eq_recovery_1}, and \eqref{eq_recovery_2}, 
we see that 
\[
h \bigg| \int_M  (\langle A^{(1)}, A^{(1)}\rangle_g- \langle A^{(2)}, A^{(2)}\rangle_g+q^{(1)}-q^{(2)})u_1\overline{u_2}dV_g\bigg|=\mathcal{O}(h), \quad h\to 0. 
\]

We get 
\begin{align*}
u_1d\overline{u_2}&-\overline{u_2}du_1 =e^{-2 i \lambda x_1} c^{-\frac{(n-2)}{2}} \\
&\big[
(v_s+r_1)(d\overline{w_s}+d\overline{r_2})- (dv_s+dr_1)(\overline{w_s}+\overline{r_2})
+\frac{2}{h}(v_s+r_1)(\overline{w_s}+\overline{r_2})dx_1
\big].
\end{align*}
Using \eqref{eq_recovery_1}, \eqref{eq_recovery_2}, we obtain that 
\[
h \bigg| \frac{2}{h}i \int_{M} \langle \tilde A, dx_1\rangle_g  [v_s \overline{r_2} +\overline{w_s} r_1+r_1\overline{r_2}] e^{-2i\lambda x_1}c^{-\frac{(n-2)}{2}} dV_g \bigg|=o(1),\quad h\to 0.
\] 

Using \eqref{eq_recovery_1_1} and \eqref{eq_recovery_2}, we conclude that 
\begin{align*}
h \bigg| \int_M [\langle \tilde A, d\overline{w_s} \rangle_g r_1 -\langle \tilde A, dv_s \rangle_g \overline{r_2}]  e^{-2i\lambda x_1}c^{-\frac{(n-2)}{2}} dV_g \bigg|=o(1),\quad h\to 0.
\end{align*}

Using \eqref{eq_recovery_1}, \eqref{eq_recovery_2}, and \eqref{eq_recovery_3}, we get 
\begin{align*}
h \bigg| \int_M [ \langle \tilde A, d\overline{r_2} \rangle_g (v_s + r_1) -\langle \tilde A, dr_1\rangle_g (\overline{w_s}+\overline{r_2}) ]  e^{-2i\lambda x_1}c^{-\frac{(n-2)}{2}} dV_g \bigg|=o(1), 
\end{align*}
as $ h\to 0$. 

Now using that $g=c(e\oplus g_0)$, we get $\langle \tilde A, dx_1\rangle_{g}=c^{-1}\tilde A_1$. Writing $x=(x_1,x')$, $x'\in M_0$, using the fact that $\tilde A=0$ outside of $M$,  Fubini's theorem, \eqref{eq_prop_gaussian_3}, and the dominated convergence theorem,  we obtain that   
\begin{align*}
2i \int_M \langle  \tilde A,& dx_1  \rangle_g  v_s  \overline{w_s}  e^{-2i\lambda x_1}  c^{-\frac{(n-2)}{2}} dV_g=2i \int_{\R} e^{-2i\lambda x_1} \bigg(\int_{M_0}
 \tilde A_1(x_1,x') v_s \overline{w_s}
dV_{g_0}\bigg)dx_1\\
&\to 2i \int_{\R} e^{-2i\lambda x_1}  \int_0^L e^{-2\lambda t} e^{\Phi^{(1)}(x_1,t)+\overline{\Phi^{(2)} (x_1,t)}} \eta(x_1,t)\tilde A_1(x_1,\gamma(t)) dt dx_1,
\end{align*}
where $\Phi^{(1)}, \Phi^{(2)}\in C(\R\times [0,L])$ satisfy the following transport equations,
\begin{equation}
\label{eq_recov_trans_tr_1}
 (\p_{x_1}-i\p_t) \Phi^{(1)}=- i A^{(1)}_1(x_1,\gamma(t))-A^{(1)}_t(x_1,\gamma(t)),
\end{equation}
\begin{equation}
\label{eq_recov_trans_tr_2}
 (\p_{x_1}+i\p_t) \Phi^{(2)}=-i \overline{A^{(2)}_1(x_1,\gamma(t))}+\overline{A^{(2)}_t (x_1,\gamma(t))}.
\end{equation}

Setting $\tilde g=e\oplus g_0$, we see that 
\[
\langle \tilde A, d\overline{w_s}\rangle_g=c^{-1}\langle \tilde A, d\overline{w_s}\rangle_{\tilde g}. 
\]
Using this formula, the fact that $\tilde A=0$ outside of $M$,  Fubini's theorem, \eqref{eq_prop_gaussian_5}, and the dominated convergence theorem, we get 
\begin{align*}
hi \int_M \langle \tilde A,  d\overline{w_s}  & \rangle_{g} v_s e^{-2i \lambda x_1}c^{-\frac{(n-2)}{2}}dV_g= i h \int_{\R} e^{-2i \lambda x_1} 
\bigg(\int_{M_0} \langle \tilde A, d\overline{w_s}\rangle_{\tilde g} v_s dV_{g_0}
\bigg)dx_1\\
&\to - i \int_{\R} e^{-2i \lambda x_1}  \int_0^L  i \tilde A_t(x_1,\gamma(t))  e^{-2\lambda t} e^{\Phi^{(1)}(x_1,t)+\overline{\Phi^{(2)} (x_1,t)}} \eta(x_1,t) dt dx_1.
\end{align*}
Similarly, using \eqref{eq_prop_gaussian_4}, we have
\begin{align*}
-hi \int_M \langle \tilde A, & d v_s \rangle_{g} \overline{w_s}  e^{-2i \lambda x_1}c^{-\frac{(n-2)}{2}}dV_g \\
&\to - i \int_{\R} e^{-2i \lambda x_1}  \int_0^L  i \tilde A_t(x_1,\gamma(t))  e^{-2\lambda t} e^{\Phi^{(1)}(x_1,t)+\overline{\Phi^{(2)} (x_1,t)}} \eta(x_1,t) dt dx_1.
\end{align*}

Hence, substituting complex geometric optics solutions of Proposition \ref{prop_CGO_general_mnfld} into the integral identify \eqref{eq_proof_1}, we obtain that 
\begin{equation}
\label{eq_recov_mag_1}
 \int_{\R}  \int_0^L   e^{-2i \lambda x_1}  (\tilde A_1(x_1,\gamma(t))-  i \tilde A_t(x_1,\gamma(t)))  e^{-2\lambda t} e^{\Phi^{(1)}(x_1,t)+\overline{\Phi^{(2)} (x_1,t)}} \eta(x_1,t) dt dx_1=0.
\end{equation}  
Here the domain of integration can be replaced by a bounded simply connected open set $\Omega\subset \C$ with smooth boundary containing the support of $\tilde A$. 
Now using \eqref{eq_recov_trans_tr_1} and \eqref{eq_recov_trans_tr_2} and writing $z=x_1+it$,  we see that 
\[
\p (\Phi^{(1)} + \overline{\Phi^{(2)}})=-\frac{i}{2} (\tilde A_1(x_1,\gamma(t))-i \tilde A_t(x_1,\gamma(t))). 
\]
It follows from \eqref{eq_recov_mag_1} that 
\[
 \int_\Omega    \p \big(\tilde\eta (x_1,t )e^{\Phi^{(1)}(x_1,t)+\overline{\Phi^{(2)} (x_1,t)}}\big)  dt dx_1=0,
\]
where
\[
\tilde \eta(x_1,t)= e^{-2i \lambda (x_1-it)} \eta(x_1,t), \quad \p \tilde \eta=0. 
\]
Repeating the arguments leading from \eqref{eq_proof_10} to 
\eqref{eq_proof_12}, we get
\begin{equation}
\label{eq_recov_mag_2}
\int_{\R}  \int_0^L e^{-2i \lambda x_1-2\lambda t}  (\tilde A_1(x_1,\gamma(t))-  i \tilde A_t(x_1,\gamma(t)))  dt dx_1=0.
\end{equation}
Letting 
\begin{align*}
&f(x',\lambda)= \int e^{-i\lambda x_1} \tilde A_1(x_1,x')dx_1, \quad x'\in M_0,\\
&\alpha(x',\lambda)=\sum_{j=2}^n \bigg( \int e^{-i\lambda x_1} \tilde A_j(x_1,x')dx_1 \bigg) dx_j,
\end{align*}
we have $f\in C(M_0)$, $\alpha\in C(M_0,T^*M_0)$ and  
we conclude from \eqref{eq_recov_mag_2}, replacing $2\lambda$ by $\lambda$, that 
\begin{equation}
\label{eq_recov_mag_3}
\int_0^L \big[ f(\gamma(t),\lambda) -i \alpha (\dot{\gamma}(t),\lambda) \big] e^{-\lambda t}dt=0,
\end{equation}
along any unit speed non-tangential geodesic $\gamma:[0,L]\to M_0$ on $M_0$, and 
for any $\lambda\in \R$.  Following \cite{DKuLS_2016} and \cite{Cekic}, we shall use the geodesic transform and proceed as follows. 
First,  evaluating \eqref{eq_recov_mag_3} at $\lambda=0$ and using the injectivity of  the unattenuated geodesic transform, we get 
\begin{equation}
\label{eq_recov_mag_4_-2}
f(x',0)=0, \quad \alpha(x',0)=idp_0(x'),
\end{equation}
for some $p_0\in C^1(M_0)$ such that $p_0|_{\p M_0}=0$. Next differentiating  \eqref{eq_recov_mag_3} with respect to $\lambda$ and letting $\lambda=0$, we get
\begin{equation}
\label{eq_recov_mag_4}
\int_0^L \big[ \p_\lambda f(\gamma(t),0) -i \p_\lambda\alpha (\dot{\gamma}(t),0) +i t \alpha(\dot{\gamma}(t),0)  \big] dt=0.
\end{equation}
Using that 
\begin{equation}
\label{eq_recov_mag_4+1}
\alpha(\dot{\gamma}(t),0)=idp_0(\dot{\gamma}(t))=i \frac{d}{dt}p_0(\gamma(t)),
\end{equation}
and integrating by parts in \eqref{eq_recov_mag_4}, we see that 
\[
\int_0^L \big[ \p_\lambda f(\gamma(t),0) +p_0(\gamma(t)) -i \p_\lambda\alpha (\dot{\gamma}(t),0) \big] dt=0.
\]
The injectivity of  the geodesic transform implies that 
\begin{equation}
\label{eq_recov_mag_5-1}
\p_\lambda f(x',0) +p_0(x')=0, \quad \p_\lambda \alpha (x',0)=i dp_1(x'),  
\end{equation}
for some $p_1\in C^1(M_0)$ such that $p_1|_{\p M_0}=0$. Differentiating  \eqref{eq_recov_mag_3} twice with respect to $\lambda$, we obtain that 
\begin{equation}
\label{eq_recov_mag_5}
\begin{aligned}
\int_0^L \big[ t^2 \big(f(\gamma(t),\lambda) -i \alpha (\dot{\gamma}(t),\lambda)\big)
- 2t \p_\lambda \big(f(\gamma(t),\lambda) -i \alpha (\dot{\gamma}(t),\lambda)\big)\\ 
+ 
\p^2_{\lambda }\big( f(\gamma(t),\lambda) -i  \alpha (\dot{\gamma}(t),\lambda) \big) \big] e^{-\lambda t}dt=0.
\end{aligned}
\end{equation}
By \eqref{eq_recov_mag_4+1} and integration by parts, we have
\begin{equation}
\label{eq_recov_mag_6}
-\int_0^L i t^2 \alpha (\dot{\gamma}(t),0)dt= \int_0^L  t^2 \frac{d}{dt}p_0(\gamma(t)) dt= -\int_0^L 2t p_0(\gamma(t))dt,
\end{equation}
and similarly, 
\begin{equation}
\label{eq_recov_mag_7}
\int_0^L 2 i t \p_\lambda \alpha (\dot{\gamma}(t),0)dt=-2\int_0^L t \frac{d}{dt}p_1(\gamma(t)) dt= 2 \int_0^L  p_1(\gamma(t))dt.
\end{equation}
Letting $\lambda=0$ in \eqref{eq_recov_mag_5} and using \eqref{eq_recov_mag_4_-2}, \eqref{eq_recov_mag_5-1}, \eqref{eq_recov_mag_6}, \eqref{eq_recov_mag_7}, we get 
\[
\int_0^L \big[ \p^2_\lambda f(\gamma(t),0)+2p_1(\gamma(t))-i \p_\lambda^2 \alpha(\dot{\gamma}(t),0) \big]dt=0,
\]
and by the injectivity of the geodesic transform, we have
\[
 \p^2_\lambda f(x',0)+2p_1(x')=0,\quad \p_\lambda^2 \alpha(x',0)=idp_2(x')
\]
 some $p_2\in C^1(M_0)$ such that $p_2|_{\p M_0}=0$.   Proceeding further by induction as in \cite{Cekic}, we conclude that
 \begin{equation}
\label{eq_recov_mag_8}
 \p^l_\lambda f(x',0)+lp_{l-1}(x')=0,\quad \p_\lambda^l \alpha(x',0)=idp_l(x'), \quad l=0,1,2,\dots,
 \end{equation}
for some $p_l\in C^1(M_0)$ such that $p_l|_{\p M_0}=0$. Here $p_{-1}=0$. 

Viewing $\tilde A_j\in C_0(\R_{x_1}, L^2(M_0))$, we see that $\mathcal{F}_{x_1\to \lambda} (\tilde A(x_1,x'))\in \text{Hol}(\C, L^2(M_0))$. 
Using \eqref{eq_recov_mag_8}, we get
\begin{align*}
\p_{\lambda}^l|_{\lambda=0} \mathcal{F}_{x_1\to \lambda} (\p_{x_j}\tilde A_k-\p_{x_k}\tilde A_j)&=\p_{x_j}\p_{\lambda}^l|_{\lambda=0}\alpha_k-
\p_{x_k}\p_{\lambda}^l|_{\lambda=0}\alpha_j\\
&=i (\p_{x_j}\p_{x_k}p_l-\p_{x_k}\p_{x_j}p_l)=0,\quad  j,k=2,\dots, n,
\end{align*}
and therefore,
\[
\p_{x_j}\tilde A_k-\p_{x_k}\tilde A_j=0, \quad  j,k=2,\dots, n.
\]
Also by \eqref{eq_recov_mag_8}, we obtain that 
\begin{align*}
\p_\lambda^l |_{\lambda=0}  \mathcal{F}_{x_1\to \lambda} (\p_{x_j}\tilde A_1-\p_{x_1}\tilde A_j)=\p_{x_j}\p_\lambda^l|_{\lambda=0}f - \p_\lambda^l|_{\lambda=0}(i\lambda \alpha_j)\\
=-l\p_{x_j}p_{l-1}-il \p_\lambda^{l-1}|_{\lambda=0}\alpha_j=0,
\end{align*}
thus, $\p_{x_j}\tilde A_1-\p_{x_1}\tilde A_j=0$. Hence, $d\tilde A=0$ in $M$, and therefore, $dA^{(1)}=dA^{(2)}$ in $M$.

\section{Determining the holonomy and completing the proof of Theorem \ref{thm_main_2}} 

\label{sec_hol}

Throughout this section we assume that $(M,g)$ is a smooth compact Riemannian manifold with smooth boundary $\p M$.  Let $A^{(1)}, A^{(2)}\in C(M,T^*M)$  and $q^{(1)}=q^{(2)}=q\in L^\infty(M,\C)$. We set 
$\tilde A=A^{(1)}-A^{(2)}$.  Let $\nabla^{\tilde A}=d+i\tilde A$ be the connection on the trivial line bundle $M\times \C$ over $M$ associated to $\tilde A$.  

Our starting point is the fact that  $d\tilde A=0$, which can be viewed as the statement that the curvature of the connection $\nabla^{\tilde A}$ vanishes, see \cite{Paternain}.

Let us recall some definitions and facts following \cite{Paternain} and \cite[Section 6]{Guillarmou_Zhou}.  Given a $C^1$ curve $\gamma:[a,b]\to M$, the parallel transport along $\gamma$ is obtained by solving the initial value problem for the linear ODE,
\begin{equation}
\label{eq_holonomy_1}
\begin{cases} \dot{s}(t)+i \tilde A(\gamma(t), \dot{\gamma}(t))s(t)=0,\\
s(a)=s_0\in \C.
\end{cases}
\end{equation}
We observe that  for $\tilde A\in C(M,T^*M)$,  problem \eqref{eq_holonomy_1} has a unique solution.  Associated to \eqref{eq_holonomy_1}, we introduce the linear map $P_\gamma^{\tilde A}:\C\to \C$ defined by  $P_\gamma^{\tilde A}(s_0)=s(b)$. We have
\[
P^{\tilde A}_\gamma=e^{-i \int_\gamma \tilde A}. 
\] 
The holonomy group of the connection $\nabla^{\tilde A}$ at a point $m\in M$ is given by
 \[
 H_m(\nabla^{\tilde A})=\{P_\gamma^{\tilde A}\in \C\setminus\{0\}: \gamma \text{ is a loop based at } m\}.
 \] 
Let $\pi_1(M,m)$ be the fundamental group of $M$ at $m$, i.e. the set of all loops based at $m$ up to homotopy equivalence. Using the fact that the curvature  $d\tilde A=0$, we see that the map $\gamma\mapsto P_\gamma^{\tilde A}$ gives rise to a natural group homomorphism 
\[
\rho_m^{\tilde A}:\pi_1(M, m) \to H_m(\nabla^{\tilde A}).
\]
The homomorphism $\rho_m^{\tilde A}$ is called the holonomy representation of $\tilde A$ into $\C\setminus\{0\}$,  and $\rho_m^{\tilde A}$ is trivial if and only if 
\begin{equation}
\label{eq_holonomy_1_0}
e^{-i \int_\gamma \tilde A}=1 \Longleftrightarrow \int_{\gamma}\tilde A\in 2\pi\Z,
\end{equation}
for all loops $\gamma$ based at $m$.  If $M$ is connected, the fact that the fundamental groups $\pi_1(M, m)$ are isomorphic for different points $m$ upon conjugation by an appropriate curve implies that  the  condition \eqref{eq_holonomy_1_0} is independent of $m$. 

The main result of this section is as follows. It can be viewed as an analog of   \cite[Theorem 6.1]{Guillarmou_Zhou} and \cite[Theorem 6.3]{Cekic} for magnetic potentials which are merely continuous.  

\begin{prop}
\label{prop_holonomy}
Assume that $A^{(1)}, A^{(2)}\in C(M,T^*M)$ and $q\in L^\infty(M,\C)$. Set 
$\tilde A=A^{(1)}-A^{(2)}$. If $C_{g,A^{(1)},q}=C_{g,A^{(2)},q}$ and $d\tilde A=0$ then  the holonomy representation $\rho_m^{\tilde A}$ is trivial for every $m\in M$. 
\end{prop}

Assuming that Proposition \ref{prop_holonomy} has been proved, let us complete the proof of Theorem \ref{thm_main_2}.

\textbf{Proof of Theorem \ref{thm_main_2}}.  Let $M_1,\dots, M_N$ be the connected components of the compact manifold $M$, and let $m_j\in \p M_j\subset \p M$.  
Let $F\in C^1(M, \C)$ be  given by
\begin{equation}
\label{eq_holonomy_2}
F(m')=e^{i\int_{\gamma_{m_jm'}}\tilde A}, \quad\text{if}\quad m'\in M_j,
\end{equation}
where $\gamma_{m_jm'}$ is a $C^\infty$ path joining $m_j$ and $m'$ in $M_j$.  The function $F$ is well defined in the sense that it does not depend on a choice of the path between $m_j$ and $m'$,  since the holonomy representation $\rho_{m_j}^{\tilde A}$ is trivial.  

As in the beginning of Section \ref{sec_thm_main_2}, without loss of generality, we may assume  that $A^{(1)}=A^{(2)}$ on $\p M$. This together with \eqref{eq_holonomy_2} implies that $F=1$ on $\p M$. 
In view of \eqref{eq_holonomy_2} we see that 
\begin{align*}
F^{-1}\circ d_{A^{(2)}}\circ F=d_{A^{(1)}},\quad F^{-1}\circ d^*_{\overline{A^{(2)}}}\circ F= d^*_{\overline{A^{(1)}}},
\end{align*}
and thus, 
\[
F^{-1}\circ L_{g,A^{(2)},q}\circ F=  L_{g,A^{(1)},q}.
\]
This completes the proof of  Theorem \ref{thm_main_2}.

\textbf{Proof of Proposition \ref{prop_holonomy}.} If suffices to check that the holonomy representation $\rho^{\tilde A}_m$ is trivial for $m\in \p M$.   Let $\gamma$ be a loop at $m\in \p M$ and let us show that 
\[
\int_\gamma \tilde A\in 2\pi\Z. 
\]
In doing so we shall follow \cite{Guillarmou_Zhou},  \cite{Cekic}, and replace $\gamma$ by a homotopically equivalent loop $\tilde \gamma: [0,2]\to M$ such that $\tilde \gamma(0)=m$, $\tilde \gamma(1)=\tilde m\in \p M$, $\tilde m\ne m$, $\gamma_1:=\tilde \gamma((0,1))\subset M^0$ and $\gamma_2:=\tilde \gamma([1,2])\subset \p M$. As explained  in \cite{Cekic}, we may assume that $\gamma_1$ and $\gamma_2$
are embedded curves. Using that the tangential component of  $\tilde A$ vanishes along $\p M$, we have 
\begin{equation}
\label{eq_holonomy_3}
\int_\gamma \tilde A=\int_{\gamma_1} \tilde A.
\end{equation}

Let $U=\{x\in M^0: \dist(x,\gamma_1)<\varepsilon \}$, $\varepsilon>0$ small, be a tubular neighborhood of $\gamma_1$. The set $U$ is diffeomorphic to $(0,1)\times B(0,\varepsilon)$, where $B(0,\varepsilon)$ is an open $(n-1)$-dimensional ball centered at $0$ of radius $\varepsilon$, and therefore, $U$ is simply connected.   When $x\in U$, let 
\begin{equation}
\label{eq_holonomy_3_1}
\varphi(x)=\int_m^x \tilde A\in C^1(\overline{U}). 
\end{equation}
The function $\varphi$ is well defined since $d\tilde A=0$ and $U$ is simply connected.  
We have $d\varphi=\tilde A=A^{(1)}-A^{(2)}$, and therefore,
\begin{equation}
\label{eq_holonomy_4}
e^{-i\varphi}\circ L_{g,A^{(2)},q}\circ e^{i\varphi}=L_{g, A^{(1)},q} \quad \text{in} \quad U. 
\end{equation}
Let $U_1=\{x\in \p M: d(x,m)<\varepsilon\}$. Then $\varphi|_{U_1}=0$. 

Let $f\in C^\infty(\p M)$ be such that $f(\tilde m)\ne 0$, and let $u_1\in H^1(M^0)$ be the solution of the following problem
\begin{equation}
\label{eq_holonomy_5}
\begin{aligned}
L_{g,A^{(1)},q}u_1&=0,\quad \text{in}\quad \mathcal{D}'(M^0),\\
u_1|_{\p M}&=f.
\end{aligned}
\end{equation}
As $C_{g,A^{(1)},q}=C_{g, A^{(2)},q}$, we conclude that there exists $u_2\in H^1(M^0)$ such that 
\begin{equation}
\label{eq_holonomy_6}
\begin{aligned}
L_{g,A^{(2)},q}u_2&=0,\quad \text{in}\quad \mathcal{D}'(M^0),\\
u_2|_{\p M}&=f,
\end{aligned}
\end{equation}
and 
\begin{equation}
\label{eq_holonomy_6_1}
\langle d_{A^{(1)}}u_1,\nu \rangle_g|_{\p M}= \langle d_{A^{(2)}}u_2,\nu \rangle_g|_{\p M}.
\end{equation}
 Letting 
\[
v=e^{-i\varphi}u_2\in H^1(U),
\]
and using \eqref{eq_holonomy_4} and \eqref{eq_holonomy_6}, we see that 
\[
L_{g,A^{(1)},q}v=0 \quad \text{in}\quad \mathcal{D}'(U),
\]
and therefore, by \eqref{eq_holonomy_5}, we get 
\begin{equation}
\label{eq_holonomy_7}
L_{g,A^{(1)},q}(v-u_1)=0 \quad \text{in}\quad \mathcal{D}'(U).
\end{equation}
We also obtain that  
\begin{equation}
\label{eq_holonomy_8}
(v-u_1)|_{U_1}=0, \quad \langle d_{A^{(1)}}(v-u_1),\nu \rangle_g|_{U_1}=0.
\end{equation}
Indeed, let $\chi\in H^{1/2}(\p M)$, $\supp(\chi)\subset U_1$,  and using \eqref{eq_trace_normal_int}, \eqref{eq_holonomy_6_1}, we get 
\begin{align*}
 \langle  \langle d_{A^{(1)}}v,\nu \rangle_g, \chi\rangle_{H^{-\frac{1}{2}}(\p M)\times H^{\frac{1}{2}}(\p M)}
 =\langle  \langle d_{A^{(2)}+ d\varphi} (e^{-i\varphi}u_2),\nu \rangle_g,  e^{i\varphi}\chi\rangle_{H^{-\frac{1}{2}}(\p M)\times H^{\frac{1}{2}}(\p M)}\\
 =\langle  \langle d_{A^{(2)}} u_2,\nu \rangle_g, \chi\rangle_{H^{-\frac{1}{2}}(\p M)\times H^{\frac{1}{2}}(\p M)}=
 \langle  \langle d_{A^{(1)}}u_1,\nu \rangle_g, \chi\rangle_{H^{-\frac{1}{2}}(\p M)\times H^{\frac{1}{2}}(\p M)},
\end{align*}
showing the second equality in \eqref{eq_holonomy_8}.

Now in view of \eqref{eq_holonomy_7} and \eqref{eq_holonomy_8}, by the unique continuation stated in  Proposition \ref{prop_unique} below, we conclude that $u_1=v=e^{-i\varphi}u_2$ in $U$.  Letting $U_2=\{x\in \p M: d(x,\tilde m)<\varepsilon\}$, we get the equality of the traces,
\[
u_1=e^{-i\varphi}u_2 \quad\text{in}\quad  H^{\frac{1}{2}}(U_2),
\]
and therefore, $f=e^{-i\varphi}f$ in $U_2$. As $f\in C^\infty(\p M)$ and $f(\tilde m)\ne 0$, we obtain that $e^{-i\varphi(\tilde m)}=1$.  Hence, in view of   \eqref{eq_holonomy_3_1}, 
\[
\varphi(\tilde m)=\int_{\gamma_1} \tilde A\in 2\pi\Z. 
\]
This together with \eqref{eq_holonomy_3} completes the proof of Proposition \ref{prop_holonomy}.

In the proof of Proposition \ref{prop_holonomy} we have used the following unique continuation result which is a direct consequence of a simplified version of \cite[Theorem 1.1]{Koch_Tataru}.

\begin{prop}
 \label{prop_unique} 
Let $(M,g)$ be a  smooth compact Riemannian manifold of dimension $n\ge 3$ with smooth boundary $\p M$,  let $U\subset M^0$ be an open connected set  and let $\Gamma\subset \p U\cap \p M$ be an open non-empty set of class $C^\infty$. Let  $A\in L^\infty(M,T^*M)$, $q\in L^\infty (M,\C)$, and let $u\in H^1(U)$ be such that 
\begin{align*}
&L_{g,A,q} u=0\quad \text{in}\quad \mathcal{D}'(U),\\
&u|_{\Gamma}=0,\quad \langle d_{A}u,\nu \rangle_g|_{\Gamma}=0.
\end{align*}
Then $u=0$ in $U$. 
\end{prop}

\begin{proof}
Let us isometrically embed the manifold $(M,g)$ into a larger closed manifold $(\tilde M,g)$ of the same dimension.  Let $\tilde U\subset \tilde M$ be  an open connected set such that  $\tilde U\supset U$,  $\p U\setminus\Gamma\subset \p \tilde U$ and  $\tilde U\setminus \overline{U}$ is non-empty. We extend $A$ and $q$ to $\tilde U$ so that $A\in L^\infty(\tilde U, T^*\tilde U)$ and $q\in L^\infty(\tilde U, \C)$. Let 
\begin{equation}
\label{eq_holonomy_9}
\tilde u=\begin{cases} u & \text{in}\quad U,\\
0 & \text{in}\quad  \tilde U\setminus\overline{U}.
\end{cases}
\end{equation}
Since $u|_{\Gamma}=0$, we have $u\in H^1(\tilde U)$ and 
\begin{equation}
\label{eq_holonomy_10}
d \tilde u=\begin{cases} du & \text{in}\quad U,\\
0 & \text{in}\quad  \tilde U\setminus\overline{U}.
\end{cases}
\end{equation}
Let us check that 
\begin{equation}
\label{eq_holonomy_11}
L_{g,A,q}\tilde u=0\quad \text{in}\quad  \mathcal{D}'(\tilde U).
\end{equation}
 Indeed, for  $\psi\in C^\infty_0(\tilde U)$, using \eqref{eq_holonomy_9}, \eqref{eq_holonomy_10}, we get 
\begin{align*}
\langle L_{g,A,q} \tilde u, \psi\rangle&=\int_{\tilde U} \big(  \langle d\tilde u, d\psi\rangle_g + i \langle  d\psi, A\tilde u \rangle_g  -i  \langle A,d\tilde u \rangle_g\psi  +  (\langle A,A \rangle_g +q)\tilde u\psi   \big)dV_g\\
&=\int_{U} \big(  \langle d\tilde u, d\psi\rangle_g + i \langle  d\psi, A\tilde u \rangle_g  -i  \langle A,d\tilde u \rangle_g\psi  +  (\langle A,A \rangle_g +q)\tilde u\psi   \big)dV_g\\
&=  \langle  \langle d_A u,\nu \rangle_g, \psi|_{\p U}  \rangle_{H^{-\frac{1}{2}}(\p U)\times H^{\frac{1}{2}}(\p U)}=0.
\end{align*}
Here we have used that $\supp(\psi|_{\p U})\subset \Gamma$ and the fact that $\langle d_{A}u,\nu \rangle_g|_{\Gamma}=0$. This shows \eqref{eq_holonomy_11}. 

Now as $\tilde u\in H^1(\tilde U)$ satisfies \eqref{eq_holonomy_11} and $\tilde u=0$ in $\tilde U\setminus \overline{U}$, by the unique continuation result of  \cite[Theorem 1.1]{Koch_Tataru} we conclude that $\tilde u$ vanishes identically on $\tilde U$. The proof is complete. 
\end{proof}

\begin{appendix}
\section{Boundary determination of a continuous magnetic potential}

\label{sec_boundary_rec}

When proving Theorem \ref{thm_main_2}, an important step consists in determining the boundary values of the tangential components of the continuous magnetic potentials. The purpose of this section is to carry out this step by adapting the method of  \cite{Brown_Salo_2006} developed in the case of magnetic Schr\"odinger operators on $\R^n$.  Compared with the latter work, here we treat the case of magnetic Schr\"odinger operators on a smooth compact Riemannian manifold with boundary and we do not assume the well-posedness of the  Dirichlet problem for the magnetic Schr\"odinger operator. 

To circumvent the difficulty related to the fact that zero may be a Dirichlet eigenvalue,  we shall use the solvability result for the magnetic Schr\"odinger operator given in Proposition \ref{prop_solvability},  which is based on the Carleman estimate with a gain of two derivatives established in Proposition \ref{prop_Carleman_magnetic_mnfld}.  We have learned of the idea of using a Carleman estimate to handle the case when zero is a Dirichlet eigenvalue from the work \cite{Salo_Tzou_2009} on the Dirac operator. 

\begin{prop} 
\label{prop_boundary_rec}
Let $(M,g)$ be a smooth compact Riemannian manifold of dimension $n\ge 3$ with smooth boundary $\p M$, and  let $A^{(1)}, A^{(2)}\in C(M, T^*M)$ and $q^{(1)}, q^{(2)}\in L^\infty(M,\C)$.  Assume that $C_{g, A^{(1)},q^{(1)}}=C_{g, A^{(2)},q^{(2)}}$. 
Then 
$\langle \tau,  A^{(1)}(x_0)- A^{(2)}(x_0)\rangle =0$,
for all points $x_0\in \p M$ and all unit tangent vectors $\tau\in T_{x_0}(\p M)$. Here $\langle \cdot, \cdot \rangle$ is the duality between the tangent and cotangent bundles of $M$. 
\end{prop}

\begin{proof}
We shall follow  \cite{Brown_Salo_2006} closely.  The idea  is to construct some special solutions to the magnetic Schr\"odinger equations, whose boundary values have an oscillatory behavior while becoming increasingly concentrated near a given point on the boundary of 
$M$.  Substituting these solutions into the integral identity of  Proposition \ref{prop_integral_identify} will allow us to recover the tangential components of the magnetic potentials along the boundary.

Let $x_0\in \p M$ and let $(x_1,\dots, x_n)$ be the boundary normal coordinates centered at $x_0$ so that in these coordinates, $x_0 =0$, the boundary $\p M$ is given by $\{x_n=0\}$, and $M^0$ is given by $\{x_n > 0\}$. It is then well known \cite{Lee_Uhlmann} that  
\[
g(x',x_n)=\sum_{\alpha,\beta=1}^{n-1}g_{\alpha\beta}(x)dx_\alpha dx_\beta+(dx_n)^2,
\]
and therefore, 
\begin{equation}
\label{eq_Laplace_boundary-nc}
-\Delta_g=D_{x_n}^2+\sum_{\alpha,\beta=1}^{n-1}g^{\alpha\beta}(x)D_{x_\alpha}D_{x_\beta}+f(x)D_{x_n}+R(x,D_{x'}),
\end{equation}
where 
$f$ is a smooth function and $R$ is a differential operator of order $1$ in $x'$ with smooth coefficients.  We shall assume, as we may, that 
\begin{equation}
\label{eq_Laplace_boundary-nc_2}
g^{\alpha \beta}(0)=\delta^{\alpha \beta}, \quad 1\le \alpha,\beta\le n-1,
\end{equation}
and therefore $T_0\p M=\R^{n-1}$, equipped with the Euclidean metric.  The unit tangent vector $\tau$ is then given by $\tau=(\tau',0)$ where $\tau'\in \R^{n-1}$, $|\tau'|=1$.   Associated to the tangent vector $\tau'$ is the covector $\xi'_\alpha=\sum_{\beta=1}^{n-1} g_{\alpha \beta}(0) \tau'_\beta=\tau'_\alpha\in T^*_{x_0}\p M$.

Let $\eta\in C^\infty_0(\R^n,\R)$ be a function such that $\supp(\eta)$ is in a small neighborhood of $0$, and 
\begin{equation}
\label{eq_int_eta_1}
\int_{\R^{n-1}}\eta(x',0)^2dx'=1.
\end{equation}
Following \cite{Brown_Salo_2006}, in the boundary normal coordinates,  we set 
\begin{equation}
\label{eq_9_1}
v_0(x)=\eta\bigg(\frac{x}{\lambda^{1/2}}\bigg)e^{\frac{i}{\lambda}(\tau'\cdot x'+ ix_n)}, \quad 0<\lambda\ll 1,
\end{equation}
so that  $v_0\in C^\infty(M)$ with $\supp(v_0)$ in  $\mathcal{O}(\lambda^{1/2})$ neighborhood of $x_0=0$. Here $\tau'$ is viewed as a covector. 

Let $v_1\in H^1_0(M^0)$ be the solution to the following Dirichlet problem for the Laplacian,
\begin{equation}
\label{eq_9_Dirichlet_lapl}
\begin{aligned}
-\Delta_g v_1=&\Delta_g v_0, \quad\textrm{in}\quad M,\\
v_1|_{\p M}=&0.   
\end{aligned}
\end{equation} 
Let $\delta(x)$ be the distance from $x\in M$ to the boundary of $M$.  Similarly  to \cite{Brown_Salo_2006}, we shall need the following estimates, 
\begin{equation}
\label{eq_9_3} 
\|v_0\|_{L^2(M)}\le \mathcal{O}(\lambda^{\frac{n-1}{4}+\frac{1}{2}}),
\end{equation}
\begin{equation}
\label{eq_9_4}
\|v_1\|_{L^2(M)}\le \mathcal{O}(\lambda^{\frac{n-1}{4}+\frac{1}{2}}),
\end{equation}
\begin{equation}
\label{eq_9_15}
\|d v_1\|_{L^2(M)}\le \mathcal{O}(\lambda^{\frac{n-1}{4}}),
\end{equation}
\begin{equation}
\label{eq_Brown_Salo_2_18}
\|\delta dv_0\|_{L^2(M)}\le \mathcal{O}(\lambda^{\frac{n-1}{4}+\frac{1}{2}}),
\end{equation}
\begin{equation}
\label{eq_Brown_Salo_2_18_new_mnfld}
\| dv_0\|_{L^2(M)}\le \mathcal{O}(\lambda^{\frac{n-1}{4}-\frac{1}{2}}),
\end{equation}
\begin{equation}
\label{eq_Brown_Salo_2_20}
\|\delta d(v_0+v_1)\|_{L^2(M)}\le \mathcal{O}(\lambda^{\frac{n-1}{4}+\frac{1}{2}}),
\end{equation}
which we shall now proceed to prove. First, the following direct computation,
\begin{align*}
\|v_0\|^2_{L^2(M)}&=\mathcal{O}(1)\int_{|x|\le c\lambda^{\frac{1}{2}}, x_n\ge 0} e^{-\frac{2x_n}{\lambda}}dx'dx_n\le \mathcal{O}(\lambda^{\frac{n-1}{2}})\int_{0}^{\infty} e^{-2t} \lambda dt\\
&\le \mathcal{O}(\lambda^{\frac{n-1}{2}+1}),
\end{align*}
gives  the bound \eqref{eq_9_3}. The bound \eqref{eq_Brown_Salo_2_18} is implied by the following estimates,
\begin{align*}
\|\delta dv_0\|^2_{L^2(M)}\le \mathcal{O}(1) \|x_n d v_0\|^2_{L^2(M)}&\le \mathcal{O}(1) \int_{|x|\le c\lambda^{\frac{1}{2}}, x_n\ge 0} x_n^2 \lambda^{-2}e^{-\frac{2x_n}{\lambda}}dx'dx_n\\
&\le \mathcal{O}(\lambda^{\frac{n-1}{2}+1})
\int_{0}^{\infty} t^2e^{-2t} dt\le \mathcal{O}(\lambda^{\frac{n-1}{2}+1}).
\end{align*}
The bound \eqref{eq_Brown_Salo_2_18_new_mnfld} can be shown similarly. 

To prove the bound \eqref{eq_9_4}, since the boundary of $M$ is smooth, we shall proceed a bit differently than in \cite{Brown_Salo_2006},  relying on the partial hypoellipticity of elliptic equations. Specifically applying \cite[Theorem 26.3]{Eskin_book} to the Dirichlet problem \eqref{eq_9_Dirichlet_lapl}, we conclude that 
\begin{equation}
\label{eq_app_1_-1}
\|v_1+v_0\|_{L^2(M)}\le C\| v_0\|_{H^{-\frac{1}{2}}(\p M)}. 
\end{equation}

To estimate $\| v_0\|_{H^{-\frac{1}{2}}(\p M)}$, we first notice that 
\begin{equation}
\label{eq_app_1}
\mathcal{F}_{x'\to \xi'}(v_0(x',0))(\xi')=\lambda^{\frac{n-1}{2}}   \mathcal{F}_{x'\to \xi'} (\eta(x',0))(\lambda^{\frac{1}{2}}\xi'-\lambda^{-\frac{1}{2}}\tau')=: \lambda^{\frac{n-1}{2}}\tilde \eta(\lambda^{\frac{1}{2}}\xi'-\lambda^{-\frac{1}{2}}\tau').
\end{equation}
Then making the change of variables $z=\lambda^{\frac{1}{2}}\xi'-\lambda^{-\frac{1}{2}}\tau'$, we get 
\begin{equation}
\label{eq_app_2}
\begin{aligned}
\| v_0 & \|_{H^{-\frac{1}{2}}(\p M)}^2\le \mathcal{O}(\lambda^{n-1})\int_{\R^{n-1}}\frac{1}{1+|\xi'|} |  \tilde \eta (\lambda^{\frac{1}{2}}\xi'-\lambda^{-\frac{1}{2}}\tau')  |^2d\xi'\\
&= \mathcal{O}(\lambda^{\frac{n-1}{2}})\int_{\R^{n-1}}\frac{\lambda^{\frac{1}{2}}}{\lambda^\frac{1}{2}+| z+\lambda^{-\frac{1}{2}}\tau'|} |\tilde \eta(z)|^2dz=\mathcal{O}(\lambda^{\frac{n-1}{2}})(J_1+J_2),
\end{aligned}
\end{equation}
where 
\begin{align*}
J_1:= \int_{|z|\le \lambda^{-\frac{1}{2}}/2 } \frac{\lambda^{\frac{1}{2}}}{\lambda^\frac{1}{2}+| z+\lambda^{-\frac{1}{2}}\tau'|} |\tilde \eta(z)|^2 dz,\\
J_2:= \int_{|z|\ge \lambda^{-\frac{1}{2}}/2 }\frac{\lambda^{\frac{1}{2}}}{\lambda^\frac{1}{2}+| z+\lambda^{-\frac{1}{2}}\tau'|} |\tilde \eta(z)|^2 dz.
\end{align*}
Using the fact that $\tilde \eta\in \mathcal{S}(\R^{n-1})$, we 
\begin{equation}
\label{eq_app_3}
\begin{aligned}
&|J_1|\le \frac{\lambda}{\lambda+1/2} \int_{\R^{n-1}}|\tilde \eta(z)|^2dz\le \mathcal{O}(\lambda),\\
&|J_2|\le \int_{|z|\ge \lambda^{-\frac{1}{2}}/2}|\tilde \eta(z)|^2dz=\mathcal{O}(\lambda^N),\quad \forall N>0.
\end{aligned}
\end{equation}
It follows from
\eqref{eq_app_2}
and \eqref{eq_app_3} that 
\[
\| v_0\|_{H^{-\frac{1}{2}}(\p M)}\le \mathcal{O}(\lambda^{\frac{n-1}{4}+\frac{1}{2}}),
\]
and therefore, this estimate together with \eqref{eq_app_1_-1} and \eqref{eq_9_3} gives \eqref{eq_9_4}.

Let us now prove the bound \eqref{eq_9_15}. Applying the Lax--Milgram lemma to \eqref{eq_9_Dirichlet_lapl}, we get 
\begin{equation}
\label{eq_bound_rec_new_3--100_1}
\|v_1\|_{H^1_0(M^0)}\le C\|\Delta_g v_0\|_{H^{-1}(M^0)},
\end{equation}
and therefore, to show \eqref{eq_9_15} we have to estimate $\|\Delta_g v_0\|_{H^{-1}(M^0)}$. In doing so, let $\varphi\in C^\infty_0(M^0)$. Then  we get 
\[
|( f(x)D_{x_n}v_0+R(x,D_{x'})v_0, \varphi)_{L^2(M)}|\le C\|v_0\|_{L^2(M)} \|\varphi\|_{H^1(M^0)},
\]
and therefore,
\begin{equation}
\label{eq_bound_rec_new_3--100}
\| f(x)D_{x_n}v_0+R(x,D_{x'})v_0\|_{H^{-1}(M^0)}\le C\|v_0\|_{L^2(M)}.
\end{equation}
In view of \eqref{eq_Laplace_boundary-nc} and \eqref{eq_Laplace_boundary-nc_2} it remains to consider 
\[
P:=D_{x_n}^2+\sum_{\alpha,\beta=1}^{n-1} g^{\alpha\beta}(x)D_{x_\alpha}D_{x_\beta}=-\Delta+   \sum_{\alpha,\beta=1}^{n-1} \mathcal{O}_{\alpha,\beta}(|x|) D_{x_\alpha}D_{x_\beta},
\]
where $-\Delta$ is the Euclidean Laplacian. 
By Hardy's inequality, see \cite{Davies_2000}, 
\begin{equation}
\label{eq_Hardy}
\int_{M}|f(x)/\delta(x)|^2 dV_g\le C\int_{M}|d f(x)|^2dV_g,
\end{equation}
where $f\in H^1_0(M^0)$, we obtain that 
\[
|(Pv_0,\varphi)_{L^2(M)}|=|(\delta Pv_0,\varphi/\delta)_{L^2(M)}|\le C \|\delta Pv_0\|_{L^2(M)}\|\varphi\|_{H^1_0(M^0)}.
\]
Thus, 
\begin{equation}
\label{eq_bound_rec_new_1}
\|Pv_0\|_{H^{-1}(M^0)}\le  C \|\delta Pv_0\|_{L^2(M)}.
\end{equation}

We have
\begin{equation}
\label{eq_bound_rec_new_2}
\begin{aligned}
\Delta v_0=& e^{\frac{i}{\lambda}(\tau'\cdot x'+ix_n)}\bigg[  \lambda^{-1} (\Delta \eta)\bigg(\frac{x}{\lambda^{\frac{1}{2}}}\bigg)+ 2i \lambda^{-\frac{3}{2}} (\nabla \eta)\bigg(\frac{x}{\lambda^{\frac{1}{2}}}\bigg)\cdot(\tau',i)\\
& -
\lambda^{-2} (\tau',i)\cdot(\tau',i)\eta\bigg(\frac{x}{\lambda^{\frac{1}{2}}}\bigg)
\bigg]\\
= & e^{\frac{i}{\lambda}(\tau'\cdot x'+ix_n)}\bigg[  \lambda^{-1} (\Delta \eta)\bigg(\frac{x}{\lambda^{\frac{1}{2}}}\bigg)+ 2i \lambda^{-\frac{3}{2}} (\nabla \eta)\bigg(\frac{x}{\lambda^{\frac{1}{2}}}\bigg)\cdot(\tau',i)\bigg],
\end{aligned}
\end{equation}
where we have used that $(\tau',i)\cdot(\tau',i)=0$. We also obtain that 
\begin{equation}
\label{eq_bound_rec_new_3}
\begin{aligned}
 D_{x_\alpha}D_{x_\beta}v_0= & e^{\frac{i}{\lambda}(\tau'\cdot x'+ix_n)}\bigg[  \lambda^{-1} (D_{x_\alpha}D_{x_\beta} \eta)\bigg(\frac{x}{\lambda^{\frac{1}{2}}}\bigg)+  \lambda^{-\frac{3}{2}} (D_{x_\alpha} \eta)\bigg(\frac{x}{\lambda^{\frac{1}{2}}}\bigg)\tau_\beta\\
& +  \lambda^{-\frac{3}{2}} (D_{x_\beta} \eta)\bigg(\frac{x}{\lambda^{\frac{1}{2}}}\bigg)\tau_\alpha + \lambda^{-2} \tau_\alpha \tau_\beta \eta\bigg(\frac{x}{\lambda^{\frac{1}{2}}}\bigg)
\bigg]. 
\end{aligned}
\end{equation}
Using \eqref{eq_bound_rec_new_1}, \eqref{eq_bound_rec_new_2} and \eqref{eq_bound_rec_new_3}, we get 
\begin{equation}
\label{eq_bound_rec_new_4}
\begin{aligned}
\|Pv_0\|^2_{H^{-1}(M^0)}\le \mathcal{O}(1) \int_{|x|\le c\lambda^{\frac{1}{2}}, x_n\ge 0} x_n^2 e^{-\frac{2x_n}{\lambda}} (\lambda^{-3}+|x|^2\lambda^{-4} )dx\\
\le \mathcal{O}(\lambda^{\frac{n-1}{2}}\lambda^{-3})\int_0^\infty x_n^2e^{-\frac{2x_n}{\lambda}} dx_n\le \mathcal{O}
(\lambda^{\frac{n-1}{2}}).
\end{aligned}
\end{equation}
Hence, by \eqref{eq_bound_rec_new_3--100}, \eqref{eq_9_3} and \eqref{eq_bound_rec_new_4}, we obtain that 
\[
\|\Delta_g v_0\|_{H^{-1}(M^0)}\le \mathcal{O}(\lambda^{\frac{n-1}{4}}).
\]
This estimate together with \eqref{eq_bound_rec_new_3--100_1} proves  \eqref{eq_9_15}. 

Finally, the estimate \eqref{eq_Brown_Salo_2_20} follows by an application of Caccioppoli's inequality
\[
\|\delta d(v_0+v_1)\|_{L^2(M)}\le C\|v_0+v_1\|_{L^2(M)},
\]
see \cite[Theorem 4.4]{Giaquinta_book}, combined with \eqref{eq_9_3} and \eqref{eq_9_4}.

Next we would like to show the  existence of a solution $u_1\in H^1(M^0)$ to the magnetic Schr\"odinger equation
\begin{equation}
\label{eq_9_10}
L_{g, A^{(1)},q^{(1)}}u_1=0\quad \textrm{in}\quad \mathcal{D}'(M^0),
\end{equation}
of the form 
\begin{equation}
\label{eq_9_11}
u_1=v_0+v_1+r_1,
\end{equation}
with 
\begin{equation}
\label{eq_9_12}
\|r_1\|_{H^1(M^0)}\le  \mathcal{O}(\lambda^{\frac{n-1}{4}+\frac{1}{2}}).
\end{equation}
To that end, plugging \eqref{eq_9_11} into \eqref{eq_9_10}, we obtain the following equation for $r_1$,
\[
L_{g, A^{(1)},q^{(1)}} r_1=- id^*(A^{(1)}(v_0+v_1)) +i \langle A^{(1)}, dv_0+dv_1\rangle_g - (\langle A^{(1)},A^{(1)} \rangle_g+q^{(1)})(v_0+v_1) 
\]
in $\mathcal{D}'(M^0)$. Applying Proposition \ref{prop_solvability} with $h>0$ small but fixed, we conclude the existence of $r_1\in H^1(M^0)$ such that
\begin{equation}
\label{eq_right_hand_r_1}
\begin{aligned}
\|r_1\|_{H^1(M^0)}
\le C\|   - id^*(A^{(1)}(v_0+v_1))& +i \langle A^{(1)}, dv_0+dv_1\rangle_g\\
& - (\langle A^{(1)},A^{(1)} \rangle_g+q^{(1)})(v_0+v_1) \|_{H^{-1}(M^0)}.
\end{aligned}
\end{equation}
Let us now compute the norm in the right hand side of \eqref{eq_right_hand_r_1}.   
Let $\psi\in C^\infty_0(M^0)$. Then using \eqref{eq_Brown_Salo_2_20} and \eqref{eq_Hardy}, we get
\begin{equation}
\label{eq_r_1_est_1}
\begin{aligned}
|\langle \langle A^{(1)}, dv_0+dv_1\rangle_g,\psi\rangle_{M^0}|
\le \mathcal{O}(1)\|A^{(1)}\|_{L^\infty(M)}\|\delta d (v_0+v_1) \|_{L^2(M)}\|\psi/\delta\|_{L^2(M)}\\
\le \mathcal{O}(\lambda^{\frac{n-1}{4}+\frac{1}{2}})\|d \psi\|_{L^2(M)}\le  \mathcal{O}(\lambda^{{\frac{n-1}{4}+\frac{1}{2}}})\| \psi\|_{H^1(M^0)}.
\end{aligned}
\end{equation}

Using \eqref{eq_9_3} and \eqref{eq_9_4}, we obtain that 
\begin{equation}
\label{eq_r_1_est_2}
\begin{aligned}
|\langle &   d^*(A^{(1)}(v_0+v_1))  ,\psi\rangle_{M^0}|=\bigg| \int_{M^0} \langle d\psi,  A^{(1)}(v_0+v_1) \rangle_g dV\bigg|\\
&\le \mathcal{O}(1)\|A^{(1)}\|_{L^\infty(M)}\|v_0+v_1\|_{L^2(M)}\|d \psi\|_{L^2(M)}\le \mathcal{O}(\lambda^{\frac{n-1}{4}+\frac{1}{2}})\| \psi\|_{H^1(M^0)},
\end{aligned}
\end{equation}
and 
\begin{equation}
\label{eq_r_1_est_3}
\begin{aligned}
|\langle  (\langle A^{(1)},A^{(1)} \rangle_g+q^{(1)})(v_0+v_1)  ,\psi\rangle_{M^0}|\le \mathcal{O}(\lambda^{\frac{n-1}{4}+\frac{1}{2}})\| \psi\|_{H^1(M^0)}.
\end{aligned}
\end{equation}
The estimate  \eqref{eq_9_12} follows from \eqref{eq_right_hand_r_1}, \eqref{eq_r_1_est_1},  \eqref{eq_r_1_est_2}, and \eqref{eq_r_1_est_3}.

Similarly, there exists a solution $u_2\in H^1(M^0)$ of $L_{g, \overline{A^{(2)}},\overline{q^{(2)}}}u_2=0$ in $\mathcal{D}'(M^0)$ of the form 
\begin{equation}
\label{eq_9_16}
u_2=v_0+v_1+r_2,
\end{equation}
where $r_2\in H^1(M^0)$ satisfies \eqref{eq_9_12}.  

The next step is to substitute the solutions $u_1$ and $u_2$, given by \eqref{eq_9_11} and \eqref{eq_9_16} into the identity \eqref{eq_proof_1} of Proposition \ref{prop_integral_identify}, multiply it by $\lambda^{-\frac{(n-1)}{2}}$ and compute the limit as $\lambda\to 0$. We write
 \begin{equation}
\label{eq_def_I}
I:= \lambda^{-\frac{(n-1)}{2}}  \int_M i  \langle A^{(1)}-A^{(2)}, u_1d\overline{u_2}-\overline{u_2}du_1 \rangle_g dV=I_1+I_2+I_3+I_4+I_5,
\end{equation}
where 
\begin{align*}
I_1&= \lambda^{-\frac{(n-1)}{2}}  \int_M i  \langle A^{(1)}-A^{(2)},  v_0 d\overline{v_0}-\overline{v_0}d v_0  \rangle_g dV,\\
I_2&= \lambda^{-\frac{(n-1)}{2}} \int_M i  \langle A^{(1)}-A^{(2)},   u_1d \overline{v_1}-\overline{u_2}d v_1\rangle_g dV,\\
I_3&= \lambda^{-\frac{(n-1)}{2}} \int_M i  \langle A^{(1)}-A^{(2)},  u_1 d \overline{r_2}-\overline{u_2}d r_1 \rangle_g dV,\\
I_4&=   \lambda^{-\frac{(n-1)}{2}}  \int_M i  \langle A^{(1)}-A^{(2)},    v_1 d \overline{v_0}-\overline{v_1} d v_0  \rangle_g dV,\\
I_5&= \lambda^{-\frac{(n-1)}{2}}  \int_M i  \langle A^{(1)}-A^{(2)},  r_1d \overline{v_0}-\overline{r_2}d  v_0  \rangle_g dV.
\end{align*}

Let us compute $\lim_{\lambda\to 0}I_1$. To that end we have
\begin{equation}
\label{eq_v_0_grad}
d v_0= e^{\frac{i}{\lambda}(\tau'\cdot x'+i x_n)}  \bigg[ d \bigg( \eta\bigg(\frac{x}{\lambda^{\frac{1}{2}}}\bigg)\bigg ) +    \eta\bigg(\frac{x}{\lambda^{\frac{1}{2}}}\bigg) \frac{i}{\lambda} (\tau' dx' +idx_n)\bigg],
\end{equation}
and 
\[
v_0 d \overline{v_0}-\overline{v_0}d  v_0=-2i \frac{1}{\lambda}\eta^2 \bigg(\frac{x}{\lambda^{\frac{1}{2}}} \bigg) e^{-\frac{2x_n}{\lambda} } \tau' \cdot dx'.
\]
Making the change of variables $y'=\frac{x'}{\lambda^{\frac{1}{2}}}$, $y_n=\frac{x_n}{\lambda}$,  using that $A^{(1)},A^{(2)}\in C(M, T^*M)$ and \eqref{eq_int_eta_1}, we get 
 \begin{equation}
\label{eq_7_18}
\begin{aligned}
\lim_{\lambda\to 0} I_1=2 &\sum_{\alpha,\beta=1}^{n-1} \lim_{\lambda\to 0}\int_{\R^{n-1}}\int_0^{+\infty} | g(\lambda^{\frac{1}{2}} y', \lambda y_n) |^{\frac{1}{2}} g^{\alpha \beta}(\lambda^{\frac{1}{2}} y', \lambda y_n)\\
 &(A^{(1)}-A^{(2)})_\alpha(\lambda^{\frac{1}{2}} y', \lambda y_n)\tau_\beta
\eta^2(y',\lambda^{\frac{1}{2}}y_n)e^{-2y_n} dy'dy_n\\
&=|g(0)|^{\frac{1}{2}} \sum_{\alpha,\beta=1}^{n-1} g^{\alpha \beta}(0) (A^{(1)}_\alpha (0)-A_\alpha^{(2)}(0))\tau_\beta= \langle A^{(1)}(0)-A^{(2)}(0), \tau \rangle.
\end{aligned}
 \end{equation}

Now it follows from  \eqref{eq_9_3}, \eqref{eq_9_4} and \eqref{eq_9_12}  that
\begin{equation}
\label{eq_9_17}
\|u_j\|_{L^2(M)}\le \mathcal{O}(\lambda^{\frac{n-1}{4}+\frac{1}{2}}),\quad j=1,2.
\end{equation}
The estimates \eqref{eq_9_15} and \eqref{eq_9_17} give that 
\begin{equation}
\label{eq_I_2}
|I_2|\le \mathcal{O}(\lambda^{-\frac{(n-1)}{2}})\|A^{(1)}-A^{(2)}\|_{L^\infty(M)}(\|u_1\|_{L^2(M)}+ \|u_2\|_{L^2(M)})\|d v_1\|_{L^2(M)}
\le \mathcal{O}(\lambda^{\frac{1}{2}}).
\end{equation}

By \eqref{eq_9_12} and \eqref{eq_9_17}, we have
\begin{equation}
\label{eq_I_3}
\begin{aligned}
|I_3|\le \mathcal{O}(\lambda^{-\frac{(n-1)}{2}} )\|A^{(1)}-A^{(2)}\|_{L^\infty(M)}(&\|u_1\|_{L^2(M)}\|d r_2\|_{L^2(M)}
\\
&+ \|u_2\|_{L^2(M)}\|d r_1\|_{L^2(M)})\le \mathcal{O}(\lambda).
\end{aligned}
\end{equation}

Using the fact that $v_1\in H^1_0(M)$, Hardy's inequality \eqref{eq_Hardy}, and the estimates  \eqref{eq_Brown_Salo_2_18},  and \eqref{eq_9_15}, we get 
\begin{equation}
\label{eq_I_4}
\begin{aligned}
|I_4|&\le   \mathcal{O}(\lambda^{-\frac{(n-1)}{2}})  \|A^{(1)}-A^{(2)}\|_{L^\infty(M)}\|v_1/\delta\|_{L^2(M)}\|\delta d v_0\|_{L^2(M)}\\
&\le \mathcal{O}(\lambda^{-\frac{(n-1)}{2}}) \|d v_1\|_{L^2(M)}\|\delta d v_0\|_{L^2(M)}\le \mathcal{O}(\lambda^{\frac{1}{2}}).
\end{aligned}
\end{equation}

Let us now estimate $|I_5|$.  First, a direct computation shows that 
\begin{equation}
\label{eq_v_0_boundary}
\| v_0\|_{L^2(\p M)}\le \mathcal{O}(\lambda^\frac{n-1}{4}).
\end{equation}
Assume that $(M,g)$ is embedded in a compact smooth manifold $(N,g)$ without boundary of the same dimension. 
Let us extend $A:=A^{(1)}-A^{(2)}\in C(M, T^*M)$ to a continuous $1$-form on $N$, and we shall write $A\in C(N,T^*N)$.  
Using a partition of unity argument together with a regularization  in each coordinate patch, we get that there exists a family  $A_\tau\in C^\infty(N,T^*N)$ such that 
\begin{equation}
\label{eq_flat_est_new}
\|A-A_\tau\|_{L^\infty}=o(1),\quad \tau\to 0,
\end{equation}
and 
\begin{equation}
\label{eq_flat_est_2_new}
\begin{aligned}
\|A_\tau\|_{L^\infty}=\mathcal{O}(1), \quad
\|\nabla A_\tau\|_{L^\infty}=\mathcal{O}(\tau^{-1}),\quad \tau\to 0.  
\end{aligned}
\end{equation}
Let us consider 
 \begin{align*}
J:=  \lambda^{-\frac{(n-1)}{2}}    \int_{M}  i \langle A,  r_1 d  \overline{v_0}\rangle_g dV=J_1+J_2,
 \end{align*}
where
\begin{align*}
J_1=  \lambda^{-\frac{(n-1)}{2}}     \int_{M}  i \langle (A-A_\tau),  r_1 d  \overline{v_0}\rangle_g dV,\quad 
J_2=  \lambda^{-\frac{(n-1)}{2}}    \int_{M}  i \langle A_\tau,  r_1 d  \overline{v_0}\rangle_g dV.
\end{align*}
Using \eqref{eq_9_12},  \eqref{eq_Brown_Salo_2_18_new_mnfld}, and \eqref{eq_flat_est_new}, we get 
\[
|J_1|\le \mathcal{O}( \lambda^{-\frac{(n-1)}{2}}) \|A-A_\tau\|_{L^\infty(M)}\|r_1\|_{L^2(M)}\|d v_0\|_{L^2(M)}=o(1),
\]
as $ \tau\to 0$. To estimate $J_2$, it is no longer sufficient to use the bound \eqref{eq_Brown_Salo_2_18_new_mnfld}, and therefore, we shall integrate by parts. We get
\[
J_2=J_{2,1}+J_{2,2}+J_{2,3},
\]
where  
\begin{align*}
 &J_{2,1}= -\lambda^{-\frac{(n-1)}{2}}\int_{M} i  \div( A^\sharp_\tau) r_1 \overline{v_0} dV,\quad J_{2,2} =- \lambda^{-\frac{(n-1)}{2}} \int_{M}i  \langle A_\tau,  d r_1\rangle_g \overline{v_0} dV,\\
 &J_{2,3}= \lambda^{-\frac{(n-1)}{2}} \int_{\p M} i \langle A_\tau, \nu\rangle_g r_1\overline{v_0} dS.
\end{align*}
Here $A^\sharp_\tau=g^{jk}(A_\tau)_j\p_{x_k}$.
By \eqref{eq_9_3}, \eqref{eq_9_12}, and \eqref{eq_flat_est_2_new}, we have
\[
|J_{2,1}|\le \mathcal{O} (\lambda^{-\frac{(n-1)}{2}}) \|\nabla_g  A_\tau \|_{L^\infty(M)}\|r_1\|_{L^2(M)}\|v_0\|_{L^2(M)}\le \mathcal{O}(\tau^{-1}\lambda),
\] 
and 
\[
|J_{2,2}|\le \mathcal{O}(\lambda^{-\frac{(n-1)}{2}})  \|A_\tau \|_{L^\infty(M)} \|d r_1\|_{L^2(M)}\|v_0\|_{L^2(M)} \le \mathcal{O}(\lambda). 
\]
Using the trace theorem, the estimates \eqref{eq_9_12} and \eqref{eq_v_0_boundary}, we get
\[
|J_{2,3}|\le  \mathcal{O}(\lambda^{-\frac{(n-1)}{2}} ) \| \langle A_\tau, \nu\rangle_g\|_{L^\infty(\p M)}\|r_1\|_{H^1(M)}\|v_0\|_{L^2(\p M)}\le \mathcal{O}(\lambda^{1/2}). 
\]
Choosing $\tau=\lambda^{1/2}$, we conclude that $|J|=o(1)$ as $\lambda\to 0$, and therefore, 
\begin{equation}
\label{eq_I_5}
|I_5|=o(1),\quad \lambda\to 0. 
\end{equation}

Hence, for $I$, defined by \eqref{eq_def_I}, using \eqref{eq_7_18}, \eqref{eq_I_2}, \eqref{eq_I_3}, \eqref{eq_I_4}, and  \eqref{eq_I_5}, we get 
\begin{equation}
\label{eq_4_1_first}
\lim_{\lambda\to 0} I=\langle A^{(1)}(0)-A^{(2)}(0), \tau \rangle. 
\end{equation}

Furthermore, for $u_1$ and $u_2$, given by \eqref{eq_9_11} and \eqref{eq_9_16}, respectively, using \eqref{eq_9_17},  we obtain that 
\begin{equation}
\label{eq_4_1_second}
\lambda^{-\frac{(n-1)}{2}}\bigg|\int_M (\langle A^{(1)}, A^{(1)}\rangle_g- \langle A^{(2)}, A^{(2)}\rangle_g +q^{(1)}-q^{(2)})u_1\overline{u_2}dV \bigg|\le \mathcal{O}(\lambda).
\end{equation}

Thus, we conclude from \eqref{eq_proof_1} with the help of \eqref{eq_4_1_first} and \eqref{eq_4_1_second} that 
$\langle A^{(1)}(0)-A^{(2)}(0), \tau \rangle=0$. The proof of Proposition \ref{prop_boundary_rec} is complete. 
\end{proof}

\end{appendix}

\section*{Acknowledgements}
The research of K.K. is partially supported by the National Science Foundation (DMS 1500703). The research of G.U. is partially supported by the National Science Foundation, Simons Fellowship, and the Academy of Finland.

\end{document}